\documentclass[11pt]{amsart}

\usepackage{
  amsmath,
  amsfonts,
  amssymb,
  amsthm,
  amscd,
  comment,
  paralist,
  etoolbox,
  mathtools,
  enumitem,
  mathdots,
  booktabs,
  longtable
}
\usepackage[usenames,dvipsnames]{xcolor}
\usepackage{mathrsfs} % to get \mathscr

\usepackage[scaled]{helvet}

\usepackage[T1]{fontenc}

\usepackage[colorlinks=true, pdfstartview=FitV, linkcolor=blue, citecolor=blue, urlcolor=blue, breaklinks=true]{hyperref}

\newcommand\C{\mathbb{C}}
\newcommand\Z{\mathbb{Z}}

\newcommand\N{\mathbb{N}}
\newcommand\kk{\Bbbk}

\newcommand\cA{\mathcal{A}}  % affine wreath product algebra
\newcommand\cC{\mathcal{C}}
\newcommand\cB{\mathcal{B}}
\newcommand\cH{\mathcal{H}}
\newcommand\cR{\mathcal{R}}
\newcommand\cS{\mathcal{S}}

\newcommand\bb{\mathbf{b}}
\newcommand\bc{\mathbf{c}}
\newcommand\bC{\mathbf{C}}
\newcommand\bsf{\mathbf{f}}
\newcommand\bsg{\mathbf{g}}
\newcommand\bF{\mathbf{F}}

\newcommand\gr{\mathrm{gr}}

\newcommand\tr{\mathrm{tr}}

\newcommand\sa{\mathsf{a}}
\newcommand\ssb{\mathsf{b}}
\newcommand\sE{\mathsf{E}}
\newcommand\sL{\mathsf{L}}

\newcommand\tM{\mathtt{M}}  % Type M module
\newcommand\tQ{\mathtt{Q}}  % Type Q module

\newcommand{\md}{\textup{-mod}}

\newcommand{\smod}{\textup{-mod}^\textup{s}}
\newcommand\Cl{\textup{Cl}}
\newcommand\flip{\textup{flip}}
\newcommand\op{\textup{op}}

\newcommand\even{\textup{even}}
\newcommand\odd{\textup{odd}}
\newcommand\irtimes{\circledast} % Irreducible summand of outer tensor product

\DeclareMathOperator{\Ann}{Ann}
\DeclareMathOperator{\Aut}{Aut}

\DeclareMathOperator{\End}{End}
\DeclareMathOperator{\END}{END}
\DeclareMathOperator{\grdim}{grdim} % Graded dimension
\DeclareMathOperator{\Hom}{Hom}
\DeclareMathOperator{\HOM}{HOM}
\DeclareMathOperator{\id}{id}
\DeclareMathOperator{\Ind}{Ind}
\DeclareMathOperator{\Res}{Res}
\DeclareMathOperator{\Span}{Span}

\newtheorem{theo}{Theorem}[section]
\newtheorem{prop}[theo]{Proposition}
\newtheorem{lem}[theo]{Lemma}
\newtheorem*{lem*}{Lemma}
\newtheorem{cor}[theo]{Corollary}

\theoremstyle{definition}
\newtheorem{defin}[theo]{Definition}
\newtheorem{rem}[theo]{Remark}
\newtheorem{eg}[theo]{Example}

\numberwithin{equation}{section}
\allowdisplaybreaks

\setenumerate[1]{label=(\alph*)}  % To change enumerate, reference with \ref{}
\setenumerate[2]{label=(\roman*)}

\newtoggle{comments}
\newtoggle{details}
\newtoggle{prelimnote}
\newtoggle{detailsnote}

\toggletrue{detailsnote}   % For including note about details toggle (default is false)

\iftoggle{comments}{%
  \newcommand{\comments}[1]{
    \ \\
    {\color{red}
      \textbf{AS:} #1
    }
    \\
    }
}{%
  \newcommand{\comments}[1]{}
}

\iftoggle{details}{%
  \newcommand{\details}[1]{
      \ \\
      {\color{OliveGreen}
        \textbf{Details:} #1
      }
      \\
  }
}{%
  \newcommand{\details}[1]{}
}

\iftoggle{prelimnote}{%
  \newcommand{\prelim}{\textsc{Preliminary version} \bigskip}
}{%
  \newcommand{\prelim}{}
}

\begin{document}
%
%%%%%%%%%%%%%%%%%%%%%%%%%%%%%%%%%%%

\title{Affine wreath product algebras}

\author{Alistair Savage}
\address{Department of Mathematics and Statistics, University of Ottawa}
\urladdr{\href{http://alistairsavage.ca}{alistairsavage.ca}, \textrm{\textit{ORCiD}:} \href{https://orcid.org/0000-0002-2859-0239}{orcid.org/0000-0002-2859-0239}}
\email{alistair.savage@uottawa.ca}
%\thanks{The author was supported by a Discovery Grant from the Natural Sciences and Engineering Research Council of Canada.}

\begin{abstract}
  We study the structure and representation theory of affine wreath product algebras and their cyclotomic quotients.  These algebras, which appear naturally in Heisenberg categorification, simultaneously unify and generalize many important algebras appearing in the literature.  In particular, special cases include degenerate affine Hecke algebras, affine Sergeev algebras (degenerate affine Hecke--Clifford algebras), and wreath Hecke algebras.  In some cases, specializing the results of the current paper recovers known results, but with unified and simplified proofs.  In other cases, we obtain new results, including proofs of two open conjectures of Kleshchev and Muth.
\end{abstract}

\subjclass[2010]{Primary: 20C08; Secondary: 16S35, 16W50, 16W55}
\keywords{Wreath product algebra, Frobenius algebra, Hecke algebra, Sergeev algebra, Hecke--Clifford algebra, Mackey Theorem, cyclotomic quotient}
%\date{\today}

\prelim

\maketitle
\thispagestyle{empty}

\tableofcontents

%%%%%%%%%%%%%%%%%%%%%%%%%%%%%%%%%%%%%%%%%%%%%%%%%%%%%%%%%%%%%%%%%%%%
%
\section{Introduction}
%
%%%%%%%%%%%%%%%%%%%%%%%%%%%%%%%%%%%%%%%%%%%%%%%%%%%%%%%%%%%%%%%%%%%%

Affine Hecke algebras and their degenerate analogs play a vital role in the representation theory of Lie algebras.  Over the past couple of decades, several modified versions of degenerate affine Hecke algebras have appeared.  In particular, a super analog arose in the work of Nazarov \cite{Naz97} studying the projective representation theory of the symmetric group, and Wan and Wang \cite{WW08} introduced the so-called wreath Hecke algebras associated to an arbitrary finite group $G$.

Recently, degenerate affine Hecke algebras and their analogs have appeared in Heisenberg categorification.  The degenerate affine Hecke algebra itself appeared in the endomorphism algebras of certain objects in Khovanov's conjectural categorification of the Heisenberg algebra \cite{Kho14}.  Then, a certain $\N$-graded superalgebra appeared in the Heisenberg categorification of Cautis and Licata \cite[\S10.3]{CL12} satisfying relations similar to those of the degenerate affine Hecke algebra.  More generally, Rosso and the author developed a Heisenberg categorification depending on an arbitrary $\N$-graded Frobenius superalgebra $F$.  Specializing $F$ recovers the categorifications of Khovanov and Cautis--Licata.  Again, in the endomorphism algebras of these categories, one sees algebras depending on $F$, and satisfying relations similar to those of the degenerate affine Hecke algebra.  In the special case where $F$ is a symmetric algebra and purely even, these algebras have recently been studied by Kleshchev and Muth \cite{KM15}. Related algebras were also considered in \cite{CG03}.

In the current paper, we investigate the aforementioned algebras appearing in the categorification of \cite{RS17}.  In particular, to every nonnegative integer $n$ and every $\N$-graded Frobenius superalgebra $F$ over a commutative ring $\kk$, we associate an \emph{affine wreath product algebra} $\cA_n(F)$, which is an affine version of the wreath product algebra $F^{\otimes n} \rtimes S_n$.  As an $\N \times \Z_2$-graded $\kk$-module we have $\cA_n(F) \cong \kk[x_1,\dotsc,x_n] \otimes_\kk (F^{\otimes n} \rtimes S_n)$.  The factors $\kk[x_1,\dotsc,x_n]$ and $F^{\otimes n} \rtimes S_n$ are subalgebras, and the relations between these two factors are given in Definition~\ref{def:affine-wreath}.  In particular, the relations between the $x_i$ and the elements of $F^{\otimes n}$ involve the Nakayama automorphism of $F$, which is trivial if and only $F$ is a symmetric algebra.  Remarkably, affine wreath product algebras turn out to provide an appropriate setting for the simultaneous unification and generalization of all of the analogs of degenerate affine Hecke algebras mentioned above.  In particular, we have the following:
\begin{enumerate}
  \item \label{intro-item:degAHA} When $F = \kk$, $\cA_n(F)$ is the degenerate affine Hecke algebra.
  \item \label{intro-item:affSergeev} When $F$ is the two-dimensional Clifford superalgebra $\Cl$, $\cA_n(F)$ is the affine Sergeev algebra (also called the degenerate affine Hecke--Clifford algebra) of Nazarov \cite{Naz97}.  Here the Nakayama automorphism of $F$ is nontrivial.
  \item \label{intro-item:wreathHecke} When $F$ is the group algebra of a group, $\cA_n(F)$ is the wreath Hecke algebra of Wan and Wang \cite{WW08}.  For $n=2$, this algebra was also defined in \cite[\S3]{Pus97}.
  \item \label{intro-item:symmetric} When $F$ is a symmetric algebra (i.e.\ the Nakayama automorphism is trivial) concentrated in even parity, $\cA_n(F)$ is the affinized symmetric algebra considered by Kleshchev and Muth \cite[\S3]{KM15}.  If $F$ is commutative, closely related algebras were defined in \cite[\S4.2]{CG03}.  (See \cite[Rem.~3.6]{KM15} for the precise relation.)
\end{enumerate}

We show in the current paper that, despite their great level of generality, the setting of wreath product algebras turns out to be very tractable.  In particular, it is possible to deduce a great deal of the structure and representation theory of these algebras.  Specializing the Frobenius algebra $F$ then recovers known results in some cases while, in other cases, yields new results and proofs of open conjectures.

We now give an overview of the structure of the paper and highlight the main results.  After discussing some necessary background on superalgebras, Frobenius algebras, and smash products in Section~\ref{sec:prelim}, we give the definition of affine wreath product algebras in Section~\ref{sec:affine-wreath}.  We also explain there how these algebras recover the special cases mentioned above, and deduce some automorphisms and additional relations that will be used in proofs throughout the paper.

In Section~\ref{sec:structure-theory}, we study the structure theory of the affine wreath product algebras $\cA_n(F)$.  We begin by introducing some operators that can be viewed as Frobenius algebra deformations of divided difference operators.  We then describe an explicit basis of $\cA_n(F)$ in Theorem~\ref{theo:AnF-basis}.  After a brief discussion of a natural filtration on $\cA_n(F)$ and the associated graded algebra (Section~\ref{subsec:filtration}), we describe the center of $\cA_n(F)$ in Theorem~\ref{theo:center}.  We then introduce analogs of Jucys--Murphy elements (see \eqref{eq:JM-elements}), state a Mackey Theorem for induction/restriction (Theorem~\ref{theo:Mackey}), and introduce intertwining elements (Section~\ref{subsec:interwining-elements}).  For certain choices of Frobenius algebra $F$, the results of this section recover results for the above-mentioned special cases of affine wreath product algebras.  However, the benefit of the current approach is that the proofs are completely uniform.  For example, we are able to unify the treatments of these subjects in Parts I and II of the book \cite{Kle05} by keeping track of the Nakayama automorphism that appears in our general definition.

In Section~\ref{sec:simples}, we turn our attention to representation theory.  Inspired by the methods of \cite{WW08}, we classify the simple $\cA_n(F)$-modules when $\kk$ is an algebraically closed field (Theorem~\ref{theo:simple-modules}) by giving an explicit equivalence of categories that reduces the classification to the known cases of degenerate affine Hecke algebras, affine Sergeev algebras, and wreath product algebras.  Outside of cases \ref{intro-item:degAHA}, \ref{intro-item:affSergeev}, and \ref{intro-item:wreathHecke} above, this classification appears to be new.

In Section~\ref{sec:cyclotomic}, we define cyclotomic quotients $\cA_n^\bC(F)$ of the algebras $\cA_n(F)$ and find explicit bases of these quotients (Theorem~\ref{theo:basis-cyclotomic}).  In cases \ref{intro-item:degAHA}, \ref{intro-item:affSergeev}, and \ref{intro-item:wreathHecke} above, Theorem~\ref{theo:basis-cyclotomic} recovers known results, but with a unified proof.  In other cases, it appears to be new.  In particular, in case \ref{intro-item:symmetric} above, it implies an open conjecture of Kleshchev and Muth (Corollary~\ref{cor:KM3.21}).  We then prove that the cyclotomic quotients are themselves $\N$-graded Frobenius superalgebras, and give an explicit description of the associated Nakayama automorphism (Theorem~\ref{theo:cyclotomic-Frobenius-algebra}).  In case \ref{intro-item:degAHA}, this recovers a known result, but with a more direct proof.  When $F$ is the group algebra of a finite cyclic group, Theorem~\ref{theo:cyclotomic-Frobenius-algebra} also recovers a known result.  However, in other cases, even where $F$ is $\Cl$ (case \ref{intro-item:affSergeev}) or the group algebra of a more general finite group, the result appears to be new.  In particular, in case \ref{intro-item:symmetric}, Theorem~\ref{theo:cyclotomic-Frobenius-algebra} implies another open conjecture of Kleshchev and Muth (Corollary~\ref{cor:KM3.22}).  We then state a cyclotomic Mackey Theorem (Theorem~\ref{theo:cyclotomic-Mackey}).  Finally, under some restrictions on the cyclotomic quotient parameter $\bC$ allowing us to view $\cA_n^\bC(F)$ as a subalgebra of $\cA_{n+1}^\bC(F)$, we prove that $\cA_{n+1}^\bC(F)$ is a Frobenius extension of $\cA_n^\bC(F)$ (Theorem~\ref{theo:cyclotomic-Frobenius-extension}).  In particular, induction is both left and right adjoint to restriction up to degree shift (Corollary~\ref{cor:biadjoint})---a property that plays a vital role in categorification.

The fact that affine wreath product algebras are so general, yet one is able to deduce so much of their structure and representation theory explicitly, leads us to believe that these algebras will play an important role in many of the areas where degenerate affine Hecke algebras and their analogs appear.  We conclude the paper in Section~\ref{sec:future} with a brief discussion of some such future directions.  For the benefit of the reader, in Appendix~\ref{app:notation} we summarize the notation used in the paper, including an index of notation.

\iftoggle{detailsnote}{
\subsection*{Note on the arXiv version} For the interested reader, the tex file of the \href{https://arxiv.org/abs/1709.02998}{arXiv version} of this paper includes hidden details of some straightforward computations and arguments that are omitted in the pdf file.  These details can be displayed by switching the \texttt{details} toggle to true in the tex file and recompiling.
}{}

%------------------------------------------
\subsection*{Relation to published version}
%------------------------------------------

This version of the paper contains some corrections of errors that were noticed after publication.  Most are minor typos.  The exception to this is the correction of \eqref{eq:t-L-hat-n-action}.  We thank Eduardo \mbox{Mendon\c{c}a} for pointing out the error in this equation.

%-----------------------------
\subsection*{Acknowledgements}
%-----------------------------

We thank Jonathan Brundan, Alexander Kleshchev, and Weiqiang Wang for useful conversations.  We also thank the referees for helpful comments that improved the paper.   This research was supported by Discovery Grant RGPIN-2017-03854 from the Natural Sciences and Engineering Research Council of Canada.

%%%%%%%%%%%%%%%%%%%%%%%
%
\section{Preliminaries\label{sec:prelim}}
%
%%%%%%%%%%%%%%%%%%%%%%%

We recall here some basic facts about graded associative superalgebras and graded Frobenius superalgebras.  Throughout, $\kk$ is an arbitrary commutative ring of characteristic not equal to $2$, unless otherwise stated.

%--------------------------------------------------
\subsection{Graded superalgebras and their modules\label{subsec:superalg-background}}
%--------------------------------------------------

Throughout this paper, we will use the term \emph{algebra} to mean $\Z$-graded associative $\kk$-superalgebra.  Similarly, the term \emph{module} will mean $\Z$-graded left supermodule.  (In the statement of theorems, etc.\ we will sometimes include the words ``graded'' and ``super'' for reference purposes.)  We use the term \emph{degree} to refer to the $\Z$-grading and \emph{parity} to refer to the $\Z_2$-grading.  Similarly, \emph{degree shift} refers to a shift in the $\Z$-grading.  We assume that algebras are concentrated in nonnegative degree (i.e.\ they are $\N$-graded).  For a homogeneous element $a$ of an algebra or module, we denote its parity by $\bar a$\label{bar-def} and its degree by $|a|$.\label{degree-notation}  Whenever we give definitions involving the parity, it is understood that we extend by linearity to nonhomogeneous elements.

Let $A$ be an algebra.  The \emph{opposite} algebra $A^\op$ has the same underlying $\kk$-module structure, but multiplication is given by
\begin{equation} \label{eq:opposite-alg-def}
  a \cdot b = (-1)^{\bar a \bar b} ba,\quad a,b \in A^\op,
\end{equation}
where juxtaposition denotes the original multiplication in $A$.

For $A$-modules $M$ and $N$, we define $\HOM_\kk(M,N)$ to be the space of all $\kk$-linear maps from $M$ to $N$.  For $i \in \Z$ and $\varepsilon \in \Z_2$, we let
\begin{multline*}
  \HOM_A(M,N)_{i,\varepsilon} := \{\alpha \in \HOM_\kk(M,N) \mid \alpha(am) = (-1)^{\varepsilon \bar a} a \alpha(m) \ \forall\ a \in A,\ m \in M,\\
  \alpha(M_{j,\varepsilon'}) \subseteq N_{i+j,\varepsilon + \varepsilon'} \ \forall\ j \in \Z,\ \varepsilon' \in \Z_2\}
\end{multline*}
denote the space of all homogeneous $A$-module maps of degree $i$ and parity $\varepsilon$.  We then define
\begin{gather}
  \HOM_A(M,N) := \bigoplus_{i \in \Z,\, \varepsilon \in \Z_2} \HOM_A(M,N)_{i,\varepsilon}, \label{eq:HOM-def} \\
  \Hom_A(M,N) := \HOM_A(M,N)_{0,0}. \label{eq:hom-def}
\end{gather}
The notation $\END$ and $\End$ is defined similarly.  We will use the terms \emph{homomorphism} and \emph{isomorphism} to mean elements of $\HOM_A(M,N)$.  The terms \emph{even homomorphism} and \emph{even isomorphism} refer to elements of $\Hom_A(M,N)$.  We use the symbol $\cong$ to denote an isomorphism (i.e.\ an invertible element of $\HOM_A(M,N)$) and $\simeq$ to denote an even isomorphism (i.e.\ an invertible element of $\Hom_A(M,N)$).\label{cong-simeq-def}  We let $A\md$ denote the category of finitely-generated $A$-modules with even homomorphisms.

We have a parity shift functor
\[
  \Pi \colon A\md \to A\md,\quad M \mapsto \Pi M,
\]
that switches the $\Z_2$-grading, that is, $(\Pi M)_{i,\varepsilon} = M_{i,\varepsilon+1}$.  The action of $A$ on $\Pi M$ is given by $a \cdot m = (-1)^{\bar a} am$, where $am$ is the action on $M$.

We let $\cS(A)$\label{cS_simeq-def} denote the set of \emph{even} isomorphism classes of simple $A$-modules (i.e.\ isomorphism classes in $A\md$) up to degree shift.  By a slight abuse of notation, we will also use the notation $\cS(A)$ to denote a set of representatives of these isomorphism classes, where we shift degrees so that each representative has nonzero degree zero piece and is concentrated in nonnegative degree.

Given an $A$-module $V$ and an algebra automorphism $\alpha$ of $A$, we define the twisted module $\prescript{\alpha}{}{V}$ to be the $A$-module with underlying $\kk$-module $V$ and with action
\begin{equation} \label{eq:twisted-module-def}
  a \cdot v = \alpha^{-1}(a) v,\quad a \in A,\ v \in V.
\end{equation}
\details{
  We take the inverse above so that $\prescript{\alpha_1}{}{\left(\prescript{\alpha_2}{}{V} \right)} \simeq \prescript{\alpha_1 \alpha_2}{}{V}$ for $\alpha_1, \alpha_2 \in \Aut A$.  Thus we have an action on $\Aut A$ on the class of $A$-modules.
}
Here and in what follows we use $\cdot$ to denote the twisted action and juxtaposition to denote the original (untwisted) action.  Since $\prescript{\alpha}{}{V}$ is simple if $V$ is, the twisting induces a permutation of $\cS(A)$.

A simple $A$-module is said to be of type $\tQ$ if it is evenly isomorphic to its own parity shift.  Otherwise it is said to be of type $\tM$.  Schur's Lemma for superalgebras states that, for simple $A$-modules $M$ and $N$ with $M \not \cong N$, we have $\HOM_A(M,N) = 0$.  On the other hand, if $M$ is simple and $\kk$ is an algebraically closed field, then
\[
  \END_A(M) \simeq
  \begin{cases}
    \kk & \text{if } M \text{ is of type } \tM, \\
    \Cl & \text{if } M \text{ is of type } \tQ.
  \end{cases}
\]
Here $\Cl$ is the two-dimensional Clifford algebra with one odd generator $c$ satisfying $c^2=1$.\label{Cl-def}

Unadorned tensor products will be understood to be over $\kk$.  Recall that multiplication in the tensor product of algebras $A_1$ and $A_2$ is given by
\[
  (a_1 \otimes a_2) (b_1 \otimes b_2) = (-1)^{\bar a_2 \bar b_1} (a_1 a_2 \otimes b_1 b_2),\quad a_1,b_1 \in A_1,\ a_2, b_2 \in A_2.
\]
The outer tensor product of an $A_1$-module $M$ and an $A_2$-module $N$ is denoted $M \boxtimes N$.\label{boxtimes-def}  As $\kk$-modules, we have $M \boxtimes N = M \otimes N$.  The action is given by
\[
  (a_1 \otimes a_2)(m \otimes n) = (-1)^{\bar a_2 \bar m} a_1 m \otimes a_2 n,\quad a_1 \in A_1,\ a_2 \in A_2,\ m \in M,\ n \in N.
\]

For the remainder of this subsection, we assume $\kk$ is an algebraically closed field.  Suppose $M$ and $N$ are simple modules.  If at least one of them is of type $\tM$, then $M \boxtimes N$ is a simple $A_1 \otimes A_2$-module.  However, if $M$ and $N$ are both of type $\tQ$, then
\[
  M \boxtimes N \simeq (M \irtimes N) \oplus \Pi(M \irtimes N)
\]
for some simple submodule $M \irtimes N \subseteq M \boxtimes N$ of type $\tM$.\label{irtimes-def}
\details{Here one needs that the characteristic of $\kk$ is not equal to two.}
The simple submodule $M \irtimes N$ is unique up to even isomorphism and parity shift.  We will also use the notation $M \irtimes N$ to denote $M \boxtimes N$ in the case where $M$ or $N$ is of type $\tM$.  It follows that, if $M_1,N_1$ are simple $A_1$-modules, and $M_2,N_2$ are simple $A_2$-modules, then for any nonzero even $A_1 \otimes A_2$-module homomorphism $M_1 \boxtimes M_2 \to N_1 \boxtimes N_2$, we have a nonzero induced even homomorphism $M_1 \irtimes M_2 \to N_1 \irtimes N_2$, (making a different choice for $M_1 \irtimes M_2$ or $N_1 \irtimes N_2$ if necessary).  Furthermore, every simple $A_1 \otimes A_2$-module is isomorphic to $M_1 \irtimes M_2$ for some simple $A_1$-module $M_1$ and simple $A_2$-module $M_2$.

%------------------------------
\subsection{Frobenius algebras\label{subsec:Frobenius-algebras}}
%------------------------------

Fix an $\N$-graded Frobenius $\kk$-superalgebra $F$\label{F-def} with homogeneous (in the $\N$-degree) parity-preserving linear trace map $\tr \colon F \to \kk$\label{tr-def}.  In other words, the map
\begin{equation} \label{eq:trace-nondegen}
  F \to \Hom_\kk(F,\kk),\quad f \mapsto \left( g \mapsto (-1)^{\bar f \bar g} \tr(gf) \right),
\end{equation}
is a homogeneous parity-preserving isomorphism.  (One could also allow the trace map $\tr$ to be parity-reversing, but we will not do so in the current paper.  See Remark~\ref{rem:odd-trace-map}.)  We will assume that $F$ is free as a $\kk$-module.  Let $\psi$\label{Nakayama-def} denote the Nakayama automorphism, so that
\[
  \tr(fg) = (-1)^{\bar f \bar g} \tr \left( g\psi(f) \right) \quad \text{for all } f,g \in F.
\]
Let $\delta \in \N$\label{delta-def} be the maximum degree of elements of $F$, so that $\tr$ is of degree $-\delta$.  We make the following assumptions:
\begin{itemize}
  \item The Nakayama automorphism has finite order $\theta$.\label{theta-def}

  \item The characteristic of $\kk$ does not divide $\theta$.

  \item The action of the Nakayama automorphism is diagonalizable (i.e.\ there exists a basis of $F$ consisting of eigenvectors of $\psi$).
\end{itemize}
Note that if $\kk$ is an algebraically closed field, the first two assumptions imply the third.  (These assumptions are not necessary in the current section, but are needed later in the paper.)

\details{
  For the curious reader not familiar with Frobenius algebras, we give here an example of a Frobenius algebra with Nakayama automorphism of infinite order.  Fix $\zeta \in \C^\times$, and let $F$ be the $\C$-algebra with generators $y,z$, subject to the relations
  \[
    y^2 = z^2 = 0,\quad yz = \zeta zy.
  \]
  Then $F$ has a basis given by $y^k z^\ell$, $k,\ell \in \{0,1\}$, and is a Frobenius algebra under the trace map determined by $\tr(y^k z^\ell) = \delta_{k,1} \delta_{\ell,1}$.

  The Nakayama automorphism is determined by
  \[
    \psi(z) = \zeta^{-1} z,\quad \psi(y) = \zeta y.
  \]
  Indeed, first we compute $\psi(z)$.  For $\ell \in \{0,1\}$, we have
  \[
    \tr \left( z z^\ell \right) = 0 = \tr \left( z^\ell (\zeta^{-1} z) \right) \quad \text{and} \quad
    \tr \left( z(y^\ell z) \right) = 0 = \tr \left( (y^\ell z)(\zeta^{-1} z) \right).
  \]
  Finally we have $\tr (zy) = \tr \left( y (\zeta^{-1} z) \right)$.  Now we compute $\psi(y)$.  For $\ell \in \{0,1\}$, we have
  \[
    \tr \left( y y^\ell \right) = 0 = \tr \left( y^\ell(\zeta y) \right)
    \quad \text{and} \quad
    \tr \left( y (y z^\ell) \right) = 0 = \tr \left( (y z^\ell)(\zeta y) \right).
  \]
  Finally, we have $\tr(yz) = \tr \left( z (\zeta y) \right)$.

  Thus, the order of $\psi$ is the multiplicative order of $\zeta$.  In particular, if $\zeta$ is not a root of unity, then $\psi$ has infinite order.
}

Let $B$\label{B-def} be a basis of $F$ consisting of homogeneous elements, and let $B^\vee = \{b^\vee \mid b \in B\}$\label{dual-basis-def} be the left dual basis.  Thus
\[
  \tr(b^\vee c) = \delta_{b,c} \quad b,c \in B.
\]
It follows that, for all $f \in F$, we have

\noindent\begin{minipage}{0.5\linewidth}
  \begin{equation} \label{eq:f-in-basis}
    \sum_{b \in B} b \tr(b^\vee f) = f,
  \end{equation}
\end{minipage}%
\begin{minipage}{0.5\linewidth}
  \begin{equation} \label{eq:f-in-dual-basis}
    \sum_{b \in B} \tr(fb) b^\vee = f.
  \end{equation}
\end{minipage}\par\vspace{\belowdisplayskip}
\noindent
\details{
  We can write $f = \sum_{c \in B} f_c c$ for $f_c \in \kk$.  Then
  \[
    \sum_{b \in B} b \tr(b^\vee f)
    = \sum_{b,c \in B} f_c b \tr(b^\vee c)
    = \sum_{c \in B} f_c c
    = f.
  \]
  We can also write $f = \sum_{c \in B} f_{c^\vee} c^\vee$ for $f_{c^\vee} \in \kk$.  Then
  \[
    \sum_{b \in B} \tr(fb) b^\vee
    = \sum_{b,c \in B} f_{c^\vee} \tr(c^\vee b) b^\vee
    = \sum_{b,c \in B} f_{c^\vee} c^\vee
    = f.
  \]
}
We also have
\begin{equation} \label{eq:double-dual}
  \left( b^\vee \right)^\vee = (-1)^{\bar b} \psi^{-1}(b).
\end{equation}
\details{
  For $b, c \in B$, we have
  \[
    \delta_{b,c}
    = \tr (b^\vee c)
    = (-1)^{\bar b \bar c} \tr \left( \psi^{-1}(c) b^\vee \right),
  \]
  and so $(b^\vee)^\vee = (-1)^{\bar b} \psi^{-1}(b)$.
}

\begin{rem}
  When comparing statements in the current paper to those of \cite{RS17}, the reader should be aware that the notation $\vee$ denotes \emph{right} duals in \cite{RS17}, as opposed to left duals as in the current paper.
\end{rem}

Suppose $F$ and $F'$ are Frobenius algebras with trace maps $\tr$ and $\tr'$, respectively.  A \emph{homomorphism of Frobenius algebras} is an algebra homomorphism $\xi \colon F \to F'$ such that $\tr = \tr' \circ \xi$.  If $\xi$ is invertible, we call it an \emph{isomorphism of Frobenius algebras}.  Note that all homomorphisms of Frobenius algebras are injective.
\details{
  Suppose $\xi(f)=0$.  Then, for all $g \in F$, we have
  \[
    \tr(gf)
    = \tr' (\xi(gf))
    = \tr' (\xi(g) \xi(f))
    = \tr'(0)
    = 0.
  \]
  It then follows from the nondegeneracy of $\tr$ that $f=0$.
}
If $\xi \colon F \to F'$ is an isomorphism of Frobenius algebras, then $\xi \circ \psi = \psi' \circ \xi$, where $\psi$ and $\psi'$ are the Nakayama automorphisms of $F$ and $F'$, respectively.
\details{
  For $f, g \in F$, we have
  \begin{multline*}
    \tr' \big( \xi(f) \xi (\psi(g)) \big)
    = \tr' \big( \xi(f \psi(g)) \big)
    = \tr \big( f\psi(g) \big)
    = (-1)^{\bar f \bar g} \tr(gf)
    \\
    = (-1)^{\bar f \bar g} \tr' \big( \xi(gf) )
    = (-1)^{\bar f \bar g} \tr' \big( \xi(g) \xi(f) \big)
    = \tr' \big( \xi(f) \psi'(\xi(g) \big).
  \end{multline*}
  Since $\xi$ is surjective, it follows from the nondegeneracy of $\tr'$ that $\xi \circ \psi = \psi' \circ \xi$.
}
The opposite algebra $F^\op$ is also a Frobenius algebra, with trace map $\tr$ and Nakayama automorphism $\psi^{-1}$.
\details{
  Let $\tr$ and $\psi$ denote the trace map and Nakayama automorphism of $F$.  Then $\tr$ is clearly a linear map $F^\op \to \kk$.  Let $\cdot$ denote multiplication in $F^\op$.  Suppose $f \in F$ has the property that $\tr(g \cdot f)=0$ for all $g \in F$.  Then, for all $g \in F$, we have
  \[
    0
    = \tr(g \cdot f)
    = (-1)^{\bar f \bar g} \tr(fg)
    = \tr(\psi^{-1}(g)f).
  \]
  Thus, by the nondegeneracy of $\tr$, we have $f=0$.  Thus the map
  \[
    F^\op \to \Hom_\kk(F^\op,\kk),\quad f \mapsto \left( g \mapsto (-1)^{\bar f \bar g} \tr(g \cdot f) \right),
  \]
  is injective.  A similar argument shows that is also surjective.

  Now, for $f,g \in F$, we have
  \[
    \tr(f \cdot g)
    = (-1)^{\bar f \bar g} \tr(gf)
    = \tr \left( \psi^{-1}(f)g \right)
    = (-1)^{\bar f \bar g} \tr \left( g \cdot \psi^{-1}(f) \right).
  \]
  Thus $\psi^{-1}$ is the Nakayama automorphism of $F^\op$.
}

\begin{lem} \label{lem:Frob-antiautom-properties}
  Suppose $\tau \colon F \to F^\op$ is an isomorphism of Frobenius algebras.  Then $\tau \circ \psi = \psi^{-1} \circ \tau$.  Furthermore, the left duals of the elements of the basis $\{\tau(b^\vee) \mid b \in B\}$ of $F$ are $\tau(b^\vee)^\vee = (-1)^{\bar b} \tau(b)$.
\end{lem}

\begin{proof}
  That $\tau \circ \psi = \psi^{-1} \circ \tau$ follows from the fact that $\psi^{-1}$ is the Nakayama automorphism of $F^\op$.  To prove the second assertion note that, for all $b,c \in B$, we have
  \[
    \delta_{b,c}
    = \tr(b^\vee c)
    = \tr(\tau(b^\vee c))
    = (-1)^{\bar b \bar c} \tr(\tau(c) \tau(b^\vee)),
  \]
  which implies that $\tau(b^\vee)^\vee = (-1)^{\bar b} \tau(b)$.
\end{proof}

The following example will help us illustrate some of the concepts to be discussed later in the paper.

\begin{eg}[Taft Hopf algebra] \label{eg:Taft}
  Fix $q \in \N$ and a primitive $q$-th root of unity $\omega \in \C$.  Consider the $\C$-algebra $T_q$ with generators $g,y$ and relations $g^q=1$, $y^q=0$, and $yg = \omega g y$.  This algebra can be given the structure of a Hopf algebra, and is called the \emph{Taft Hopf algebra}.  We can define a grading on $T_q$ by declaring $g$ to be even of degree zero and $y$ to be even of arbitrary degree.  The algebra $T_q$ has basis given by $y^k g^\ell$, $0 \le k,\ell \le q-1$.  It is straightforward to verify that the linear map $\tr \colon T_q \to \kk$ determined by $\tr(y^k g^\ell) = \delta_{k,q-1} \delta_{\ell,0}$ is nondegenerate and hence makes $T_q$ a Frobenius algebra.  The Nakayama automorphism is given by $\psi(g) = \omega g$ and $\psi(y)=y$.
  \details{
    For $k \ne q-1$ or $0 \le \ell < q-1$, we have
    \[
      \tr \left( g y^k g^\ell \right) = 0 = \tr \left( y^k g^\ell \omega g \right).
    \]
    We also have
    \[
      \tr \left( g y^{q-1} g^{q-1} \right)
      = \omega^{1-q} \tr \left( y^{q-1} \right)
      = \omega
      = \tr \left( y^{q-1} g^{q-1} w g \right).
    \]
    Thus $\psi(g) = \omega g$.

    For $k \ne q-2$ or $0 < \ell \le q-1$, we have
    \[
      \tr \left( y y^k g^\ell \right) = 0 = \tr \left( y^k g^\ell y \right).
    \]
    We also have
    \[
      \tr \left( y y^{q-2} \right)
      = 1
      = \tr \left( y^{q-2} y \right).
    \]
    Thus $\psi(y) = y$.
  }

  Since the element $y$ is nilpotent, simple modules for $F$ are equivalent to simple modules for $T_q/\langle y \rangle$, which is isomorphic to the group algebra of the cyclic group of order $q$.  Thus, a complete list of simple $T_q$-modules, up to isomorphism, is given by $L_0,\dotsc,L_{q-1}$, where, for $0 \le k \le q-1$, we have $L_k \simeq \C$ as $\C$-modules, and
  \[
    g \cdot z = e^{2 k \pi i/q} z,\quad y \cdot z = 0,\quad z \in L_k.
  \]
\end{eg}

\details{
  Another interesting example (not used in the current paper) is the cohomology ring of a compact oriented manifold.  This is a graded commutative Frobenius algebra with trace map given by evaluation on the fundamental homology class of the manifold.
}

%------------------------------------------------------
\subsection{Smash products and wreath product algebras\label{subsec:smash}}
%------------------------------------------------------

Suppose we have a left action of a group $G$ on a $\kk$-algebra $A$ via algebra automorphisms.  Then we can form the \emph{smash product algebra} $A \rtimes G$.  As a $\kk$-module, we have
\[
  A \rtimes G = A \otimes \kk G.
\]
The multiplication is determined by the fact that $A$ and $\kk G$ are subalgebras, and
\begin{equation} \label{rel:smash-product}
  w a = (w \cdot a) w,\quad w \in \kk G,\ a \in A.
\end{equation}
As a special case of the above construction, we have the natural action of $S_n$ on $F^{\otimes n}$ by superpermuting the factors:
\[
  s_i \cdot (f_1 \otimes \dotsb \otimes f_n)
  = (-1)^{\bar f_i \bar f_{i+1}} f_1 \otimes \dotsb \otimes f_{i-1} \otimes f_{i+1} \otimes f_i \otimes f_{i+2} \otimes \dotsb \otimes f_n,
\]
where $s_i$, $1 \le i \le n-1$, denotes the simple transposition of $i$ and $i+1$.  We use $\prescript{\pi}{}{a}$\label{perm-notation} to denote $\pi \cdot a$ in this case.  We call $F^{\otimes n} \rtimes_\rho S_n$ the \emph{wreath product algebra} associated to $F$ where, here and in what follows, we use the subscript $\rho$ when the action is by superpermutations.\label{rtimes_rho-def}

The above construction is a special case of the more general notion of a Hopf smash product.  Another example that will be important for us is the following.  An algebra automorphism $\alpha$ of an algebra $A$ gives rise to a natural right action of the polynomial algebra $\kk[x]$ on $A$.  Then we can form the algebra $\kk[x] \ltimes A$.  Again, we have
\[
  \kk[x] \ltimes A = \kk[x] \otimes A
\]
as $\kk$-modules.  The multiplication is determined by the fact that $\kk [x]$ and $A$ are subalgebras, and
\[
  a x = x \alpha(a),\quad a \in A.
\]

%%%%%%%%%%%%%%%%%%%%%%%%%%%%%%%%%%%%%%%%
%
\section{Affine wreath product algebras\label{sec:affine-wreath}}
%
%%%%%%%%%%%%%%%%%%%%%%%%%%%%%%%%%%%%%%%%

In this section we introduce our main object of study, the affine wreath product algebras.  In this section $\kk$ is an arbitrary commutative ring of characteristic not equal to two.

%----------------------
\subsection{Definition}
%----------------------

Recall that $F$ is an $\N$-graded Frobenius superalgebra with basis $B$, Nakayama automorphism $\psi$, and top degree $\delta$.  Fix $n \in \N_+$.  For $f \in F$ and $1 \le i \le n$, we define
\begin{gather}
  f_i = 1^{\otimes (i-1)} \otimes f \otimes 1^{\otimes (n-i)} \in F^{\otimes n}, \label{eq:f_i-def} \\
  \psi_i = \id^{\otimes (i-1)} \otimes \psi \otimes \id^{\otimes (n-i)} \colon F^{\otimes n} \to F^{\otimes n}. \label{eq:psi_i-def}
\end{gather}

\begin{defin}[Affine wreath product algebra $\cA_n(F)$] \label{def:affine-wreath}
  We define the \emph{affine wreath product algebra} $\cA_n(F)$ to be the $\N$-graded superalgebra that is the free product of $\kk$-algebras
  \[
    \kk[x_1,\dotsc,x_n] \star F^{\otimes n} \star \kk S_n,
  \]
  modulo the relations
  \begin{align}
    \bsf x_i &= x_i \psi_i(\bsf),& 1 \le i \le n,\ \bsf \in F^{\otimes n}, \label{rel:xF-commutation} \\
    s_i x_j &= x_j s_i,& 1 \le i \le n-1,\ 1 \le j \le n,\ j \ne i,i+1, \label{rel:sx-triv-commutation} \\
    s_i x_i &= x_{i+1} s_i - t_{i,i+1},& 1 \le i \le n-1, \label{rel:sx-commutation} \\
    \pi \bsf &= \prescript{\pi}{}{\bsf} \pi,& \pi \in S_n,\ \bsf \in F^{\otimes n}, \label{rel:SF-commutation}
  \end{align}
  where
  \begin{equation} \label{def:tij}
    t_{i,j} := \sum_{b \in B} b_i b_j^\vee
    \quad \text{for} \quad
    1 \le i,j \le n,\ i \ne j,
  \end{equation}
  and where we always interpret $b^\vee_j$ as $\left( b^\vee \right)_j$.  The degree and parity on $\cA_n(F)$ are determined by
  \begin{gather}
    |x_i| = \delta,\quad \bar x_i = 0,\quad 1 \le i \le n, \label{eq:x-degrees} \\
    |f_i| = |f|,\quad \bar f_i = \bar f,\quad 1 \le i \le n,\ f \in F, \nonumber \\
    |\pi| = 0,\quad \bar \pi = 0,\quad \pi \in S_n. \nonumber
  \end{gather}
  By convention, we set $\cA_0(F) = \kk$.
\end{defin}

Note that $|t_{i,j}| = \delta$ and $\overline{t_{i,j}} = 0$ for all $i,j$.  It is straightforward to verify that $t_{i,j}$ is independent of the choice of basis $B$.
\details{
  Enumerate the elements of $B$ so that $B = \{b^{(1)},\dotsc,b^{(a)}\}$.  Let $B' = \{c^{(1)},\dotsc,c^{(a)}\}$ be another basis of $F$.  Then there exist invertible $a \times a$ matrices $M = (m_{k \ell})$ and $M' = (m_{k \ell}')$ with entries in $\kk$ such that, for $k = 1, \dotsc, a$,
  \[
    b^{(k)}
    = \sum_{\ell=1}^a m_{k \ell} c^{(\ell)},
    \quad
    {b^{(k)}}^\vee
    = \sum_{\ell=1}^a m_{k \ell}' {c^{(\ell)}}^\vee.
  \]
  Then, for $k,\ell \in \{1,\dotsc,a\}$, we have
  \[
     \delta_{k,\ell}
     = \tr \left( b^{(k)} {b^{(\ell)}}^\vee \right)
     = \sum_{r,s=1}^a m_{k r} m_{\ell s}' \tr \left( c^{(r)} {c^{(s)}}^\vee \right)
     = \sum_{r,s=1}^a m_{k r} m_{\ell s}' \delta_{r,s}
     = \sum_{r=1}^a m_{k r} m_{\ell r}'.
  \]
  It follows that $(M')^T = M^{-1}$.  Thus
  \[
    \sum_{k=1}^a b^{(k)}_i {b^{(k)}}_j^\vee
    = \sum_{k,\ell,r=1}^a m_{k \ell} m_{k r}' c^{(\ell)}_i {c^{(r)}}_j^\vee
    = \sum_{\ell=1}^a c^{(\ell)}_i {c^{(\ell)}}_j^\vee.
  \]
}
Conjugating by $s_i$, we see that \eqref{rel:sx-commutation} is equivalent to
\begin{equation} \label{rel:sx-commutation2}
  s_i x_{i+1} = x_i s_i + t_{i+1,i}.
\end{equation}

It follows from Definition~\ref{def:affine-wreath} that we have an algebra homomorphism
\[
  F^{\otimes n} \rtimes_\rho S_n \to \cA_n(F),\quad
  \bsf \mapsto \bsf,\ \pi \mapsto \pi, \quad
  \bsf \in F^{\otimes n},\ \pi \in S_n.
\]
We also have a natural algebra homomorphism
\[
  \kk[x_1,\dotsc,x_n] \to \cA_n(F).
\]
We will use these homomorphisms to view elements of $F^{\otimes n} \rtimes_\rho S_n$ and $\kk[x_1,\dotsc,x_n]$ as elements of $\cA_n(F)$.  In fact, we will see in Theorem~\ref{theo:AnF-basis} that both of the above homomorphisms are injective, allowing us to view $F^{\otimes n} \rtimes_\rho S_n$ and $\kk[x_1,\dotsc,x_n]$ as subalgebras of $\cA_n(F)$.

\begin{lem} \label{lem:independence-of-trace-map}
  Up to isomorphism, $\cA_n(F)$ depends only on the underlying algebra $F$, and not on the trace map $\tr$.
\end{lem}

\begin{proof}
  Trace maps for a given algebra differ by multiplication by an invertible element.  (For the graded super setting of the current paper, see \cite[Prop.~4.7]{PS16}.)  Thus, if we have another trace map $\tr'$, there exists an invertible $u \in F^\times$ (which must be even since both trace maps are even) such that $\tr'(f) = \tr(fu)$ for all $f \in F$.  It is then straightforward to verify that the map determined by
  \begin{equation} \label{eq:trace-change-isom}
    x_i \mapsto x_i u_i,\
    \bsf \mapsto \bsf,\
    \pi \mapsto \pi,\quad
    1 \le i \le n,\ \bsf \in F^{\otimes n},\ \pi \in S_n,
  \end{equation}
  is an isomorphism of algebras from $\cA_n(F)$ to $\cA_n(F')$, where $F'$ is the Frobenius algebra that is the same underlying algebra as $F$, but with trace map $\tr'$.
  \details{
    Under $\tr'$, the left dual of the basis $\{bu^{-1} \mid b \in B\}$ is $\{b^\vee \mid b \in B\}$, where $b^\vee$ denotes, as usual, the dual of $b$ (in the basis $B$) with respect to $\tr$.  Thus, using primes to denote quantities in $\cA_n(F')$, we have
    \[
      t_{i,j}'
      = \sum_{b \in B} b_i u_i^{-1} b_j^\vee
      = \sum_{b \in B} b_i b_j^\vee u_i^{-1}
      = t_{i,j} u_i^{-1}.
    \]
    Now, for $f,g \in F$, we have
    \[
      \tr'(fg)
      = \tr(fgu)
      = (-1)^{\bar f \bar g} \tr \left( gu \psi(f) \right)
      = (-1)^{\bar f \bar g} \tr' \left( gu\psi(f)u^{-1} \right).
    \]
    Therefore, the Nakayama automorphism for $\tr'$ is given by $\psi'(f) = u \psi(f) u^{-1}$.

    Let $\phi$ be the map \eqref{eq:trace-change-isom}.  Then
    \[
      \phi (\bsf x_i)
      = \bsf x_i u_i
      \stackrel{\eqref{rel:xF-commutation}}{=} x_i \psi_i'(\bsf) u_i
      = x_i u_i \psi_i(\bsf)
      = \phi \left( x_i \psi_i(\bsf) \right).
    \]
    So $\phi$ respects \eqref{rel:xF-commutation}.  We also have
    \[
      \phi (s_i x_i)
      = s_i x_i u_i
      \stackrel{\eqref{rel:sx-commutation}}{=} x_{i+1} s_i u_i - t_{i,i+1}' u_i
      \stackrel{\eqref{rel:SF-commutation}}{=} x_{i+1} u_{i+1} s_i - t_{i,i+1}
      = \phi(x_{i+1} s_i - t_{i,i+1}).
    \]
    So $\phi$ respects \eqref{rel:sx-commutation}.  It is clear that $\phi$ respects \eqref{rel:sx-triv-commutation} and \eqref{rel:SF-commutation}.  Hence $\phi$ is a homomorphism of algebras.  Since it is clearly invertible, the proof is complete.
  }
\end{proof}

\begin{rem} \label{rem:odd-trace-map}
  We could allow the trace map of $F$ to be parity-reversing.  Then we would define the $x_i$ to be odd and we would require them to anticommute.  In addition, \eqref{rel:xF-commutation} would involve signs.  While it would be interesting to investigate the resulting ``odd wreath product algebras'', for simplicity we restrict our attention in the current paper to the case where the trace map of $F$ is parity-preserving.
\end{rem}

%--------------------
\subsection{Examples}
%--------------------

Before investigating the properties of the algebra $\cA_n(F)$, we first discuss how various choices of the Frobenius algebra $F$ recover well-studied algebras.  In its full generality, the algebra $\cA_n(F)$ first appeared in \cite{RS17}, where it occurred naturally in the endomorphism space of the object $\mathsf{Q}^n$ of the diagrammatic category $\cH_F$.  More precisely, $\cA_n(F)$ is isomorphic to the opposite algebra of the algebra $D_n$ defined in \cite[Def.~8.12]{RS17}, after relabelling indices by interchanging $i$ and $n-i$.

\begin{eg}[Degenerate affine Hecke algebra] \label{eg:daHa}
  The algebra $\cA_n(\kk)$ is the usual degenerate affine Hecke algebra.  We have $t_{i,j}=1$ for all $i,j$.
\end{eg}

\begin{eg}[Wreath Hecke algebra] \label{eg:wreath-Hecke}
  Fix a finite group $G$ and consider the group algebra $\kk G$ (with trivial grading) with trace map $\tr \colon \kk G \to \kk$ given by $\tr \left( \sum_{g \in G} a_g g \right) = a_e$, where $e$ is the identity element of $G$.  Then
  \[
    t_{i,j} = \sum_{g \in G} g_i g_j^{-1},\quad i \ne j,
  \]
  and the Nakayama automorphism $\psi$ is trivial.  Thus, $\cA_n(\kk G)$ is the wreath Hecke algebra of \cite[Def.~2.4]{WW08}.  Of course, there are other trace maps for $G$, given by projecting onto the coefficients of other elements of the group.  However, by Lemma~\ref{lem:independence-of-trace-map}, these yield isomorphic affine wreath product algebras.
\end{eg}

\begin{eg}
  Let $F = \kk[z]/(z^2)$, with $|z|=2$ and $\bar z = 0$.  This is an $\N$-graded Frobenius superalgebra with trace map $\tr(a+bz) = b$ for $a,b \in \kk$.  The Nakayama automorphism is trivial, $1^\vee = z$, and $z^\vee = 1$.  Thus, $t_{i,j} = z_i + z_j$.  Then $\cA_n(F)$ is precisely the algebra $H_i^n$ defined in \cite[\S10.3]{CL12}.
\end{eg}

\begin{eg}[Affine Sergeev algebra] \label{eg:affine-Sergeev}
  Consider the Clifford superalgebra $\Cl$ generated by a single odd generator $c$ satisfying $c^2=1$.  This is a Frobenius superalgebra with trace given by $\tr(c)=0$, $\tr(1)=1$.  Then the Nakayama automorphism satisfies $\psi(c)=-c$.  If we choose $B = \{1,c\}$, then $1^\vee = 1$ and $c^\vee = c$.  Thus,
  \[
    t_{i,j} = 1_i 1_j^\vee + c_i c_j^\vee = 1 + c_i c_j.
  \]
  It follows that $\cA_n(\Cl)$ is the algebra introduced in \cite[\S3]{Naz97}, where it was called the degenerate affine Sergeev algebra.  It is also sometimes called the degenerate affine Hecke--Clifford algebra.  Note that \cite[\S3]{Naz97} uses a slightly different presentation, where $c^2=-1$.  Here we follow the conventions used in \cite[\S14.1]{Kle05}.
\end{eg}

\begin{eg}[Affine zigzag algebra]
  When $F$ is a certain skew-zigzag algebra (see \cite[\S3]{HK01} and \cite[\S5]{Cou16}), the algebras $\cA_n(F)$ appear in the endomorphism algebras of the categories constructed in \cite{CL12} to study Heisenberg categorification and the geometry of Hilbert schemes.  They were then also considered in \cite{KM15}, where they were related to imaginary strata for quiver Hecke algebras (also known as KLR algebras).
  \details{
    The skew-zigzag algebras considered in \cite{CL12} are symmetric (i.e.\ $\psi = \id$) in the super sense, but not in the usual (i.e.\ non-super) sense.  The zigzag algebras appearing in \cite{KM15} (which assumes all algebras are even) differ by some signs from those of \cite{CL12}, and are symmetric in the usual sense.
  }
\end{eg}

%-----------------------------------------------
\subsection{Automorphisms\label{subsec:automs}}
%-----------------------------------------------

It is straightforward to verify that we have an algebra automorphism of $\cA_n(F)$ given by
\begin{equation} \label{eq:reverse-automorphism}
  x_i \mapsto x_{n+1-i},\ f_i \mapsto f_{n+1-i},\ s_j \mapsto -s_{n-j},\quad
  1 \le i \le n,\ 1 \le j \le n-1.
\end{equation}
\details{
  Let $\phi$ be the given map.  It suffices to verify that $\phi$ is a homomorphism of algebras, since it squares to the identity.  The only nontrivial relation to verify is \eqref{rel:sx-commutation}.  For $1 \le i \le n-1$, we compute
  \[
    \phi(s_i x_i)
    = -s_{n-i} x_{n+1-i}
    \stackrel{\eqref{rel:sx-commutation2}}{=} - x_{n-i}s_{n-i} - t_{n+1-i,n-i}
    = \phi \left( x_{i+1} s_i - t_{i,i+1} \right).
  \]
}

Any automorphism $\xi \colon F \to F$ of Frobenius algebras induces an algebra automorphism of $\cA_n(F)$ given by
\begin{equation}
  x_i \mapsto x_i,\quad
  \bsf \mapsto \xi^{\otimes n}(\bsf),\quad
  s_j \mapsto s_j,\qquad
  1 \le i \le n,\ \bsf \in F^{\otimes n},\ 1 \le j \le n-1.
\end{equation}
\details{
  Let $\phi$ be the given map.  It suffices to verify that $\phi$ is a homomorphism of algebras, since it is obviously invertible (with inverse given by the automorphism associated to $\xi^{-1}$).  The only nontrivial relation to verify is \eqref{rel:sx-commutation}.  For $1 \le i \le n-1$, we compute
  \[
    \phi(s_i x_i)
    = s_i x_i
    \stackrel{\eqref{rel:sx-commutation}}{=} x_{i+1} s_i - t_{i,i+1}
    = \phi \left( x_{i+1} s_i - t_{i,i+1} \right),
  \]
  where the last equality follows from the fact that $\xi$ fixes $t_{i,j}$ since it preserves the trace map (so $\xi(b^\vee) = \xi(b)^\vee$ for $b \in B$).
}

\begin{lem}
  Suppose $\tau \colon F \to F^\op$ is an isomorphism of Frobenius algebras.  Then
  \begin{gather*}
    \hat \tau \colon \cA_n(F) \to \cA_n(F)^\op,\\
    \hat \tau (x_i) = x_i,\ \hat \tau(\bsf) = \tau^{\otimes n}(\bsf),\ \hat \tau(\pi) = \pi^{-1},\quad
    1 \le i \le n,\ \bsf \in F^{\otimes n},\ \pi \in S_n,
  \end{gather*}
  is an isomorphism of algebras.
\end{lem}

\begin{proof}
  It is clear that $\hat \tau$ is a homomorphism when restricted to the subalgebras $\kk[x_1,\dotsc,x_n]$, $F^{\otimes n}$, and $\kk S_n$.  Using Lemma~\ref{lem:Frob-antiautom-properties}, it is straightforward to verify that it preserves relations \eqref{rel:xF-commutation}, \eqref{rel:sx-triv-commutation}, and \eqref{rel:SF-commutation}.
  \details{
    Suppose $1 \le i \le n$ and $\bsf \in F^{\otimes n}$.  Then
    \begin{gather*}
      \hat \tau (\bsf x_i)
      = \hat \tau (x_i) \hat \tau (\bsf)
      = x_i \tau^{\otimes n}(\bsf)
      = \psi_i^{-1} \left( \tau^{\otimes n}(\bsf) \right) x_i
      = \tau^{\otimes n} (\psi_i(\bsf)) x_i
      = \hat \tau (x_i \psi_i(\bsf)),
    \end{gather*}
    where we used Lemma~\ref{lem:Frob-antiautom-properties} in the fourth equality.  So $\hat \tau$ preserves \eqref{rel:xF-commutation}.

    Now suppose $1 \le i \le n-1$, $1 \le j \le n$, and $j \ne i,i+1$.  Then
    \[
      \hat \tau (s_i x_j)
      = \hat \tau(x_j) \hat \tau(s_i)
      = x_j s_i
      \stackrel{\eqref{rel:sx-triv-commutation}}{=} s_i x_j
      = \hat \tau(s_i) \hat \tau(x_j)
      = \hat \tau (x_j s_i),
    \]
    which proves that $\hat \tau$ preserves \eqref{rel:sx-triv-commutation}.

    If $\pi \in S_n$ and $\bsf \in F^{\otimes n}$, then
    \[
      \hat \tau (\pi \bsf)
      = \hat \tau (\bsf) \hat \tau(\pi)
      = \tau^{\otimes n}(\bsf) \pi^{-1}
      = \pi^{-1} \prescript{\pi}{}{(\tau^{\otimes n}(\bsf))}
      = \pi^{-1} (\tau^{\otimes n}(\prescript{\pi}{}{\bsf}))
      = \hat \tau (\prescript{\pi}{}{\bsf} \pi),
    \]
    which proves that $\hat \tau$ preserves \eqref{rel:SF-commutation}.
  }
  To verify that it preserves \eqref{rel:sx-commutation}, we compute
  \[
    \hat \tau (s_i x_i)
    = \hat \tau(x_i) \hat \tau(s_i)
    = x_i s_i
    \stackrel{\eqref{rel:sx-commutation2}}{=} s_i x_{i+1} - t_{i+1,i}
  \]
  and
  \[
    \hat \tau (x_{i+1} s_i - t_{i,i+1})
    = s_i x_{i+1} - \sum_{b \in B} \tau(b)_i \tau(b^\vee)_{i+1}
    = s_i x_{i+1} - \sum_{c \in B'} (-1)^{\bar c} c^\vee_i c_{i+1}
    = s_i x_{i+1} - t_{i+1,i},
  \]
  where, in the second equality, we introduced the basis $B' = \{\tau(b^\vee) \mid b \in B\}$ of $F$ and used Lemma~\ref{lem:Frob-antiautom-properties}.  It follows that $\hat \tau$ is a homomorphism.  Since it has inverse $\widehat{\tau^{-1}}$, it is an isomorphism.
\end{proof}

\begin{rem}
  In \cite[Lem.~3.11]{KM15}, which assumes that $\psi = \id$ (i.e.\ $F$ is a symmetric algebra) and $F$ is purely even, it is asserted that the map $\hat \tau$ is an isomorphism of algebras for any algebra isomorphism $\tau \colon F \to F^\op$, without the assumption that $\tau$ preserves the trace map.  However, this appears to be false as one can see by considering the example $F = \C[z]/(z^2)$, with trace map $\tr(a+bz) = b$, $a, b \in \C$, and $\tau$ determined by $\tau(z)=2z$.
  \details{
    We have $t_{i,i+1} = t_{i+1,i} = z_i + z_{i+1}$.  Hence
    \[
      \hat \tau(s_i x_i)
      = x_i s_i
      \stackrel{\eqref{rel:sx-commutation2}}{=} s_i x_{i+1} - z_i - z_{i+1},
    \]
    whereas
    \[
      \hat \tau(x_{i+1} s_i - t_{i,i+1})
      = s_i x_{i+1} - 2z_i - 2z_{i+1}.
    \]
    The problem in \cite{KM15} seems to stem from their use of $\varphi \circ \nu = \nu^* \circ \varphi$ in the first line (second equality) after ``Therefore'' in the proof of \cite[Lem.~3.11]{KM15} and their use of $(\varphi^{-1} \otimes \varphi^{-1}) \circ (\nu \otimes \nu)^* = (\nu \otimes \nu) \circ (\varphi^{-1} \otimes \varphi^{-1})$ in the fifth equality there.
  }
\end{rem}

%--------------------------------
\subsection{Additional relations\label{subsec:additional-relations}}
%--------------------------------

We now deduce some relations that will be useful in our computations to follow.  For $k \in \N_+$, $1 \le i,j \le n$, $i \ne j$, define
\begin{equation} \label{eq:t^k-def}
  t_{i,j}^{(k)} := \sum_{b \in B} b_i \frac{x_i^k - x_j^k}{x_i-x_j} b_j^\vee.
\end{equation}
(Note that the rational expression is actually a polynomial in $x_i$ and $x_j$.)  Hence $\left| t_{i,j}^{(k)} \right| = k \delta$, $\overline{t_{i,j}^{(k)}} = 0$, and $t_{i,j}^{(1)} = t_{i,j}$.  Note that
\begin{equation} \label{eq:t-symmetry-psi-id}
  t_{i,j}^{(k)} = t_{j,i}^{(k)}, \quad \text{if $\psi=\id$}.
\end{equation}
\details{
  We have
  \[
    t_{i,j}^{(k)}
    = \sum_{b \in B} b_i \frac{x_i^k - x_j^k}{x_i - x_j} b_j^\vee
    \stackrel{\eqref{eq:double-dual}}{=} \sum_{b \in B} (-1)^{\bar b} \left( b^\vee \right)^\vee_i \frac{x_i^k - x_j^k}{x_i - x_j} b_j^\vee
    = \sum_{b \in B} b_j^\vee \frac{x_j^k - x_i^k}{x_j - x_i} \left( b^\vee \right)^\vee_i
    = t_{j,i}^{(k)},
  \]
  where, in the last equality, we used the fact that the definition of $t_{j,i}^{(k)}$ is independent of the basis.
}

\begin{lem} \label{lem:sx-higher-commutation}
  For $k \in \N_+$ and $1 \le i \le n-1$, we have
  \begin{gather}
    s_i x_i^k
    = x_{i+1}^k s_i - t_{i,i+1}^{(k)}, \label{eq:sx^k-commutation1}
    \\
    s_i x_{i+1}^k
    = x_i^k s_i + t_{i+1,i}^{(k)}. \label{eq:sx^k-commutation2}
  \end{gather}
\end{lem}

\begin{proof}
  This follows from \eqref{rel:sx-commutation} and \eqref{rel:sx-commutation2} by a straightforward induction.
  \details{
    We prove the results by induction on $k$.  For $k=1$, \eqref{eq:sx^k-commutation1} is simply \eqref{rel:sx-commutation}.  Assuming \eqref{eq:sx^k-commutation1} holds for $k$, we compute
    \[
      s_i x_i^{k+1}
      = x_{i+1}^k s_i x_i - \sum_{\ell = 0}^{k-1} x_{i+1}^\ell t_{i,i+1} x_i^{k-\ell}
      \stackrel{\eqref{rel:sx-commutation}}{=} x_{i+1}^{k+1} s_i - \sum_{\ell = 0}^k x_{i+1}^\ell t_{i,i+1} x_i^{k-\ell}
      = x_{i+1}^{k+1} s_i - t_{i,i+1}^{(k+1)}.
    \]
    Similarly, for $k=1$, \eqref{eq:sx^k-commutation2} is simply \eqref{rel:sx-commutation2}.  Assuming \eqref{eq:sx^k-commutation2} holds for $k$, we compute
    \[
      s_i x_{i+1}^{k+1}
      = x_i^k s_i x_{i+1} + \sum_{\ell=0}^{k-1} x_i^\ell t_{i+1,i} x_{i+1}^{k-\ell}
      \stackrel{\eqref{rel:sx-commutation2}}{=} x_i^{k+1} s_i + \sum_{\ell=0}^k x_i^\ell t_{i+1,i} x_{i+1}^{k-\ell}
      = x_i^{k+1} s_i + t_{i+1,i}^{(k+1)}.
    \]
  }
\end{proof}

\begin{lem} \label{lem:Ft-commutation}
  For $1 \le i,j \le n$, $i \ne j$, $k \in \N_+$, and $\bsf \in F^{\otimes n}$, we have
  \begin{equation} \label{eq:Ft-commutation}
    \bsf t_{i,j}^{(k)}
    = t_{i,j}^{(k)} \psi_i^k \left( \prescript{s_{i,j}}{}{\bsf} \right)
    = t_{i,j}^{(k)} \prescript{s_{i,j}}{}{\left( \psi_j^k(\bsf) \right)},
  \end{equation}
  where $s_{i,j} \in S_n$ is the transposition of $i$ and $j$.
\end{lem}

\begin{proof}
  It suffices to prove that, for $1 \le i,j \le n$, $i \ne j$, $k \in \N_+$, and $f \in F$, we have
  \begin{equation} \label{eq:Ft-commutation-ne}
    f_\ell t_{i,j} = t_{i,j} f_\ell, \quad \ell \ne i,j,
  \end{equation}
  \begin{equation}
    f_i t_{i,j}^{(k)} = t_{i,j}^{(k)} f_j, \label{eq:ftij-commutation1} \\
  \end{equation}
  \begin{equation}
    f_j t_{i,j}^{(k)} = t_{i,j}^{(k)} \psi^k(f)_i, \label{eq:ftij-commutation2}
  \end{equation}
  Relation \eqref{eq:Ft-commutation-ne} is immediate.  For \eqref{eq:ftij-commutation1}, we compute
  \[
    f_i t_{i,j}^{(k)}
    = f_i \sum_{b \in B} b_i \frac{x_i^k - x_j^k}{x_i - x_j} b_j^\vee
    \stackrel{\eqref{eq:f-in-basis}}{=} \sum_{b,c \in B} c_i \tr(c^\vee f b) \frac{x_i^k - x_j^k}{x_i - x_j} b_j^\vee
    \stackrel{\eqref{eq:f-in-dual-basis}}{=} \sum_{c \in B} c_i \frac{x_i^k - x_j^k}{x_i - x_j} c_j^\vee f_j
    = t_{i,j}^{(k)} f_j.
  \]
  Recall that $F$ has a basis consisting of eigenvectors for $\psi$ (see Section~\ref{subsec:Frobenius-algebras}).  Thus, to prove \eqref{eq:ftij-commutation2}, it suffices to consider $f \in F$ satisfying $\psi(f) = \omega f$, where $\omega$ is a $\theta$-th root of unity.  Then
  \begin{multline*}
    f_j t_{i,j}^{(k)}
    = f_j \sum_{b \in B} b_i \frac{x_i^k - x_j^k}{x_i - x_j} b_j^\vee
    \stackrel{\eqref{eq:f-in-dual-basis}}{=} \sum_{b,c \in B} (-1)^{\bar b \bar f} b_i \frac{x_i^k - \omega^k x_j^k}{x_i - \omega x_j} \tr(f b^\vee c) c_j^\vee
    \\
    = \sum_{b,c \in B} (-1)^{\bar c \bar f} b_i \frac{x_i^k - \omega^k x_j^k}{x_i - \omega x_j} \omega \tr(b^\vee c f) c_j^\vee
    \stackrel{\eqref{eq:f-in-basis}}{=} \sum_{c \in B} (-1)^{\bar c \bar f} c_i f_i \frac{x_i^k - \omega^k x_j^k}{x_i - \omega x_j} \omega c_j^\vee
    \\
    = \sum_{c \in B} c_i \frac{\omega^k x_i^k - \omega^k x_j^k}{\omega x_i - \omega x_j} \omega c_j^\vee f_i
    = \sum_{c \in B} c_i \frac{x_i^k - x_j^k}{x_i - x_j} c_j^\vee \omega^k f_i
    = t_{i,j}^{(k)} \psi^k(f)_i. \qedhere
  \end{multline*}
\end{proof}

\begin{lem}
  For $1 \le i,j \le n$, $i \ne j$, we have
  \begin{equation} \label{eq:psi-tij-reverse}
    \psi_j(t_{i,j}) = t_{j,i}.
  \end{equation}
\end{lem}

\begin{proof}
  We have
  \[
    \psi_j^{-1}(t_{j,i})
    = \sum_{b \in B} \psi^{-1}(b)_j b_i^\vee
    = \sum_{b \in B} (-1)^{\bar b} b_i^\vee \psi^{-1}(b)_j
    \stackrel{\eqref{eq:double-dual}}{=} \sum_{b \in B} b^\vee_i (b^\vee)^\vee_j
    = t_{i,j},
  \]
  where we have used the fact that the parities of $b$ and $b^\vee$ are equal (since the trace map is even) in the second equality, and used the fact that the definition of $t_{i,j}$ is independent of the basis in the final equality (where we are summing over the basis $\{b^\vee \mid b \in B\}$).
\end{proof}

\begin{lem}
  For $1 \le i,j \le n$, $i \ne j$, $k \in \N_+$, $\pi \in S_n$, and $p \in \kk[x_1^\theta,\dotsc,x_n^\theta]$, we have
  \begin{equation} \label{eq:xt-commutation-far}
    x_\ell t_{i,j}^{(k)} = t_{i,j}^{(k)} x_\ell,\quad \ell \ne i,j,
  \end{equation}
  \begin{equation} \label{eq:xt-commutation-R}
    x_i t_{i,j} = t_{j,i} x_i,
  \end{equation}
  \begin{equation} \label{eq:pf-commutation}
    p t_{i,j} = t_{i,j} p,
  \end{equation}
  \begin{equation} \label{eq:tij-conjugation}
    \pi t_{i,j} = t_{\pi(i), \pi(j)} \pi,
  \end{equation}
  \begin{equation} \label{eq:tii+1-conjugation}
    s_i t_{i,i+1}^{(k)} = t_{i+1,i}^{(k)} s_i.
  \end{equation}
\end{lem}

\begin{proof}
  Relation~\eqref{eq:xt-commutation-far} follows immediately from \eqref{rel:xF-commutation}.  For relation \eqref{eq:xt-commutation-R}, we have
  \[
    x_i t_{i,j}
    \stackrel{\eqref{rel:xF-commutation}}{=} \psi_i^{-1}(t_{i,j}) x_i
    \stackrel{\eqref{eq:psi-tij-reverse}}{=} t_{j,i} x_i.
  \]
  Relation \eqref{eq:pf-commutation} follows immediately from \eqref{rel:xF-commutation} and the fact that $\theta$ is the order of the Nakayama automorphism $\psi$.  Relation~\eqref{eq:tij-conjugation} follows immediately from \eqref{rel:SF-commutation}.  Finally, \eqref{eq:tii+1-conjugation} follows from multiplying \eqref{eq:sx^k-commutation1} on the left by $s_i$, multiplying  \eqref{eq:sx^k-commutation2} on the right by $s_i$, and adding the resulting two equations.
\end{proof}

%%%%%%%%%%%%%%%%%%%%%%%%%%
%
\section{Structure theory\label{sec:structure-theory}}
%
%%%%%%%%%%%%%%%%%%%%%%%%%%

In this section we examine the structure theory of affine wreath product algebras.  In particular, we give an explicit basis, describe the center, define Jucys--Murphy elements, give a Mackey Theorem, and define intertwining elements.  In this section $\kk$ is an arbitrary commutative ring of characteristic not equal to $2$.

%-------------------------------------------------
\subsection{Deformed divided difference operators\label{subsec:divided-diff}}
%-------------------------------------------------

Let $P_n = \kk[x_1,\dotsc,x_n]$,\label{P_n-def} and let $P_n(F)$ be the graded superalgebra such that
\[
  P_n(F) = P_n \otimes F^{\otimes n}
\]
as a $\kk$-module, where the two factors are subalgebras, and where we impose the relations \eqref{rel:xF-commutation}.  Equivalently, $P_n(F)$ is the free product of $\kk$-algebras $P_n \star F^{\otimes n}$ modulo the relations \eqref{rel:xF-commutation}.  We also have a natural isomorphism of algebras
\begin{equation} \label{eq:P_n(F)-def}
  P_n(F) \simeq (\kk[x] \ltimes F)^{\otimes n},
\end{equation}
and we will often identify the two.  The parity and degrees of the $x_i$ are given by \eqref{eq:x-degrees}.  We define the \emph{polynomial degree} of an element of $P_n(F)$ to be its total degree as a polynomial in the $x_i$.

For $1 \le i \le n-1$, we define a skew derivation $\Delta_i \colon P_n(F) \to P_n(F)$ inductively as follows.  First, we define $\Delta_i(F^{\otimes n})=0$ and, on elements of $P_n$ of polynomial degree one, we define
\begin{equation} \label{eq:Delta_i-degree-one-def}
  \Delta_i(x_i) = t_{i,i+1},\quad
  \Delta_i(x_{i+1}) = - t_{i+1,i},\quad
  \Delta_i(x_j) = 0,\quad j \ne i,i+1,
\end{equation}
and extend $\kk$-linearly.  Then we extend $\Delta_i$ to all of $P_n(F)$ by requiring that
\[
  \Delta_i(a_1 a_2)
  = \Delta_i(a_1) a_2 + \prescript{s_i}{}{a_1} \Delta_i(a_2),\quad a_1,a_2 \in P_n(F).
\]
Note, in particular, that
\begin{equation} \label{eq:Deltai-F-skew-linear}
  \Delta_i (\bsf a)
  = \prescript{s_i}{}{\bsf} \Delta_i(a),\qquad
  \bsf \in F^{\otimes n},\ a \in P_n(F).
\end{equation}

\begin{lem} \label{lem:s_i-commutation-derivation}
  For all $a \in P_n(F)$ and $1 \le i \le n-1$, in $\cA_n(F)$ we have
  \[
    s_i a = \prescript{s_i}{}{a} s_i - \Delta_i(a).
  \]
\end{lem}

\begin{proof}
  The result for $a$ of polynomial degree zero and one follows immediately from \eqref{rel:sx-triv-commutation}--\eqref{rel:SF-commutation} and \eqref{rel:sx-commutation2}.  Suppose it holds for $a_1, a_2 \in P_n(F)$.  Then
  \[
    s_i(a_1 a_2)
    = \prescript{s_i}{}{a_1} s_i a_2 - \Delta_i(a_1) a_2
    = \prescript{s_i}{}{a_1} \prescript{s_i}{}{a_2} - \prescript{s_i}{}{a_1} \Delta_i(a_2) - \Delta_i(a_1) a_2
    = \prescript{s_i}{}{(a_1 a_2)} - \Delta_i(a_1a_2),
  \]
  and hence the result follows by induction.
\end{proof}

The operators $\Delta_i$ can be thought of as $F$-deformations of divided difference operators.  In particular, it follows from Lemmas~\ref{lem:s_i-commutation-derivation} and~\ref{lem:sx-higher-commutation} that, for $1 \le i \le n-1$ and $k \in \N_+$, we have
\begin{gather}
  \Delta_i \left( x_i^k \right)
  = t_{i,i+1}^{(k)}
  = \sum_{b \in B} b_i \frac{x_i^k - x_{i+1}^k}{x_i - x_{i+1}} b_{i+1}^\vee
  = \sum_{b \in B} b_i \partial_i \left( x_i^k \right) b_{i+1}^\vee, \label{eq:Delta-i}
  \\
  \Delta_i \left( x_{i+1}^k \right)
  = - t_{i+1,i}^{(k)}
  = \sum_{b \in B} b_{i+1} \frac{x_{i+1}^k-x_i^k}{x_i - x_{i+1}} b_i^\vee
  = \sum_{b \in B} b_{i+1} \partial_i \left( x_{i+1}^k \right) b_i^\vee, \label{eq:Delta-i+1}
\end{gather}
where
\begin{equation} \label{eq:usual-divided-diff}
  \partial_i(p) = \frac{p - \prescript{s_i}{}{p}}{x_i-x_{i+1}},\quad p \in P_n,
\end{equation}
is the usual divided difference operator.  In addition, we have the following result.

\begin{prop}
  We have
  \begin{align}
    \Delta_i \left( \prescript{s_j}{}{a} \right) &= \prescript{s_j}{}{\Delta_i(a)}, & 1 \le i, j \le n-1,\ |i-j| > 1,\ a \in P_n(F), \label{eq:Delta-s-triv-commutation} \\
    \Delta_i \left( \prescript{s_i}{}{a} \right) &= - \prescript{s_i}{}{\Delta_i(a)}, & 1 \le i \le n-1,\ a \in P_n(F), \label{eq:Delta-s-commutation} \\
    \Delta_i \Delta_j &= \Delta_j \Delta_i, & 1 \le i,j \le n-1,\ |i-j| > 1, \label{eq:Delta-Delta-triv-commutation} \\
    \Delta_i^2 &= 0, & 1 \le i \le n-1. \label{eq:Delta-square-zero}
  \end{align}
\end{prop}

\begin{proof}
  We prove \eqref{eq:Delta-s-triv-commutation} and \eqref{eq:Delta-s-commutation} by induction on the polynomial degree of $a$.  The results for $a$ of polynomial degree less than or equal to one follow immediately from the definition of $\Delta_i$.  Assume the results hold for elements of polynomial degree less than or equal to $k$, and let $a_1, a_2 \in P_n(F)$ have polynomial degree less than or equal to $k$.  Then, if $|i-j|>1$,
  \begin{multline*}
    \prescript{s_j}{}{\Delta_i(a_1 a_2)}
    = \prescript{s_j}{}{\big( \Delta_i(a_1) a_2 + \prescript{s_i}{}{a_1} \Delta_i(a_2) \big)}
    = \prescript{s_j}{}{\Delta_i(a_1)} \prescript{s_j}{}{a_2} + \prescript{s_j s_i}{}{a_1} \prescript{s_j}{}{\Delta_i(a_2)}
    \\
    = \Delta_i \left( \prescript{s_j}{}{a_1} \right) \prescript{s_j}{}{a_2} + \prescript{s_i s_j}{}{a_1} \Delta_i \left(  \prescript{s_j}{}{a_2} \right)
    = \Delta_i \left( \prescript{s_j}{}{(a_1 a_2)} \right),
  \end{multline*}
  completing the inductive step for \eqref{eq:Delta-s-triv-commutation}.  The proof of the inductive step for \eqref{eq:Delta-s-commutation} is similar.
  \details{
    We have
    \begin{multline*}
      \prescript{s_i}{}{\Delta_i (a_1 a_2)}
      = \prescript{s_i}{}{\big( \Delta_i(a_1) a_2 + \prescript{s_i}{}{a_1} \Delta_i(a_2) \big)}
      = \prescript{s_i}{}{\Delta_i(a_1)} \prescript{s_i}{}{a_2} + a_1 \prescript{s_i}{}{\Delta_i(a_2)}
      \\
      = - \Delta_i \left( \prescript{s_i}{}{a_1} \right) \prescript{s_i}{}{a_2} - a_1 \Delta_i \left( \prescript{s_i}{}{a_2} \right)
      = - \Delta_i \left( \prescript{s_i}{}{a_1} \prescript{s_i}{}{a_2} \right)
      = - \Delta_i \left( \prescript{s_i}{}{(a_1 a_2)} \right).
    \end{multline*}
  }

  We prove \eqref{eq:Delta-Delta-triv-commutation} again by induction on the polynomial degree of $a$.  Suppose $|i-j|>1$.  For $a$ of polynomial degree less than or equal to one, we have $\Delta_i \Delta_j(a) = 0 = \Delta_j \Delta_i(a)$.  Assume the result holds for elements of polynomial degree less than or equal to $k$, and let $a_1, a_2 \in P_n(F)$ have polynomial degree less than or equal to $k$.  Then
  \begin{align*}
    \Delta_i \Delta_j (a_1 a_2)
    &= \Delta_i \left( \Delta_j(a_1) a_2 + \prescript{s_j}{}{a_1} \Delta_j(a_2) \right)
    \\
    &= \Delta_i (\Delta_j(a_1)) a_2 + \prescript{s_i}{}{\Delta_j(a_1)} \Delta_i(a_2) + \Delta_i \left( \prescript{s_j}{}{a_1} \right) \Delta_j(a_2) + \prescript{s_i s_j}{}{a_1} \Delta_i \Delta_j(a_2)
    \\
    &\stackrel{\mathclap{\eqref{eq:Delta-s-triv-commutation}}}{=}\ \Delta_j (\Delta_i(a_1)) a_2 + \Delta_j \left( \prescript{s_i}{}{a_1} \right) \Delta_i(a_2) + \prescript{s_j}{}{\Delta_i (a_1)} \Delta_j(a_2) + \prescript{s_j s_i}{}{a_1} \Delta_j \Delta_i(a_2)
    \\
    &= \Delta_j \left( \Delta_i(a_1) a_2 + \prescript{s_i}{}{a_1} \Delta_i(a_2) \right)
    \\
    &= \Delta_j \Delta_i (a_1 a_2).
  \end{align*}

  To prove \eqref{eq:Delta-square-zero}, note that $\Delta_i^2(a) = 0$ for $a$ of polynomial degree less than or equal to one.  Assume the result holds for elements of polynomial degree less than or equal to $k$, and let $a_1, a_2 \in P_n(F)$ have polynomial degree less than or equal to $k$.  Then
  \begin{multline*}
    \Delta_i^2(a_1 a_2)
    = \Delta_i \big( \Delta_i(a_1) a_2 + \prescript{s_i}{}{a_1} \Delta_i(a_2) \big)
    \\
    = \Delta_i^2 (a_1) a_2 + \prescript{s_i}{}{\Delta_i(a_1)} \Delta_i(a_2) + \Delta_i \left( \prescript{s_i}{}{a_1} \right) \Delta_i(a_2) + a_1 \Delta_i^2(a_2)
    \stackrel{\eqref{eq:Delta-s-commutation}}{=} 0. \qedhere
  \end{multline*}
\end{proof}

\begin{lem}
  If $\psi = \id$ (i.e.\ the Frobenius algebra $F$ is symmetric), then
  \[
    \Delta_i(\bsf p)
    = \prescript{s_i}{}{\bsf}\, t_{i,i+1} \partial_i(p), \quad
    \bsf \in F^{\otimes n},\ p \in P_n,
  \]
  where $\partial_i$ is the usual divided difference operator of \eqref{eq:usual-divided-diff}.
\end{lem}

\begin{proof}
  By \eqref{eq:Deltai-F-skew-linear}, it suffices to prove the case where $\bsf = 1$.  Since $\psi = \id$, the case where $p$ is of polynomial degree one follows from \eqref{eq:t-symmetry-psi-id} and \eqref{eq:Delta_i-degree-one-def}. The general result then follows by a straightforward induction.
  \details{
    Assuming the result holds for $p_1$ and $p_2$, we have
    \[
      \Delta_i(p_1p_2)
      = \Delta_i(p_1) p_2 + \prescript{s_i}{}{p_1} \Delta_i(p_2)
      = t_{i,i+1} \left( \frac{p_1 - \prescript{s_i}{}{p_1}}{x_i-x_{i+1}} p_2 + \prescript{s_i}{}{p_1} \frac{p_2 - \prescript{s_i}{}{p_2}}{x_i - x_{i+1}} \right)
      = t_{i,i+1} \partial_i(p_1 p_2).
    \]
  }
\end{proof}

%-------------------------
\subsection{Basis theorem}
%-------------------------

We have a natural action of $S_n$ on $P_n(F)$ by superpermutation of the $x_i$ and the factors of $F^{\otimes n}$.  Recall that we also have a natural algebra homomorphism $P_n(F) \to \cA_n(F)$ and we use this to view elements of $P_n(F)$ as elements of $\cA_n(F)$.  (We will see in Theorem~\ref{theo:AnF-basis} that this homomorphism is injective.)  Let $\le$ denote the strong Bruhat ordering on $S_n$.

\begin{lem} \label{lem:Sn-commutation-mod-lower-terms}
  For $\pi \in S_n$ and $a \in P_n(F)$, we have
  \[
    \pi a = \prescript{\pi}{}{a} \pi + \sum_{\sigma < \pi} a_\sigma \sigma
    \quad \text{and} \quad
    a \pi = \pi \left( \prescript{\pi^{-1}}{}{a} \right) + \sum_{\sigma < \pi} \sigma a_\sigma',
  \]
  for some $a_{\sigma}, a_{\sigma}' \in P_n(F)$ of polynomial degree less than the polynomial degree of $a$.
\end{lem}

\begin{proof}
  This follows from the defining relations (or Lemma~\ref{lem:s_i-commutation-derivation}).
\end{proof}

\begin{prop} \label{prop:dHn-action-on-V}
  Let $V$ denote $P_n(F) \otimes \kk S_n$, considered as a graded $\kk$-supermodule.  Then $V$ is an $\cA_n(F)$-module under the action
  \[
    z \cdot (a \otimes w)
    = za \otimes w,
    \qquad
    s_i \cdot (a \otimes w)
    = \prescript{s_i}{}{a} \otimes s_i w - \Delta_i(a) \otimes w,
  \]
  for all $z,a \in P_n(F)$, $w \in \kk S_n$, and $1 \le i \le n-1$.
\end{prop}

\begin{proof}
  It is clear that the action satisfies the defining relations of $P_n(F)$.  This includes relation \eqref{rel:xF-commutation}.  For $1 \le i \le n-1$ and $\bsf \in F^{\otimes n}$, we have
  \begin{multline*}
    s_i \cdot (\bsf \cdot (a \otimes w))
    = s_i \cdot (\bsf a \otimes w)
    = \prescript{s_i}{}{(\bsf a)} \otimes s_i w - \Delta_i(\bsf a) \otimes w
    \\
    \stackrel{\eqref{eq:Deltai-F-skew-linear}}{=} \prescript{s_i}{}{\bsf} \prescript{s_i}{}{a} \otimes s_i w - \prescript{s_i}{}{\bsf} \Delta_i(a) \otimes w
    = \prescript{s_i}{}{\bsf} \cdot (s_i \cdot (a \otimes w)).
  \end{multline*}
  Thus, the action satisfies relation \eqref{rel:SF-commutation}.

  For $1 \le i \le n-1$, $1 \le j \le n$, $j \ne i,i+1$, we have
  \begin{multline*}
    s_i \cdot (x_j \cdot (a \otimes w))
    = s_i (x_j a \otimes w)
    = \prescript{s_i}{}{(x_j a)} \otimes s_i w - \Delta_i(x_j a) \otimes w
    \\
    = x_j \prescript{s_i}{}{a} \otimes s_i w - x_j \Delta_i(a) \otimes w
    = x_j \cdot (s_i \cdot (a \otimes w)).
  \end{multline*}
  Thus the action satisfies \eqref{rel:sx-triv-commutation}.  We also have
  \begin{multline*}
    s_i \cdot (x_i \cdot (a \otimes w))
    = s_i \cdot (x_i a \otimes w)
    = \prescript{s_i}{}{(x_i a)} \otimes s_i w - \Delta_i(x_i a) \otimes w
    \\
    = x_{i+1} \prescript{s_i}{}{a} \otimes s_i w - \big( t_{i,i+1} a + x_{i+1} \Delta_i(a) \big) \otimes w
    = (x_{i+1} s_i - t_{i,i+1}) \cdot (a \otimes w),
  \end{multline*}
  and so the action satisfies \eqref{rel:sx-commutation}.

  It remains to verify the Coxeter relations of $\kk S_n$.  For $1 \le i \le n-1$, we have
  \begin{multline*}
    s_i \cdot (s_i \cdot (a \otimes w))
    = s_i \cdot \left( \prescript{s_i}{}{a} \otimes s_i w - \Delta_i(a) \otimes w \right)
    \\
    = a \otimes w - \Delta_i( \prescript{s_i}{}{a} ) \otimes s_i w - \prescript{s_i}{}{\Delta_i(a)} \otimes s_i w + \Delta_i^2(a) \otimes w
    = a \otimes w,
  \end{multline*}
  where the final equality holds by \eqref{eq:Delta-s-commutation} and \eqref{eq:Delta-square-zero}.  For $1 \le i,j \le n-1$ with $|i-j| > 1$, we have
  \begin{multline*}
    s_i \cdot (s_j \cdot (a \otimes w))
    = s_i \cdot \left( \prescript{s_j}{}{a} \otimes s_i w - \Delta_j(a) \otimes w \right)
    \\
    = \prescript{s_i s_j}{}{a} \otimes s_i s_j w - \Delta_i( \prescript{s_j}{}{a} ) \otimes s_j w - \prescript{s_i}{}{\Delta_j(a)} \otimes s_i w + \Delta_i \Delta_j(a) \otimes w
    = s_j \cdot (s_i \cdot (a \otimes w)),
  \end{multline*}
  where the last equality uses \eqref{eq:Delta-s-triv-commutation}, \eqref{eq:Delta-Delta-triv-commutation}, and the fact that $s_i s_j = s_j s_i$.

  Finally, suppose $1 \le i \le n-2$.  We claim that, as operators on $V$,
  \begin{equation} \label{eq:braid-x-relation}
    (s_i s_{i+1} s_i - s_{i+1} s_i s_{i+1}) x_j
    = x_{s_i s_{i+1} s_i(j)} (s_i s_{i+1} s_i - s_{i+1} s_i s_{i+1}),\quad
    1 \le j \le n.
  \end{equation}
  Indeed, by what has already been shown above, we have
  \begin{multline*}
    s_i s_{i+1} s_i x_i
    \stackrel{\eqref{rel:sx-commutation}}{=} s_i s_{i+1} x_{i+1} s_i - s_i s_{i+1} t_{i,i+1}
    \stackrel{\eqref{rel:sx-commutation}}{=} s_i x_{i+2} s_{i+1} s_i - s_i t_{i+1,i+2} s_i - s_i s_{i+1} t_{i,i+1}
    \\
    \stackrel{\substack{\eqref{rel:sx-triv-commutation} \\ \eqref{eq:tij-conjugation}}}{=} x_{i+2} s_i s_{i+1} s_i - t_{i,i+2} - s_i s_{i+1} t_{i,i+1}
  \end{multline*}
  and
  \begin{multline*}
    s_{i+1} s_i s_{i+1} x_i
    \stackrel{\eqref{rel:sx-triv-commutation}}{=} s_{i+1} s_i x_i s_{i+1}
    \stackrel{\eqref{rel:sx-commutation}}{=} s_{i+1} x_{i+1} s_i s_{i+1} - s_{i+1} t_{i,i+1} s_{i+1}
    \\
    \stackrel{\substack{\eqref{rel:sx-commutation} \\ \eqref{eq:tij-conjugation}}}{=} x_{i+2} s_{i+1} s_i s_{i+1} - t_{i+1,i+2} s_i s_{i+1} - t_{i,i+2}
    \stackrel{\eqref{eq:tij-conjugation}}{=} x_{i+2} s_{i+1} s_i s_{i+1} - s_i s_{i+1} t_{i,i+1} - t_{i,i+2}.
  \end{multline*}
  This proves \eqref{eq:braid-x-relation} when $j=i$.  The other cases are proved by similar direct computations.
  \details{
    For $j=i+1$, we have
    \begin{multline*}
      s_{i+1} s_i s_{i+1} x_{i+1}
      \stackrel{\eqref{rel:sx-commutation}}{=} s_{i+1} s_i x_{i+2} s_{i+1} - s_{i+1} s_i t_{i+1,i+2}
      \stackrel{\eqref{rel:sx-triv-commutation}}{=} s_{i+1} x_{i+2} s_i s_{i+1} - s_{i+1} s_i t_{i+1,i+2}
      \\
      \stackrel{\eqref{rel:sx-commutation2}}{=} x_{i+1} s_{i+1} s_i s_{i+1} + t_{i+2,i+1} s_i s_{i+1} - s_{i+1} s_i t_{i+1,i+2}
    \end{multline*}
    and
    \begin{multline*}
      s_i s_{i+1} s_i x_{i+1}
      \stackrel{\eqref{rel:sx-commutation2}}{=} s_i s_{i+1} x_i s_i + s_i s_{i+1} t_{i+1,i}
      \stackrel{\eqref{rel:sx-triv-commutation}}{=} s_i x_i s_{i+1} s_i + s_i s_{i+1} t_{i+1,i}
      \\
      \stackrel{\eqref{rel:sx-commutation}}{=} x_{i+1} s_i s_{i+1} s_i - t_{i,i+1} s_{i+1} s_i + s_i s_{i+1} t_{i+1,i}
      \stackrel{\eqref{eq:tij-conjugation}}{=} x_{i+1} s_i s_{i+1} s_i - s_{i+1} s_i t_{i+1,i+2} + t_{i+2,i+1} s_i s_{i+1}.
    \end{multline*}
    Hence
    \[
      (s_i s_{i+1} s_i - s_{i+1} s_i s_{i+1}) x_{i+1}
    = x_{i+1} (s_i s_{i+1} s_i - s_{i+1} s_i s_{i+1})
    \]
    This proves \eqref{eq:braid-x-relation} when $j=i+1$.

    For $j=i+2$, we have
    \begin{multline*}
      s_{i+1} s_i s_{i+1} x_{i+2}
      \stackrel{\eqref{rel:sx-commutation2}}{=} s_{i+1} s_i x_{i+1} s_{i+1} + s_{i+1} s_i t_{i+2,i+1}
      \\
      \stackrel{\eqref{rel:sx-commutation2}}{=} s_{i+1} x_i s_i s_{i+1} + s_{i+1} t_{i+1,i} s_{i+1} + s_{i+1} s_i t_{i+2,i+1}
      \\
      \stackrel{\eqref{rel:sx-triv-commutation}}{=} x_i s_{i+1} s_i s_{i+1} + s_{i+1} t_{i+1,i} s_{i+1} + s_{i+1} s_i t_{i+2,i+1}
      \stackrel{\eqref{eq:tij-conjugation}}{=} x_i s_{i+1} s_i s_{i+1} + t_{i+2,i} + s_{i+1} s_i t_{i+2,i+1}
    \end{multline*}
    and
    \begin{multline*}
      s_i s_{i+1} s_i x_{i+2}
      \stackrel{\eqref{rel:sx-triv-commutation}}{=} s_i s_{i+1} x_{i+2} s_i
      \stackrel{\eqref{rel:sx-commutation2}}{=} s_i x_{i+1} s_{i+1} s_i + s_i t_{i+2,i+1} s_i
      \\
      \stackrel{\eqref{rel:sx-commutation2}}{=} x_i s_i s_{i+1} s_i + t_{i+1,i} s_{i+1} s_i + s_i t_{i+2,i+1} s_i
      \stackrel{\eqref{eq:tij-conjugation}}{=} x_i s_i s_{i+1} s_i + s_{i+1} s_i t_{i+2,i+1} + t_{i+2,i}
    \end{multline*}
    This proves \eqref{eq:braid-x-relation} when $j=i+2$.

    The remaining cases follow immediately from \eqref{rel:sx-triv-commutation}.
  }

  We now prove that $s_i s_{i+1} s_i = s_{i+1} s_i s_{i+1}$ as operators on $V$ by induction on the polynomial degree of $a$ in $a \otimes w \in V$.  The case where $a$ has polynomial degree zero is immediate.  If the claim holds for $a$, then, for $1 \le i \le n$, we have
  \begin{multline*}
    (s_i s_{i+1} s_i - s_{i+1} s_i s_{i+1})(x_i a \otimes w)
    = \big( (s_i s_{i+1} s_i - s_{i+1} s_i s_{i+1}) x_i \big) \cdot (a \otimes w)
    \\
    \stackrel{\eqref{eq:braid-x-relation}}{=} \big( x_{s_i s_{i+1} s_i(j)} (s_i s_{i+1} s_i - s_{i+1} s_i s_{i+1}) \big) \cdot (a \otimes w)
    = 0,
  \end{multline*}
  proving the inductive step.
\end{proof}

For $\alpha = (\alpha_1,\dotsc,\alpha_n) \in \N^n$, we let
\begin{equation} \label{eq:x^alpha-def}
  x^\alpha = x_1^{\alpha_1} x_2^{\alpha_2} \dotsm x_n^{\alpha_n}.
\end{equation}
We define the \emph{graded dimension} $\grdim V$ of a $(\Z \times \Z_2)$-graded $\kk$-module, with finite-dimensional graded pieces, to be
\[
  \grdim V = \sum_{i \in \Z,\ \varepsilon \in \Z_2} (\dim V_{i,\varepsilon}) q^i \varpi^\varepsilon \in \N[q^{\pm 1}, \varpi]/(\varpi^2-1).
\]

\begin{theo}[Basis theorem for $\cA_n(F)$] \phantomsection \label{theo:AnF-basis}
  \begin{enumerate}
    \item \label{theo-item:AnF-basis} The map
      \[
        V = P_n(F) \otimes \kk S_n \to \cA_n(F),\quad
        a \otimes w \mapsto aw,
      \]
      is an isomorphism of graded $\cA_n(F)$-supermodules.

    \item \label{theo-item:dHn-dim} The algebra $\cA_n(F)$ is free as a $\kk$-module, with graded dimension
  \[
    \grdim \cA_n(F) = n! \left( \frac{\grdim F}{1-q^\delta} \right)^n.
  \]
  \end{enumerate}
\end{theo}

\begin{proof}
  Let
  \begin{gather*}
    \cB_1 = \{x^\alpha \bb \otimes \pi \mid \alpha \in \N^n,\ \bb \in B^{\otimes n},\ \pi \in S_n\} \subseteq V, \\
    \cB_2 = \{x^\alpha \bb \pi \mid \alpha \in \N^n,\ \bb \in B^{\otimes n},\ \pi \in S_n\} \subseteq \cA_n(F).
  \end{gather*}
  Thus $\cB_1$ is a basis of $V$.  It follows easily from Lemma~\ref{lem:Sn-commutation-mod-lower-terms} that the elements of $\cB_2$ span $\cA_n(F)$.  Furthermore, we have $x^\alpha \bb \pi \cdot (1 \otimes 1) = x^\alpha \bb \otimes \pi$, and so the elements of $\cB_2$ are linearly independent, and hence $\cB_2$ is a basis for $\cA_n(F)$.

  Since $V$ is a cyclic module generated by $1 \otimes 1$, there is an $\cA_n(F)$-module homomorphism $\cA_n(F) \to V$ determined by $1 \mapsto 1 \otimes 1$.  This map sends $x^\alpha \bb \pi \in \cB_2$ to $x^\alpha \bb \otimes \pi \in \cB_1$.  Since the map is a bijection on $\kk$-bases, it is an isomorphism.  This proves \ref{theo-item:AnF-basis}.  Part~\ref{theo-item:dHn-dim} follows from \ref{theo-item:AnF-basis}.
\end{proof}

\begin{eg}
  Specializing $F$ in Theorem~\ref{theo:AnF-basis} recovers several results that have appeared in the literature:
  \begin{enumerate}
    \item When $F=\kk$ (see Example~\ref{eg:daHa}) we recover a basis for the degenerate affine Hecke algebra.  See \cite[Th.~3.2.2]{Kle05}.

    \item When $F$ is the group algebra of a finite group (see Example~\ref{eg:wreath-Hecke}), we recover \cite[Th.~2.8]{WW08}.

    \item When $F=\Cl$ (see Example~\ref{eg:affine-Sergeev}) we recover a basis for the affine Sergeev algebra.  See \cite[Th.~14.2.2]{Kle05}.

    \item When $F$ is symmetric (i.e.\ $\psi=\id$) and purely even, we recover \cite[Th.~3.8]{KM15}.  In fact, our proof closely follows the proof of \cite[Th.~3.8]{KM15}.
  \end{enumerate}
\end{eg}

\begin{cor}
  The sets
  \[
    \{x^\alpha \bb \pi \mid \alpha \in \N^n,\ \bb \in B^{\otimes n},\ \pi \in S_n\}
    \quad \text{and} \quad
    \{\pi x^\alpha \bb \mid  \alpha \in \N^n,\ \bb \in B^{\otimes n},\ \pi \in S_n\},
  \]
  are $\kk$-bases for $\cA_n(F)$.
\end{cor}

\begin{proof}
  It was shown in the proof of Theorem~\ref{theo:AnF-basis}\ref{theo-item:AnF-basis} that the first set is a basis.  The fact that the second set is also a basis then follows from Lemma~\ref{lem:Sn-commutation-mod-lower-terms} using induction on the length of $\pi \in S_n$.
\end{proof}

Theorem~\ref{theo:AnF-basis} allows us to view $\cA_n(F)$ as an affine version of the wreath product algebra.

\begin{cor} \label{cor:affine-wreath-presentation}
  As $(\Z \times \Z_2)$-graded $\kk$-modules, we have
  \[
    \cA_n(F) = \kk[x_1,\dotsc,x_n] \otimes \left( F^{\otimes n} \rtimes_\rho S_n \right).
  \]
  The multiplication is determined by the fact that the factors $\kk[x_1,\dotsc,x_n]$ and $F^{\otimes n} \rtimes_\rho S_n$ are subalgebras and the relations \eqref{rel:xF-commutation}--\eqref{rel:sx-commutation}.
\end{cor}

By Theorem~\ref{theo:AnF-basis} and Corollary~\ref{cor:affine-wreath-presentation}, we can identify $\kk[x_1,\dotsc,x_n]$, $F^{\otimes n}$, $\kk S_n$, $F^{\otimes n} \rtimes_\rho S_n$, and $P_n(F)$ as subalgebras of $\cA_n(F)$.  We will do so in the remainder of the paper.

%----------------------------------------------------------
\subsection{A filtration and the associated graded algebra\label{subsec:filtration}}
%----------------------------------------------------------

By Theorem~\ref{theo:AnF-basis}, we can extend the notion of \emph{polynomial degree} to all of $\cA_n(F)$ in the natural way, and we obtain a filtration on $\cA_n(F)$.  Note that all the relations of Definition~\ref{def:affine-wreath} are homogeneous except for \eqref{rel:sx-commutation}.  It follows that the associated graded algebra is
\[
  \gr \cA_n(F) = (\kk[x] \ltimes F)^{\otimes n} \rtimes_\rho S_n,
\]
where we recall that the subscript $\rho$ indicates that the action of $S_n$ is by superpermutation of the factors.

%-------------------------------------
\subsection{Description of the center}
%-------------------------------------

Let
\begin{equation} \label{eq:F_psi-def}
  F_\psi := \{f \in F \mid \psi(f) = f\}
\end{equation}
be the subalgebra of $F$ consisting of those elements fixed by the Nakayama automorphism.  Then we have
\[
  P_n(F_\psi) := \kk[x_1,\dotsc,x_n] \otimes F_\psi^{\otimes n} \subseteq P_n(F)
\]
where the tensor product in the center expression is of graded superalgebras.

For $k \in \Z$, define
\begin{equation} \label{eq:F^(k)-def}
  F^{(k)} := \{f \in F \mid g f = (-1)^{\bar f \bar g} f \psi^k(g) \text{ for all } g \in F\}
  \quad \text{and} \quad
  F_\psi^{(k)} := F^{(k)} \cap F_\psi.
\end{equation}
Note that, for all $k \in \Z$, $F^{(k \theta)} = Z(F)$ is the center of $F$.  In particular, if $\psi = \id$ (i.e.\ if $F$ is symmetric), then $F^{(k)} = Z(F)$ for all $k \in \Z$.

It is clear that
\[
  F^{(k)} F^{(\ell)} \subseteq F^{(k+\ell)}
  \quad \text{and} \quad
  F_\psi^{(k)} F_\psi^{(\ell)} \subseteq F_\psi^{(k+\ell)},\quad k,\ell \in \Z.
\]

\begin{lem} \label{lem:center-kxF}
  The center of $\kk[x] \ltimes F$ is
  \[
    Z(\kk[x] \ltimes F) = \bigoplus_{k=0}^\infty x^k F_\psi^{(-k)}.
  \]
\end{lem}

\begin{proof}
  It is clear from the definitions that $\bigoplus_{k=0}^\infty x^k F_\psi^{(-k)} \subseteq Z(\kk[x] \ltimes F)$.  Now let
  \[
    z = \sum_{k \in \N,\ b \in B} a_{k,b} x^k b \in Z(\kk[x] \ltimes F),
  \]
  where $a_{k,b} \in \kk$ for all $k, b$.  Without loss of generality, we may assume that $B = \{b_1,\dotsc,b_m\}$ is a basis of $F$ consisting of eigenvectors for $\psi$ (see Section~\ref{subsec:Frobenius-algebras}).  Let $i \in \{1,2,\dotsc,n\}$.  Then we have
  \begin{equation} \label{eq:center-kxF-computation}
    0
    = z x - x z
    = x \sum_{k \in \N,\ b \in B} a_{k,b} x^k (\psi - \id)(b).
  \end{equation}
  It follows that $a_{k,b} = 0$ for all $b \in B$ such that $\psi(b) \ne b$.  Hence $z \in \kk[x] \otimes F_\psi$.

  Thus, we can write
  \[
    z = \sum_{k \in \N} x^k f_k,
  \]
  where $f_k \in F_\psi$ for all $k \in \N$.  Without loss of generality, we may assume that $z$ is homogeneous in the $\Z_2$-grading, so that $\bar f_k = \bar z$ for all $k$.  Then, for $g \in F$ homogeneous in the $\Z_2$-grading, we then have
  \[
    0
    = gz - (-1)^{\bar z \bar g} zg
    = \sum_{k \in \N} x^k \left( \psi^k(g) f_k - (-1)^{\bar z \bar g} f_k g \right).
  \]
  It follows from Theorem~\ref{theo:AnF-basis} that $f_k \in F^{(-k)}$.
\end{proof}

\begin{eg}
  \begin{enumerate}
    \item When $F = \kk$, we have $Z(\kk[x] \ltimes F) = \kk[x]$.

    \item If $F$ is the Taft Hopf algebra of Example~\ref{eg:Taft} and $m \in \N_+$, we have $y^{m-1} \in F^{(1-m)}_\psi$. So $x^{m-1} y^{m-1} \in Z(\kk[x] \ltimes F)$.
  \end{enumerate}
\end{eg}

\begin{lem} \label{lem:center-in-PnF}
  The centralizer of $\kk[x_1,\dotsc,x_n]$ in $\cA_n(F)$ is contained in the subalgebra $P_n(F)$.
\end{lem}

\begin{proof}
  Let $z = \sum_{\pi \in S_n} z_\pi \pi$ be an element of the centralizer of $\kk[x_1,\dotsc,x_n]$ in $\cA_n(F)$, where $z_\pi \in P_n(F)$ for all $\pi \in S_n$.  Let $\pi \in S_n$ be maximal with respect to the strong Bruhat order such that $z_\pi \ne 0$.  Assume $\pi \ne 1$ and choose $i \in \{1,2,\dotsc,n\}$ such that $\pi(i) \ne i$.  Then, by Lemma~\ref{lem:Sn-commutation-mod-lower-terms}, we have
  \[
    x_i^\theta z - z x_i^\theta
    = \left( x_i^\theta - x_{\pi(i)}^\theta \right) z_\pi \pi + \sum_{\sigma : \sigma \not \ge \pi} z_\sigma' \sigma,
  \]
  for some $z_\sigma' \in P_n(F)$.  Thus, by Theorem~\ref{theo:AnF-basis}, $z$ is not central, giving a contradiction.  Hence the center of $\cA_n(F)$ is contained in $P_n(F)$.
\end{proof}

For $\alpha = (\alpha_1,\dotsc,\alpha_n) \in \Z^n$, let
\begin{equation} \label{eq:bF^alpha-def}
  \bF^{(\alpha)} := F^{(\alpha_1)} \otimes \dotsb \otimes F^{(\alpha_n)}
  \quad \text{and} \quad
  \bF_\psi^{(\alpha)} := F_\psi^{(\alpha_1)} \otimes \dotsb \otimes F_\psi^{(\alpha_n)}.
\end{equation}
It follows that
\[
  \bigoplus_{\alpha \in \N^n} x^\alpha \bF^{(-\alpha)}
  \quad \text{and} \quad
  \bigoplus_{\alpha \in \N^n} x^\alpha \bF_\psi^{(-\alpha)}
\]
are subalgebras of $\cA_n(F)$.  Note that $x^\alpha \bF^{(-\alpha)} = \big( x^{\alpha_1} F^{(-\alpha_1)} \big) \otimes \dotsb \otimes \big( x^{\alpha_1} F^{(-\alpha_n)} \big)$.

\begin{lem} \label{lem:centralizer-xF}
  The centralizer of $P_n(F)$ in $\cA_n(F)$ is the subalgebra $\bigoplus_{\alpha \in \N^n} x^\alpha \bF_\psi^{(-\alpha)}$.  In particular, the center of $\cA_n(F)$ is contained in the subalgebra $\bigoplus_{\alpha \in \N^n} x^\alpha \bF_\psi^{(-\alpha)}$.
\end{lem}

\begin{proof}
  This follows immediately from Lemmas~\ref{lem:center-kxF} and~\ref{lem:center-in-PnF}, together with the fact that \break $Z(A_1 \otimes A_2) = Z(A_1) \otimes Z(A_2)$ for $\kk$-algebras $A_1$ and $A_2$.
\end{proof}

\begin{theo} \label{theo:center}
  The center of $\cA_n(F)$ consists of those elements of $\bigoplus_{\alpha \in \N^n} x^\alpha \bF_\psi^{(-\alpha)}\subseteq P_n(F_\psi) = (\kk[x] \otimes F_\psi)^{\otimes n}$ that are invariant under the action of $S_n$ by superpermuting the factors of $(\kk[x] \otimes F_\psi)^{\otimes n}$.  In other words, the center of $\cA_n(F)$ consists of finite sums of the form
  \begin{equation} \label{eq:center-form}
    \sum_{\alpha \in \N^n} x^\alpha \bsf_\alpha,
    \quad \bsf_\alpha \in \bF_\psi^{(-\alpha)},
  \end{equation}
  such that $\bsf_{\pi \cdot \alpha} = \prescript{\pi}{}{\bsf_\alpha}$ for all $\alpha \in \N^n$ and $\pi \in S_n$.
\end{theo}

\begin{proof}
  Let $z$ be a central element of $\cA_n(F)$.  By Lemma~\ref{lem:centralizer-xF}, we have
  \[
    z = \sum_{\alpha \in \N^n} x^\alpha \bsf_\alpha \quad \text{for some } \bsf_\alpha \in \bF_\psi^{(-\alpha)}.
  \]
  For $i \in \{1,2,\dotsc,n-1\}$, it follows from Lemma~\ref{lem:Sn-commutation-mod-lower-terms} that $s_i z - \left( \prescript{s_i}{}{z} \right) s_i \in P_n(F)$.  Thus
  \[
    z
    = s_i z s_i
    = \sum_{\alpha \in \N^n} x^{s_i \cdot \alpha} \left( \prescript{s_i}{}{\bsf_\alpha} \right) + a s_i,
  \]
  for some $a \in P_n(F)$.  It then follows from Theorem~\ref{theo:AnF-basis} that $z$ is invariant under the superpermutation action of $S_n$.

  In remains to prove that elements of the form \eqref{eq:center-form} commute with elements of $S_n$.  Let $1 \le i \le n-1$.  Since $s_i$ commutes with $x_j$ and $f_j$ for all $j \ne i,i+1$ and $f \in F$, it suffices to check that $s_i$ commutes with elements of the form
  \[
    x_i^q x_{i+1}^r \bsf + x_i^r x_{i+1}^q \prescript{s_i}{}{\bsf},\quad q,r \in \N,\ \bsf \in \bF_\psi^{(-\alpha)},
  \]
  where
  \[
    \alpha = (0,\dotsc,0,q,r,0,\dotsc,0),
  \]
  with $q$ appearing in the $i$-th component.

  We compute
  \begin{equation} \label{eq:center-comp-simplify1}
    s_i x_i^q x_{i+1}^r \bsf s_i
    \stackrel{\eqref{eq:sx^k-commutation1}}{=} x_{i+1}^q s_i x_{i+1}^r \bsf s_i - t_{i,i+1}^{(q)} x_{i+1}^r \bsf s_i
    \stackrel{\eqref{eq:sx^k-commutation2}}{=} x_i^r x_{i+1}^q \prescript{s_i}{}{\bsf} + x_{i+1}^q t_{i+1,i}^{(r)} \bsf s_i - t_{i,i+1}^{(q)} x_{i+1}^r \bsf s_i.
  \end{equation}
  By symmetry, we have
  \begin{equation} \label{eq:center-comp-simplify2}
    s_i x_i^r x_{i+1}^q \prescript{s_i}{}{\bsf} s_i
    = x_i^q x_{i+1}^r \bsf + x_{i+1}^r t_{i+1,i}^{(q)} \prescript{s_i}{}{\bsf} s_i - t_{i,i+1}^{(r)} x_{i+1}^q \prescript{s_i}{}{\bsf} s_i.
  \end{equation}
  We would like to show that
  \begin{equation} \label{eq:center-computation-want-zero}
    s_i \left( x_i^q x_{i+1}^r \bsf + x_i^r x_{i+1}^q \prescript{s_i}{}{\bsf} \right) s_i - x_i^q x_{i+1}^r \bsf - x_i^r x_{i+1}^q \prescript{s_i}{}{\bsf}
  \end{equation}
  is equal to zero.  By \eqref{eq:center-comp-simplify1} and \eqref{eq:center-comp-simplify2}, we see that \eqref{eq:center-computation-want-zero} is equal to
  \begin{equation} \label{eq:center-computation-want-zero2}
    x_{i+1}^q t_{i+1,i}^{(r)} \bsf s_i - t_{i,i+1}^{(q)} x_{i+1}^r \bsf s_i + x_{i+1}^r t_{i+1,i}^{(q)} \prescript{s_i}{}{\bsf} s_i - t_{i,i+1}^{(r)} x_{i+1}^q \prescript{s_i}{}{\bsf} s_i.
  \end{equation}
  Now
  \begin{align*}
    x_{i+1}^q t_{i+1,i}^{(r)} \bsf
    &= \sum_{b \in B} \sum_{\ell=0}^{r-1} x_{i+1}^q b_{i+1} x_i^{r-\ell-1} x_{i+1}^\ell b_i^\vee \bsf & \\
    &\stackrel{\mathclap{\eqref{rel:xF-commutation}}}{=}\  \sum_{b \in B} \sum_{\ell=0}^{r-1} x_{i+1}^\ell \psi^{\ell-q}(b)_{i+1} \psi^{\ell-r-1}(b^\vee)_i x_{i+1}^q x_i^{r-\ell-1} \bsf & \\
    &= \sum_{b \in B} \sum_{\ell=0}^{r-1} x_{i+1}^\ell \bsf \psi^{\ell-q-r}(b)_{i+1} \psi^{\ell-q-r-1}(b^\vee)_i x_{i+1}^q x_i^{r-\ell-1} & \left(\text{since $\bsf \in \bF_\psi^{(-\alpha)}$}\right) \\
    &= \sum_{b \in B} \sum_{\ell=0}^{r-1} x_{i+1}^\ell \bsf \psi^{-1}(b)_{i+1} b^\vee_i x_{i+1}^q x_i^{r-\ell-1} & \text{(see explanation below)} \\
    &\stackrel{\mathclap{\eqref{eq:psi-tij-reverse}}}{=}\ \ \sum_{\ell=0}^{r-1} x_{i+1}^\ell \bsf t_{i,i+1} x_{i+1}^q x_i^{r-\ell-1} & \\
    &\stackrel{\mathclap{\eqref{eq:Ft-commutation}}}{=}\ \ \sum_{\ell=0}^{r-1} x_{i+1}^\ell t_{i,i+1} x_{i+1}^q x_i^{r-\ell-1} \prescript{s_i}{}{\bsf} & \text{(since $\psi(\bsf) = \bsf$)} \\
    &= t_{i,i+1}^{(r)} x_{i+1}^q \prescript{s_i}{}{\bsf},
  \end{align*}
  where, in the fourth equality, we used the fact that the definition of $t_{i,j}$ is independent of the choice of basis and so we summed over the basis $\{\psi^{\ell-q-r+1}(b) \mid b \in B\}$ (and abused notation by calling this new basis $B$ again.)

  Thus, the first and last terms in \eqref{eq:center-computation-want-zero2} cancel.  The proof that the second and third cancel is similar.

  \details{
    We have
    \begin{align*}
      t_{i,i+1}^{(q)} x_{i+1}^r \bsf
      &= \sum_{b \in B} \sum_{\ell=0}^{q-1} b_i x_{i+1}^{q-1-\ell} x_i^\ell b_{i+1}^\vee x_{i+1}^r \bsf & \\
      &\stackrel{\mathclap{\eqref{rel:xF-commutation}}}{=}\ \sum_{b \in B} \sum_{\ell=0}^{q-1} x_i^\ell x_{i+1}^r \psi^\ell(b)_i \psi^{\ell+r-q+1}(b^\vee)_{i+1} x_{i+1}^{q-\ell-1} \bsf & \\
      &= \sum_{b \in B} \sum_{\ell=0}^{q-1} x_i^\ell x_{i+1}^r \bsf \psi^{\ell-q}(b)_i \psi^{\ell-q+1}(b^\vee)_{i+1} x_{i+1}^{q-\ell-1} & \left( \text{since $\bsf \in \bF_\psi^{(-\alpha)}$} \right) \\
      &= \sum_{b \in B} \sum_{\ell=0}^{q-1} x_i^\ell x_{i+1}^r \bsf \psi^{-1}(b_i) b_{i+1}^\vee x_{i+1}^{q-\ell-1} & \text{(changing basis $B$)} \\
      &\stackrel{\mathclap{\eqref{eq:Ft-commutation}}}{=}\ \  \sum_{\ell=0}^{q-1} x_i^\ell x_{i+1}^r \bsf t_{i+1,i} x_{i+1}^{q-\ell-1} & \\
      &\stackrel{\mathclap{\eqref{eq:Ft-commutation}}}{=}\ \ \sum_{\ell=0}^{q-1} x_i^\ell x_{i+1}^r t_{i+1,i} x_{i+1}^{q-\ell-1} \prescript{s_i}{}{\bsf} & \text{(since $\psi(\bsf) = \bsf$)} \\
      &= x_{i+1}^r t_{i+1,i}^{(q)} \prescript{s_i}{}{\bsf}.
    \end{align*}
  }
\end{proof}

\begin{eg}
  \begin{enumerate}
    \item When $F = \kk$, Theorem~\ref{theo:center} recovers the well-known result that the center of the degenerate affine Hecke algebra consists of all symmetric polynomials in $x_1,\dotsc,x_n$ (see, for example, \cite[Th.~3.3.1]{Kle05}).

    \item When $F$ is the group algebra of a finite group (see Example~\ref{eg:wreath-Hecke}), Theorem~\ref{theo:center} recovers \cite[Th.~2.10]{WW08}.

    \item When $F = \Cl$ (see Example~\ref{eg:affine-Sergeev}), we have $F_\psi^{(k)} = \kk$ for $k$ even and $F_\psi^{(k)} = 0$ for $k$ odd.  Thus, Theorem~\ref{theo:center} recovers \cite[Prop.~3.1]{Naz97} (see also \cite[Th.~14.3.1]{Kle05}), which states that the center of the affine Sergeev algebra consists of all symmetric polynomials in $x_1^2,\dotsc,x_n^2$.
  \end{enumerate}
\end{eg}

\begin{cor} \label{cor:fg-over-center}
  The center of $\cA_n(F)$ contains the ring of symmetric polynomials in $x_1^\theta,\dotsc,x_n^\theta$.  In particular, $\cA_n(F)$ is finitely generated as a module over its center.
\end{cor}

\begin{proof}
  That the center contains the ring of symmetric polynomials in $x_1^\theta,\dotsc,x_n^\theta$ follows from Theorem~\ref{theo:center} and the fact that $\kk \subseteq F_\psi^{(-k\theta)}$ for all $k \in \N$.  Thus $\cA_n(F)$ is finitely generated as a module over its center by Theorem~\ref{theo:AnF-basis}.
\end{proof}

\begin{rem} \label{rem:simples-fd}
  If $\kk$ is an algebraically closed field, then Corollary~\ref{cor:fg-over-center} implies that all simple $\cA_n(F)$-modules are finite-dimensional.
  \details{
    This follows from the more general fact that, over an algebraically closed field $\kk$, if $A$ is finitely generated over its center, then all simple $A$-modules are finite dimensional.  Indeed, let $V$ be a simple $A$-module.  Passing to the quotient algebra $A/\Ann_A V$, we may assume that the annihilator of $V$ in $A$ is zero.  Let $Z$ denote the center of $A$.  Since $V$ is a simple $A$-module, it is generated (as an $A$-module) by any nonzero element.  Since $A$ is finitely generated as a $Z$-module, it follows that $V$ is finitely generated as a $Z$-module.  Since multiplication by elements of $Z$ commutes with the $A$-action, it follows from Schur's Lemma (and the fact that the annihilator of $V$ is zero) that every nonzero element of $Z$ acts on $A$ as a nonzero scalar.  Now, let $z \in Z$, $z \ne 0$.  Then, $zV=V$ and so, by Nakayama's Lemma, there exists $a \in 1 + zZ$ such that $aV=0$.  This implies that $a=0$, and so $z$ is invertible.  Hence $Z$ is a field.  Since $\kk$ is algebraically closed, we have $Z = \kk$.  Thus $V$ is finitely generated over $\kk$.
  }
\end{rem}

\begin{prop} \label{prop:maximal-commutative}
  Suppose $A$ is a maximal commutative subalgebra of $F_\psi$.  Then $\kk[x_1,\dotsc,x_n] A^{\otimes n}$ is a maximal commutative subalgebra of $\cA_n(F)$.
\end{prop}

\begin{proof}
  Suppose $z \in \cA_n(F)$ commutes with all elements of $\kk[x_1,\dotsc,x_n] A^{\otimes n}$.  It follows immediately from Lemma~\ref{lem:center-in-PnF} that $z \in P_n(F)$.  Hence we may write $z = \sum_{\alpha \in \N^n} x^\alpha \bsf_{\alpha}$ for some $\bsf_{\alpha} \in F^{\otimes n}$ with all but finitely many $\bsf_{\alpha}$ equal to zero.  Then, for $1 \le i \le n$, we have, by \eqref{rel:xF-commutation},
  \[
    x_i z = z x_i
    \implies \sum_\alpha x_i x^\alpha \bsf_{\alpha}
    = \sum_\alpha x_i x^\alpha \psi_i \left( \bsf_{\alpha} \right).
  \]
  It then follows from Theorem~\ref{theo:AnF-basis} that $\psi_i \left( \bsf_{\alpha} \right) = \bsf_{\alpha}$ for all $\alpha$.  Hence $\bsf_{\alpha} \in F_\psi^{\otimes n}$ for all $\alpha$.

  Now suppose $\bsg \in A^{\otimes n}$.  Then
  \[
    \bsg z = z \bsg \implies
    \sum_\alpha x^\alpha \bsg \bsf_{\alpha}
    = \sum_\alpha x^\alpha \bsf_{\alpha} \bsg.
  \]
  Thus, by Theorem~\ref{theo:AnF-basis}, $\bsg \bsf_{\alpha} = \bsf_{\alpha} \bsg$.  Since $A$ is a maximal commutative subalgebra of $F_\psi$, it follows that $\bsf_{\alpha} \in A^{\otimes n}$, completing the proof of the proposition.
\end{proof}

\begin{eg}
  \begin{enumerate}
    \item When $F = \kk$ (see Example~\ref{eg:daHa}), Proposition~\ref{prop:maximal-commutative} recovers the well-known fact that $\kk[x_1,\dotsc,x_n]$ is a maximal commutative subalgebra of the degenerate affine Hecke algebra.

    \item When $F = \Cl$ (see Example~\ref{eg:affine-Sergeev}), we have $F_\psi = \kk$.  Thus, Proposition~\ref{prop:maximal-commutative} recovers the fact that $\kk[x_1,\dotsc,x_n]$ is a maximal commutative subalgebra of the affine Sergeev algebra (see \cite[Prop.~3.1]{Naz97}).
  \end{enumerate}
\end{eg}

%----------------------------------
\subsection{Jucys--Murphy elements}
%----------------------------------

Define the \emph{Jucys--Murphy elements}
\begin{equation} \label{eq:JM-elements}
  J_1 = 0,\quad
  J_k = \sum_{i=1}^{k-1} t_{i,k} s_{i,k},\quad 2 \le k \le n,
\end{equation}
where $s_{i,k} \in S_n$ is the transposition of $i$ and $k$.  (See \cite[(8.7)]{RS17}.)

\begin{prop} \label{prop:JM-quotient}
  We have a surjective algebra homomorphism $\cA_n(F) \twoheadrightarrow F^{\otimes n} \rtimes_\rho S_n$ that is the identity on $F^{\otimes n} \rtimes_\rho S_n$ and maps $x_k$ to $J_k$ for $1 \le k \le n$.
\end{prop}

\begin{proof}
  The given map is clearly a map of $\kk$-modules.  To show that it is an algebra homomorphism, it suffices to prove that it respects the relations of $\cA_n(F)$.  This follows inductively from the fact that $x_{k+1} = s_k x_k s_k + t_{k,k+1} s_k$ by \eqref{rel:sx-commutation} and that
  \[
    s_k J_k s_k + t_{k,k+1} s_k
    = \sum_{i=1}^{k-1} s_k t_{i,k} s_{i,k} s_k + t_{k,k+1} s_k
    = \sum_{i=1}^{k-1} t_{i,k+1} s_{i,k+1} + t_{k,k+1} s_k \\
    = J_{k+1}. \qedhere
  \]
\end{proof}

\begin{eg}
  \begin{enumerate}
    \item When $F = \kk$ (see Example~\ref{eg:daHa}), the $J_k$ are the usual Jucys--Murphy elements of the symmetric group.

    \item When $F$ is the group algebra of a finite group (see Example~\ref{eg:wreath-Hecke}), the $J_k$ are the Jucys--Murphy elements of the wreath Hecke algebra introduced independently in \cite[Def.~2(a)]{Pus97} and \cite[Def.~3.1]{Wan04}.

    \item When $F = \Cl$ (see Example~\ref{eg:affine-Sergeev}), the $J_k$ are the Jucys--Murphy elements of the Sergeev algebra.  See, for example, \cite[(13.22)]{Kle05}.
  \end{enumerate}
\end{eg}

%--------------------------
\subsection{Mackey Theorem\label{sec:Mackey}}
%--------------------------

For a composition $\mu = (\mu_1,\dotsc,\mu_r)$ of $n$, let
\begin{equation} \label{eq:S_mu-def}
  S_\mu \cong S_{\mu_1} \times \dotsb \times S_{\mu_r}
\end{equation}
denote the corresponding Young subgroup of $S_n$.  Then let $\cA_\mu(F)$\label{parabolic-subalg-def} denote the \emph{parabolic subalgebra} of $\cA_n(F)$ generated by $F^{\otimes n}$, $\kk[x_1,\dotsc,x_n]$, and $S_\mu$.  So we have an isomorphism of graded superalgebras
\[
  \cA_\mu(F) \cong \cA_{\mu_1}(F) \otimes \dotsb \otimes \cA_{\mu_r}(F),
\]
and an even isomorphism of $(\Z \times \Z_2)$-graded $\kk$-modules
\[
  \cA_\mu(F) \simeq P_n(F) \otimes \kk S_\mu.
\]

Let $D_{\mu,\nu}$\label{D_munu-def} denote the set of minimal length $(S_\mu,S_\nu)$-double coset representatives in $S_n$. By \cite[Lem.~1.6(ii)]{DJ86}, for $\pi \in D_{\mu,\nu}$, $S_\mu \cap \pi S_\nu \pi^{-1}$ and $\pi^{-1} S_\mu \pi \cap S_\nu$ are Young subgroups of $S_n$.  So we can define compositions $\mu \cap \pi \nu$ and $\pi^{-1} \mu \cap \nu$ by
\[
  S_\mu \cap \pi S_\nu \pi^{-1} = S_{\mu \cap \pi \nu}
  \quad \text{and} \quad
  \pi^{-1} S_\mu \pi \cap S_\nu = S_{\pi^{-1} \mu \cap \nu}.
\]
Furthermore, the map $w \mapsto \pi^{-1} w \pi$ restricts to a length preserving isomorphism
\[
  S_{\mu \cap \pi \nu} \to S_{\pi^{-1}\mu \cap \nu}.
\]
One can verify that, for $\pi \in D_{\mu,\nu}$ and $s_i \in S_{\mu \cap \pi \nu}$, we have $\pi^{-1}(i+1) = \pi^{-1}(i)+1$, and hence $\pi^{-1} s_i \pi = s_{\pi^{-1} i}$.
\details{
  The fact that $s_i \in S_{\mu \cap \pi \nu} = S_\mu \cap \pi S_\nu \pi^{-1}$ implies that $i$ and $i+1$ are in the same subset of the partition $\mu$ (thought of as a partition of the set $\{1,2,\dotsc,n\}$ into the subsets $\{1,\dotsc,\mu_1\}$, $\{\mu_1+1,\dotsc, \mu_1+\mu_2\}$, etc.) and that $\pi^{-1}(i)$ and $\pi^{-1}(i+1)$ are in the same subset of the partition $\nu$.  On the other hand, $\pi \in D_{\mu,\nu}$ means that $\pi$ does not invert elements in the same subset of $\nu$ and $\pi^{-1}$ does not invert elements in the same subset of $\mu$.  It follows that $\pi^{-1}(i) < \pi^{-1}(i+1)$.  Suppose, towards a contradiction, that $\pi^{-1}(i) + 1 < \pi^{-1}(i+1)$.  Then there exists an integer $j$ with $\pi^{-1}(i) < j < \pi^{-1}(i+1)$.  Then  $\pi^{-1}(i)$, $j$, and $\pi^{-1}(i+1)$ are are in the same subset of the partition $\nu$.  Hence $\pi$ does not invert these elements, and so $i < \pi(j) < i+1$, which is clearly a contradiction.
}
Thus, for each $\pi \in D_{\mu,\nu}$, we have an algebra isomorphism
\begin{gather*}
  \varphi_{\pi^{-1}} \colon \cA_{\mu \cap \pi v}(F) \to \cA_{\pi^{-1} \mu \cap \nu}(F), \\
  \varphi_{\pi^{-1}}(\sigma) = \pi^{-1} \sigma \pi,\quad \varphi_{\pi^{-1}}(f_i) = f_{\pi^{-1}i},\quad \varphi_{\pi^{-1}}(x_i) = x_{\pi^{-1}i},\quad \sigma \in S_{\mu \cap \pi \nu},\ f \in F,\ 1 \le i \le n.
\end{gather*}
\details{
  As noted above, $\varphi_{\pi^{-1}}$ induces an isomorphism $\kk S_{\mu \cap \pi \nu} \to \kk S_{\pi^{-1} \mu \cap \nu}$.  It is also clear that it induces an automorphism of $F^{\otimes n}$ and of $\kk[x_1,\dotsc,x_n]$.  It obviously preserves the relations \eqref{rel:xF-commutation} and \eqref{rel:SF-commutation}.  Now suppose $j \ne i,i+1$.  Then $\pi^{-1} i \ne \pi^{-1} j$ and $\pi^{-1}(i)+1 = \pi^{-1}(i+1) \ne \pi^{-1} j$ by the above.  Thus
  \[
    \varphi_{\pi^{-1}} (s_i x_j)
    = s_{\pi^{-1} i} x_{\pi^{-1} j}
    \stackrel{\eqref{rel:sx-triv-commutation}}{=} x_{\pi^{-1} j} s_{\pi^{-1} i}
    = \varphi_{\pi^{-1}} (x_j s_i),
  \]
  so $\varphi_{\pi^{-1}}$ preserves \eqref{rel:sx-triv-commutation}.  Finally,
  \begin{multline*}
    \varphi_{\pi^{-1}} (s_i x_i)
    = s_{\pi^{-1} i} x_{\pi^{-1} i}
    \stackrel{\eqref{rel:sx-commutation}}{=} x_{\pi^{-1}(i)+1} s_{\pi^{-1} i} - t_{\pi^{-1} i,\pi^{-1}(i)+1}
    \\
    = x_{\pi^{-1}(i+1)} s_{\pi^{-1} i} - t_{\pi^{-1} i,\pi^{-1}(i+1)}
    = \varphi_{\pi^{-1}} (x_{i+1} s_i - t_{i,i+1}),
  \end{multline*}
  and so $\varphi_{\pi^{-1}}$ preserves \eqref{rel:sx-commutation}.
}
If $N$ is a left $\cA_{\pi^{-1} \mu \cap \nu}(F)$-module, we let $\prescript{\pi}{}{N}$ denote the left $\cA_{\mu \cap \pi \nu}$-module with action given by
\[
  a \cdot v = \varphi_{\pi^{-1}}(a) v,\quad a \in \cA_{\mu \cap \pi \nu},\ v \in \prescript{\pi}{}{N} = N,
\]
where juxtaposition denotes the original action on $N$.

The inclusion $\cA_\mu(F) \subseteq \cA_n(F)$ gives rise to restriction and induction functors
\begin{equation} \label{eq:Ind-Res-def}
  \Res^n_\mu \colon \cA_n(F)\md \to \cA_\mu(F)\md,\qquad
  \Ind^n_\mu \colon \cA_\mu(F)\md \to \cA_n(F)\md.
\end{equation}

\begin{theo}[Mackey Theorem for $\cA_n(F)$] \label{theo:Mackey}
  Suppose $M$ is an $\cA_\nu(F)$-module.  Then $\Res^n_\mu \Ind_\nu^n M$ admits a filtration with subquotients evenly isomorphic to
  \[
    \Ind_{\mu \cap \pi \nu}^\mu \prescript{\pi}{}{(\Res^\nu_{\pi^{-1} \mu \cap \nu}M)},
  \]
  one for each $\pi \in D_{\mu,\nu}$.  Furthermore, the subquotients can be taken in any order refining the strong Bruhat order on $D_{\mu,\nu}$.  In particular, $\Ind_{\mu \cap \nu}^\mu \Res_{\mu \cap \nu}^\nu M$ appears as a submodule.
\end{theo}

\begin{proof}
  The proof is almost identical to the proofs of \cite[Th.~3.5.2]{Kle05} and \cite[Th.~14.5.2]{Kle05} and hence will be omitted.
\end{proof}

\details{
  Fix some total order $\prec$ refining the strong Bruhat order $<$ on $D_{\mu,\nu}$.  For $\pi \in D_{\mu,\nu}$, set
  \begin{gather*}
    \cB_{\preceq \pi} = \bigoplus_{\sigma \in D_{\mu,\nu},\, \sigma \preceq \pi} \cA_\mu(F) \sigma S_\nu, \\
    \cB_{\prec \pi} = \bigoplus_{\sigma \in D_{\mu,\nu},\, \sigma \prec \pi} \cA_\mu(F) \sigma S_\nu, \\
    \cB_\pi = \cB_{\preceq \pi}/\cB_{\prec \pi}.
  \end{gather*}
  It follows from Lemma~\ref{lem:Sn-commutation-mod-lower-terms} that $\cB_{\preceq \pi}$ and $\cB_{\prec \pi}$ are invariant under right multiplication by $P_n(F)$.  Since $\cA_\nu(F) = S_\nu P_n(F)$, it follows that
  \begin{equation} \label{eq:dH-B-filtation}
    \cA_n(F) = \sum_{\pi \in D_{\mu,\nu}} \cB_{\preceq \pi} \tag{$*$}
  \end{equation}
  is a filtration of $(\cA_\mu(F), \cA_\nu(F))$-modules.

  \begin{lem*}
    Consider $\cA_\mu(F)$ as a $(\cA_\mu(F),\cA_{\mu \cap \pi \nu}(F))$-bimodule and $\cA_\nu(F)$ as a $(\cA_{\pi^{-1}\nu \cap \nu}, \cA_\nu(F))$-bimodule in the natural ways.  Then $\prescript{\pi}{}{\cA}_\nu(F)$ is a $(\cA_{\mu \cap \pi \nu}(F), \cA_\nu(F))$-bimodule and
    \[
      \cB_\pi \simeq \cA_\mu(F) \otimes_{\cA_{\mu \cap \pi \nu}(F)} \prescript{\pi}{}{\cA_\nu(F)}
    \]
    as $(\cA_\mu(F), \cA_\nu(F))$-bimodules.
  \end{lem*}

  \begin{proof}
    Define a bilinear map
    \[
      \cA_\mu(F) \times \prescript{\pi}{}{\cA_\nu(F)} \to \cB_\pi = \cB_{\preceq \pi}/\cB_{\prec \pi},\quad
      (a,a') \mapsto a \pi a' + \cB_{\prec \pi}.
    \]
    Suppose $a \in \cA_\mu(F)$, $a' \in \prescript{\pi}{}{\cA_\nu(F)}$, and $z \in \cA_{\mu \cap \pi \nu}(F)$.  Then, if $\phi$ denotes the given bilinear map, we have
    \[
      \phi(az,a')
      = az \pi a' + \cB_{\prec \pi}
      = a \pi \left( \pi^{-1} z \pi \right) a' + \cB_{\prec \pi}
      = a \pi \varphi_{\pi^{-1}}(z) a' + \cB_{\prec \pi}
      = a\pi (z \cdot a') + \cB_{\prec \pi}
      = \phi(a, z \cdot a').
    \]
    Hence $\phi$ is $\cA_{\mu \cap \pi \nu}(F)$-balanced.  Thus, it yields a $(\cA_\mu(F), \cA_\nu(F))$-bimodule homomorphism
    \[
      \Phi \colon \cA_\mu(F) \otimes_{\cA_{\mu \cap \pi \nu}(F)} \prescript{\pi}{}{\cA_\nu(F)} \to \cB_\pi.
    \]
    By \cite[\S3.5, Property (3)]{Kle05}, the elements
    \[
      \pi^\alpha f u \otimes v,\quad \alpha \in \N^n,\ f \in B^{\otimes n},\ u \in S_\mu,\ v \in S_\nu \cap D_{\pi^{-1} \mu \cap \nu}^{-1},
    \]
    form a basis of $\cA_\mu \otimes_{\cA_{\mu \cap \pi \nu}(F)} \prescript{\pi}{}{\cA_\nu(F)}$ as a $\kk$-module.  Then, again by \cite[\S3.5, Property (3)]{Kle05}, the images of these elements under $\phi$ is a basis of $\cB_\pi$.
  \end{proof}

  \begin{proof}[Proof of Theorem~\ref{theo:Mackey}]
    This follows from the filtration \eqref{eq:dH-B-filtation}, the above lemma, and the isomorphism
    \[
      \left( \cA_\mu(F) \otimes_{\cA_{\mu \cap \pi \nu}(F)} \prescript{\pi}{}{\cA_\nu(F)} \right) \otimes_{\cA_\nu(F)} M
      \simeq \Ind_{\mu \cap \pi \nu}^\mu \prescript{\pi}{}{\left( \Res_{\pi^{-1} \mu \cap \nu}^\nu M \right)},
    \]
    which is straightforward to verify.
  \end{proof}\
}

%---------------------------------
\subsection{Intertwining elements\label{subsec:interwining-elements}}
%---------------------------------

Intertwining elements play a fundamental role in the treatment of integral modules for degenerate affine Hecke algebras (see \cite[\S3.8]{Kle05}), affine Sergeev algebras (see \cite[\S14.8]{Kle05}), and wreath Hecke algebras (see \cite[\S5.2]{WW08}).  While the treatment of integral modules for affine wreath product algebras is beyond the scope of the current paper, we introduce here intertwining elements in these algebras and prove that they have properties analogous to those in the aforementioned special cases.

For $1 \le i < n$, define
\[
  \Omega_i
  := x_{i+1}^\theta s_i - s_i x_{i+1}^\theta
  = s_i(x_i^\theta - x_{i+1}^\theta) + t_{i,i+1}^{(\theta)}
  = (x_{i+1}^\theta - x_i^\theta) s_i - t_{i+1,i}^{(\theta)},
\]
where the last two equalities follow from Lemma~\ref{lem:sx-higher-commutation}.

\begin{lem}
  For $1 \le i,j < n$, we have
  \begin{gather}
    \Omega_i^2
    = \left( t_{i,i+1}^{(\theta)} \right)^2 - \left( x_i^\theta - x_{i+1}^\theta \right)^2, \label{eq:Omega-squared} \\
    \Omega_i f_j = f_{s_i(j)} \Omega_i, \label{eq:Omega-f-commutation} \\
    \Omega_i x_j = x_{s_i(j)} \Omega_i. \label{eq:Omega-x-commutation} \\
    \Omega_i \Omega_j = \Omega_j \Omega_i \quad \text{if} \quad |i-j|>1, \label{eq:Omega-distant-commute}
  \end{gather}
\end{lem}

\begin{proof}
  We have
  \[
    \Omega_i^2
    = s_i (x_i^\theta - x_{i+1}^\theta) s_i (x_i^\theta - x_{i+1}^\theta) + s_i (x_i^\theta - x_{i+1}^\theta) t_{i,i+1}^{(\theta)}
    + t_{i,i+1}^{(\theta)} s_i (x_i^\theta - x_{i+1}^\theta) + \left( t_{i,i+1}^{(\theta)} \right)^2.
  \]
  Now, using Lemma~\ref{lem:sx-higher-commutation}, we have
  \[
    (x_i^\theta-x_{i+1}^\theta) s_i
    = s_i (x_{i+1}^\theta - x_i^\theta) - t_{i,i+1}^{(\theta)} - t_{i+1,i}^{(\theta)}.
  \]
  Using this and \eqref{eq:tii+1-conjugation}, relation \eqref{eq:Omega-squared} follows.  Relation \eqref{eq:Omega-f-commutation} follows easily from \eqref{rel:xF-commutation} and the fact that $\theta$ is the order of $\psi$.

  To prove \eqref{eq:Omega-x-commutation}, we compute
  \begin{multline*}
    x_{i+1} \Omega_i - \Omega_i x_i
    = (x_{i+1}s_i - s_i x_i) (x_i^\theta - x_{i+1}^\theta) - \sum_{b \in B} b_i (x_i^\theta - x_{i+1}^\theta) b_{i+1}^\vee
    \\
    \stackrel{\eqref{rel:sx-commutation}}{=} t_{i,i+1} (x_i^\theta - x_{i+1}^\theta)  - \sum_{b \in B} b_i (x_i^\theta - x_{i+1}^\theta) b_{i+1}^\vee
    = 0,
  \end{multline*}
  and
  \begin{multline*}
    x_i \Omega_i - \Omega_i x_{i+1}
    = (x_{i+1}^\theta - x_i^\theta) (x_i s_i - s_i x_{i+1}) + \sum_{b \in B} b_{i+1} (x_{i+1}^\theta - x_i^\theta) b_i^\vee
    \\
    \stackrel{\eqref{rel:sx-commutation2}}{=} -(x_{i+1}^\theta - x_i^\theta) t_{i+1,i} + \sum_{b \in B} b_{i+1} (x_{i+1}^\theta - x_i^\theta) b_i^\vee
    = 0.
  \end{multline*}
  This completes the proof of \eqref{eq:Omega-x-commutation}.  Relation \eqref{eq:Omega-distant-commute} is straightforward.
\end{proof}

%%%%%%%%%%%%%%%%%%%%%%%%%%%%%%%%%%%%%%%%%%
%
\section{Classification of simple modules\label{sec:simples}}
%
%%%%%%%%%%%%%%%%%%%%%%%%%%%%%%%%%%%%%%%%%%

In this section we assume that $\kk$ is an algebraically closed field of characteristic not equal to two.  We also continue to assume, except as noted in the first subsection, that its characteristic does not divide the order $\theta$ of the Nakayama automorphism $\psi$.  Hence all simple $\cA_n(F)$-modules are finite-dimensional (see Remark~\ref{rem:simples-fd}).  In this section we classify these modules.  Our approach is inspired by that of \cite{WW08}.  However, the fact that our setup is more general (e.g., we allow the Nakayama automorphism to be nontrivial) makes the proofs somewhat more involved.  We note that the most important case is when $\delta = 0$ (i.e.\ the $\Z$-grading on $F$ is trivial).  See Remark~\ref{rem:nontrivial-gradings}.

%---------------------------------------------
\subsection{\texorpdfstring{Simple $\kk[x] \ltimes F$-modules}{Simple k[x]F-modules}\label{subsec:kxF-modules}}
%---------------------------------------------

The results of this subsection are valid for any algebra $F$ over an algebraically closed field $\kk$ of characteristic not equal to two, and any algebra automorphism $\psi$ of $F$ of finite order $\theta$.  We continue to use the notation $F$ and $\psi$ that were introduced for Frobenius algebras earlier since that will be our main interest.

Fix a simple $F$-module $L$.  For $k \in \Z$, write $\prescript{k}{}{L}$\label{^kL-def} for $\prescript{\psi^k}{}{L}$ (see \eqref{eq:twisted-module-def}).  Let $r_L$\label{r_L-def} be the smallest positive integer such that $\prescript{r_L}{}{L} \simeq L$, and let $m_L$\label{m_L-def} be the smallest positive integer such that $\psi^{m_L}(f)(v) = f(v)$ for all $f \in F$ and $v \in L$.  It follows that $r_L$ divides $m_L$ and that $m_L$ divides $\theta$.
\details{
  Any $a \in \Z m_L + \Z \theta$ has the property that $\psi^a(f)v = fv$ for all $f \in F$ and $v \in L$.  Thus $\gcd (m_L,\theta)$ has this property.  By our choice of $m_L$, this implies that $m_L = \gcd (m_L,\theta)$ and so $m_L$ divides $\theta$.
}

\begin{eg}
  \begin{enumerate}
    \item If $F = \Cl$ and $L$ is its unique simple module, then $r_L=1$ (since there is only one simple module and it is of type $\tQ$) and $m_L=\theta=2$.

    \item Recall the Taft Hopf algebra of Example~\ref{eg:Taft}.  We have $\prescript{\psi}{}{L}_k \simeq L_{k+1}$ if we take $\omega = e^{2 \pi i / q}$.  Thus $r_{L_k}=m_{L_k}=\theta=q$.

    \item Let $G$ be a finite group with a noncentral element $h$.  Let $\psi$ be conjugation by $h$, and let $L$ be the one-dimensional trivial representation of $G$.  Then $r_L = m_L = 1$, while $\theta > 1$.
  \end{enumerate}
\end{eg}

By Schur's Lemma and the definition of $r_L$, there is an even $F$-module isomorphism
\begin{equation} \label{eq:tau_L-def}
  \tau_L \colon L \xrightarrow{\simeq} \prescript{r_L}{}{L},
\end{equation}
which is unique up to a nonzero scalar.  Thus $\tau_L \left( \psi^{r_L}(f) v \right) = f \tau_L (v)$ for all $f \in F$ and $v \in L$.
\details{
  We have
  \[
    \tau \left( \psi^{r_L}(f) v \right)
    = \psi^{r_L}(f) \cdot \tau (v)
    = f \tau(v).
  \]
}
Note that $\tau_L^{m_L/r_L} \colon L \xrightarrow{\cong} \prescript{m_L}{}{L} = L$.  Thus, by Schur's Lemma, $\tau_L^{m_L/r_L}$ is multiplication by a nonzero scalar.  Rescaling $\tau_L$ if necessary, we may assume that
\begin{equation} \label{eq:tau-normalization}
  \tau_L^{m_L/r_L} = \id.
\end{equation}

Now, for $a \in \kk$ we define an $F$-module $L(a)$ as follows.  We let
\begin{equation} \label{eq:L(a)-def}
  L(a) := \bigoplus_{\ell=0}^{r_L-1} \prescript{\ell}{}{L},
\end{equation}
as $F$-modules.  The action of $x$ is given by
\[
  x (v_0,\dotsc,v_{r_L-1}) = (v_1,v_2,\dotsc,v_{r_L-1},a \tau_L(v_0)),\quad v_i \in \prescript{i}{}{L},\quad 1 \le i \le r_L-1.
\]
It is straightforward to verify that $L(a)$ is a simple $\kk[x] \ltimes F$-module when $a \in \kk^\times$.
\details{
  Let $r = r_L$.  First of all, for $f \in F$, we have
  \begin{align*}
    x \psi(f)(v_0,\dotsc,v_{r-1})
    &= x \left( \psi(f)v_0, fv_1, \dotsc, \psi^{2-r}(f)v_{r-1} \right)
    \\
    &= \left( fv_1,\psi^{-1}(f)v_2, \dotsc, \psi^{2-r}(f)v_{r-1}, a \tau_L(\psi(f)v_0) \right)
  \end{align*}
  and
  \begin{align*}
    fx(v_0,\dotsc,v_{r-1})
    &= f \left( v_1,v_2,\dotsc,v_{r-1},a \tau_L(v_0) \right) \\
    &= \left( fv_1,\psi^{-1}(f)v_2,\dotsc,\psi^{2-r}(f)v_{r-1}, \psi^{1-r}(f) a \tau_L(v_0) \right) \\
    &= \left( fv_1,\psi^{-1}(f)v_2, \dotsc, \psi^{2-r}(f)v_{r-1}, a \tau_L(\psi(f)v_0) \right).
  \end{align*}

  To see that $L(a)$ is simple when $a \in \kk^\times$, suppose $V$ is a nonzero $\kk[x] \ltimes F$-submodule of $L(a)$.  Since $V$ is an $F$-submodule, it is a sum of some of the summands $\prescript{\ell}{}{L}$.  But then invariance under the action of $x$ implies that it must contain \emph{all} the summands.
}

\begin{eg}
  Suppose $F = \Cl$ and $\psi$ is the Nakayama automorphism (see Example~\ref{eg:affine-Sergeev}).  The algebra $F$ has one simple module $L$, which arises from the action of $F$ on itself by left multiplication.  Then $r_L = 1$ and $m_L = \theta = 2$.  In the notation of \cite[\S16.1]{Kle05}, choose $\tau_L$ to be the map given by $\tau_L(v_1)=v_1$ and $\tau_L(v_{-1})=-v_{-1}$.  Then the module $L \left( \sqrt{q(i)} \right)$, $i \in I$, (notation as in the current paper, with $I$ defined in \cite[(15.2)]{Kle05}) is precisely the module denoted by $L(i)$ in \cite[\S16.1]{Kle05}.
\end{eg}

\begin{lem} \label{lem:kxF-irred-isom}
  Suppose $L$ and $L'$ are simple $F$-modules and $a,b \in \kk^\times$.  The $\kk[x] \ltimes F$ modules $L(a)$ and $L'(b)$ are evenly isomorphic if and only if $a=b$ and $L' \simeq \prescript{\ell}{}{L}$ as $F$-modules for some $\ell \in \{0,\dotsc,r_L-1\}$.
\end{lem}

\begin{proof}
  Suppose $L(a) \simeq L'(b)$ are evenly isomorphic as $\kk[x] \ltimes F$-modules.  Considering the decomposition as $F$-modules, we see that $L' \simeq \prescript{\ell}{}{L}$ for some $\ell \in \{0,\dotsc,r_L-1\}$.  Relabelling if necessary, we may assume $L = L'$.  Then, since $x^{r_L}$ acts on $L(a)$ as $a \tau_L$ and on $L(b)$ as $b \tau_L$, it follows that $a=b$.  The converse statement is clear.
\end{proof}

\begin{rem}
  \begin{enumerate}
    \item Note that the indexing $L(a)$ depends on the choice of $\tau_L$.  In light of the condition \eqref{eq:tau-normalization}, the choice of $\tau_L$ is unique up to multiplication by an $(m_L/r_L)$-th root of unity.  Different choices of $\tau_L$ simply shift the parameter $a$ by multiplication by this root of unity.

    \item It is important to note that Lemma~\ref{lem:kxF-irred-isom} is a statement about \emph{even} isomorphisms.  In general, it is possible for $L(a) \cong L(b)$, via an odd isomorphism, for $a \ne b$.  In particular, when $F = \Cl$, we have $\Pi L(a) \simeq L(-a)$, and so $L(a)$ is isomorphic to $L(-a)$ via an odd isomorphism.   See \cite[\S16.1]{Kle05}.
  \end{enumerate}
\end{rem}

We have an algebra homomorphism
\begin{equation} \label{eq:x-to-0-map}
  \kk[x] \ltimes F \to F,\quad x \mapsto 0,\ f \mapsto f,\quad f \in F.
\end{equation}
For an $F$-module $V$, let $V(0)$ denote the $\kk[x] \ltimes F$-module obtained by inflating via the homomorphism \eqref{eq:x-to-0-map}.  Clearly $V(0)$ is simple if $V$ is, and $V(0) \simeq V'(0)$ as $\kk[x] \ltimes F$-modules if and only if $V \simeq V'$ as $F$-modules.

\begin{prop} \label{prop:kxF-simple-modules}
  The modules
  \begin{gather*}
    L(a),\quad L \in \cS(F),\ a \in \kk,
  \end{gather*}
  are a complete list, up to even isomorphism and degree shift, of simple $\kk[x] \ltimes F$-modules.  Furthermore,
  \begin{itemize}
    \item if $a \ne b$, then $L(a) \not \simeq L(b)$,
    \item $L(0) \simeq L'(0)$ if and only if $L \simeq L'$, and
    \item for $a \ne 0$, $L(a) \simeq L'(a)$ if and only if $L' \simeq \prescript{\ell}{}{L}$ for some $\ell \in \Z$.
  \end{itemize}
\end{prop}

\begin{proof}
  It remains to prove that every simple $\kk[x] \ltimes F$-module is evenly isomorphic to $L(a)$ for some simple $F$-module $L$ and $a \in \kk$.

  Let $V$ be a simple $\kk[x] \ltimes F$-module.  Shifting the degree if necessary, we assume that $V$ is concentrated in nonnegative degree with nonzero degree zero piece.  We first prove that $V$ is semisimple as an $F$-module.  Since the center of $\kk[x] \ltimes F$ contains $\kk[x^\theta]$, and $\kk[x] \ltimes F$ is of finite rank over $\kk[x^\theta]$, it follows that $V$ is finite dimensional.  Now let $L$ be a simple $F$-submodule of $V$.  Then $x^k L$ is either zero or is a simple $F$-submodule of $V$ evenly isomorphic to $\prescript{k}{}{L}$.
  \details{
    We have
    \[
      x^k(fv) = (x^kf) v = (\psi^{-k}(f)x) v = \psi^{-k}(f) (xv) = f \cdot (xv).
    \]
    Furthermore, $L$ is simple as an $F$-module.  Hence any homomorphic image is either zero or simple.
  }
  Since $\sum_{k=0}^\infty x^k L$ is a $\kk[x] \ltimes F$-submodule of $V$ and $V$ is simple as a $\kk[x] \ltimes F$-module, we have $V = \sum_{k=0}^\infty x^k L$ (with the sum actually being finite by the finite-dimensionality of $V$), and so $V$ is semisimple as an $F$-module.

  It follows that
  \begin{equation} \label{eq:kxF-reduced-isotypic}
    V = \bigoplus_{\ell=0}^{r_L-1} V_\ell,
  \end{equation}
  where $V_\ell$ is the $\prescript{\ell}{}{L}$-isotypic component of $V$.

  Now, if $x$ acts by zero on any $F$-submodule of a summand $V_\ell$ in \eqref{eq:kxF-reduced-isotypic}, that submodule would be a $\kk[x] \ltimes F$-submodule, and hence equal to all of $V$.  Furthermore, simplicity of $V$ would imply that $V_\ell \simeq \prescript{\ell}{}{L}$.  Thus $V \simeq \prescript{\ell}{}{L}(0)$ and we are done.

  Now assume that $x$ does not act by zero on any $F$-submodule of a summand in \eqref{eq:kxF-reduced-isotypic}.  It follows from Schur's Lemma that the action of $x$ induces linear isomorphisms $V_\ell \to V_{\ell+1}$ for all $0 \le \ell \le r_L-1$.  In addition, $x^{r_L}$ induces a linear automorphism of $V_0$.  Let $v \in V_0$ be an eigenvector of $x^{r_L}$.  It follows from the relation $x f = \psi^{-1}(f) x$, $f \in F$, that $x^{r_L}$ leaves $Fv$ invariant.  Then $\bigoplus_{\ell=0}^{r_L-1}\, x^\ell Fv$ is a $\kk[x] \ltimes F$-submodule of $V$, with $x^\ell Fv \subseteq V_\ell$.  Since $V$ is simple, it follows that $Fv = V_0$.  So $V_0$ is simple and hence isomorphic to $L$.

  It follows from the definition of $r_L$ that the action of $x^{r_L}$ induces an even isomorphism $L \simeq \prescript{r_L}{}{L}$.  Thus there is some nonzero scalar $a \in \kk^\times$ such that $x^{r_L}$ acts by $a \tau_L$.  Then clearly $V \simeq L(a)$.
\end{proof}

%----------------------------------
\subsection{\texorpdfstring{Action of $t_{k,\ell}$}{Action of $t{\textunderscore}\{k,l\}$}}
%----------------------------------

For $\Z$-graded super vector spaces $V_1$ and $V_2$ define the linear map
\begin{equation} \label{eq:flip-def}
  \flip \colon V_1 \otimes V_2 \to V_2 \otimes V_1,\quad \flip(v_1 \otimes v_2) = (-1)^{\bar v_1 \bar v_2} v_2 \otimes v_1.
\end{equation}
In particular, $\flip \colon F^{\otimes 2} \to F^{\otimes 2}$ is an algebra homomorphism, and for $F$-modules $V_1$ and $V_2$,
\[
  \flip \colon \prescript{\flip}{}{(V_1 \boxtimes V_2)} \to V_2 \boxtimes V_1
\]
is an isomorphism of $F^{\otimes 2}$-modules.
\details{
  For $f,g \in F$, $v_1 \in V_1$, and $v_2 \in V_2$, we have
  \begin{multline*}
    \flip \Big( (f \otimes g) \cdot (v_1 \otimes v_2) \Big)
    = (-1)^{\bar f \bar g+ \bar f \bar v_1}\, \flip (gv_1 \otimes fv_2)
    = (-1)^{(\bar g + \bar v_1) \bar v_2} fv_2 \otimes gv_1
    \\
    = (-1)^{\bar v_1 \bar v_2} (f \otimes g)(v_2 \otimes v_1)
    = (f \otimes g) \flip(v_1 \otimes v_2),
  \end{multline*}
  where the $\cdot$ denotes the twisted action and juxtaposition denotes the usual (i.e.\ untwisted) action.
}
As explained in Section~\ref{subsec:superalg-background}, if $V_1$ and $V_2$ are simple, we have an induced even isomorphism
\begin{equation} \label{eq:flip-irtimes}
  \flip \colon \prescript{\flip}{}{(V_1 \irtimes V_2)} \to V_2 \irtimes V_1.
\end{equation}

\begin{lem} \label{lem:t12-action}
  Suppose $L$ and $L'$ are simple $F$-modules.
  \begin{enumerate}
    \item \label{lem-item:t12-action-zero} The element $t_{1,2} \in F^{\otimes 2}$ acts as zero on the $F^{\otimes 2}$-module $L \irtimes L'$ unless $\prescript{\psi}{}{L} \simeq L \cong L'$ and $\delta = 0$ (i.e.\ $F$ is concentrated in degree zero).

    \item \label{lem-item:t12-action-scalar} If $\prescript{\psi}{}{L} \simeq L$ as $F$-modules and $\delta = 0$, then $t_{1,2}$ acts on $L \irtimes L$ either as zero or as $\flip \circ (\tau \otimes \id)$, where $\tau \colon L \xrightarrow{\simeq} \prescript{\psi}{}{L}$ is an even isomorphism of $F$-modules.
  \end{enumerate}
\end{lem}

\begin{proof}
  By Lemma~\ref{lem:Ft-commutation}, we have
  \[
    (f \otimes g) t_{1,2} = (-1)^{\bar f \bar g} t_{1,2} \big( \psi(g) \otimes f \big),\quad f,g \in F.
  \]
  Thus, $t_{1,2}$ induces an even homomorphism (since $t_{1,2}$ is even) of $F^{\otimes 2}$-modules
  \[
    L \irtimes L' \to \prescript{\flip}{}{\left( L \irtimes \prescript{\psi}{}{L'} \right)}.
  \]
  By \eqref{eq:flip-irtimes}, we then have an even homomorphism of $F^{\otimes 2}$-modules
  \[
    \flip \circ t_{1,2} \colon  L \irtimes L' \to \prescript{\psi}{}{L'} \irtimes L.
  \]
  Thus, by Schur's Lemma, $t_{1,2}$ acts as zero unless $L \irtimes L' \simeq \prescript{\psi}{}{L'} \irtimes L$ as $F^{\otimes 2}$-modules or, equivalently, unless $\prescript{\psi}{}{L} \simeq L \cong L'$.
  \details{
    The isomorphism $L \irtimes L' \simeq \prescript{\psi}{}{L'} \irtimes L$ holds if and only if $L \cong \prescript{\psi}{}{L'}$, $L' \cong L$, and these isomorphisms are either both even or both odd.  But this is equivalent to $\prescript{\psi}{}{L} \simeq L \cong L'$.
  }
  Furthermore, since $|t_{1,2}| = \delta$, it follows that $t_{1,2}$ must act as zero if $\delta > 0$ (since isomorphisms live in degree zero).  This proves part~\ref{lem-item:t12-action-zero}.

  Now suppose $\tau \colon L \to \prescript{\psi}{}{L}$ is an even isomorphism of $F$-modules and $\delta = 0$.  Then
  \[
    (\tau^{-1} \otimes \id) \circ \flip \circ t_{1,2} \colon L \irtimes L \to L \irtimes L
  \]
  is an even homomorphism of $F^{\otimes 2}$-modules and thus must be multiplication by a scalar, by Schur's Lemma.  If this scalar is nonzero, we may rescale $\tau$ so that this scalar is equal to one.  This completes the proof of part~\ref{lem-item:t12-action-scalar}.
\end{proof}

If $L_1,\dotsc,L_n$ are simple $F$-modules with $L_k=L_\ell$, we have an even $F$-module homomorphism
\begin{equation} \label{eq:rho_kl-def}
  \rho_{k,\ell} \colon L_1 \irtimes \dotsb \irtimes L_n \to L_1 \irtimes \dotsb \irtimes L_n
\end{equation}
that superpermutes the $k$-th and $\ell$-th factors.

\begin{cor} \label{cor:tkl-action}
  Suppose $k,\ell \in \{1,\dotsc,n\}$ and $k \ne \ell$.  If $L_k = L_\ell$ and $\delta = 0$, then $t_{k,\ell}$ acts on $L_1 \irtimes \dotsb \irtimes L_n$ as zero or as
  \[
    \rho_{k,\ell} \circ \left( \id^{\otimes (k-1)} \otimes \tau \otimes \id^{\otimes (n-k)} \right),
  \]
  where $\tau \colon L_k \to \prescript{\psi}{}{L}_k$ is an even isomorphism of $F$-modules.  Otherwise, $t_{k,\ell}$ acts as zero.
\end{cor}

\begin{proof}
  This follows from Lemma~\ref{lem:t12-action}.
\end{proof}

\begin{eg}
  \begin{enumerate}
    \item In the setting of Example~\ref{eg:wreath-Hecke}, the Nakayama automorphism $\psi$ is the identity and the content of Lemma~\ref{lem:t12-action} is contained in \cite[Lem.~3.1]{WW08}.

    \item If $F = \Cl$, then $F$ has one simple module $L$, $\psi$ has order 2, $\delta = 0$, and $\psi \colon \prescript{\psi}{}{L} \xrightarrow{\simeq} L$.  Then $t_{1,2}$ acts on $L \irtimes L$ as a nonzero scalar multiple of $\flip \circ (\psi \otimes \id)$.
        \details{
          The scalar is $1-i$ or $1+i$, depending on the choice of simple tensor product $L \irtimes L$.
        }
  \end{enumerate}
\end{eg}

%-------------------------------
\subsection{Associated algebras}
%-------------------------------

Recall that $\psi$ acts on the set $\cS(F)$ by twisting (see \eqref{eq:twisted-module-def}).  Let $N$\label{N-def} be the number of $\psi$-orbits in $\cS(F)$, and let
\begin{equation} \label{eq:sL-def}
  \sL_1, \sL_2,\dotsc, \sL_N
\end{equation}
be a set of representatives of these orbits.  For $k \in \{1,2,\dotsc,N\}$, let $r_k := r_{\sL_k}$,\label{r_k-def} so that $r_k$ is the smallest positive integer such that $\prescript{r_k}{}{\sL}_k \simeq \sL_k$.  Hence
\[
  \sL_1, \prescript{1}{}{\sL}_1,\dotsc, \prescript{r_1-1}{}{\sL}_1,\dotsc, \sL_N, \prescript{1}{}{\sL}_N,\dotsc, \prescript{r_N-1}{}{\sL}_N,
\]
is an enumeration of the elements of $\cS(F)$.

Let $\cA_n(F)\smod$\label{smod-def} denote the full subcategory of $\cA_n(F)\md$ consisting of finite-dimensional $\cA_n(F)$-modules that are semisimple as $F^{\otimes n}$-modules.  Let\label{cC-def}
\begin{equation} \label{eq:cC_n-def}
  \cC_n := \left\{ \mu = (\mu_1,\dotsc,\mu_N) \mid \mu_1,\dotsc,\mu_N \in \N,\ \mu_1 + \dotsb + \mu_N = n \right\}
\end{equation}
denote the set of compositions of $n$ of length at most $N$.  Recall the parabolic subalgebras $\cA_\mu(F)$ of $\cA_n(F)$ introduced in Section~\ref{sec:Mackey}.  We let $\cA_\mu(F)\smod$ denote the full subcategory of $\cA_\mu(F)\md$ consisting of finite-dimensional $\cA_\mu(F)$-modules that are semisimple as $F^{\otimes n}$-modules.

For $\mu \in \cC_n$ and $1 \le k \le n$, let $\ell_k$ denote the unique integer such that
\begin{equation} \label{eq:l_k-def}
  \mu_1 + \dotsb + \mu_{\ell_k-1} < k \le \mu_1 + \dotsb + \mu_{\ell_k}.
\end{equation}

For $1 \le k \le N$, let $\tau_k = \tau_{\sL_k}$\label{tau_k-def} (see \eqref{eq:tau_L-def}).  If $t_{1,2}$ acts as zero on $\sL_k \irtimes \sL_k$, let $\sa_k = 1$.\label{sa_k-def}  (This choice is not crucial; we choose $\sa_k=1$ for simplicity.)  Otherwise, choose $\ssb_k \in \kk^\times$ such that $t_{1,2}$ acts on $\sL_k \irtimes \sL_k$ as $\flip \circ (\ssb_k \tau_k \otimes \id)$ (see Lemma~\ref{lem:t12-action}) and define
\[
  \sa_k =
  \begin{cases}
    \ssb_k & \text{if } \sL_k \text{ is of type } \tM, \\
    \frac{1-\sqrt{-1}}{2}{\ssb_k} & \text{if } \sL_k \text{ is of type } \tQ.
  \end{cases}
\]

Let
\begin{equation} \label{eq:Lmu-def}
  \sL(\mu)
  := \sL_1 \left( \sa_1 \right)^{\boxtimes \mu_1} \boxtimes \dotsb \boxtimes \sL_N \left( \sa_N \right)^{\boxtimes \mu_N}.
\end{equation}
Thus $\sL(\mu)$ is a $P_n(F)$-module on which the action of $x_k$ is invertible for all $1 \le k \le n$.  By a slight abuse of notation, we will write $x_k^{-1}$ for the endomorphism of $\sL(\mu)$ that is the inverse of the action of $x_k$.  If $\ell_k = \ell_{k+1}$ (i.e.\ $s_k \in S_\mu$), it follows from our choices that, as an operator on $\sL(\mu)$, we have
\begin{equation} \label{eq:t-L-hat-n-action}
  t_{k,k+1} =
  \begin{cases}
    0 & \text{if $t_{1,2}$ acts as zero on $\sL_{\ell_k} \irtimes \sL_{\ell_k}$}, \\
    \tilde{t}_{k,k+1} \rho_{k,k+1} x_k & \text{if $t_{1,2}$ does not act as zero on $\sL_{\ell_k} \irtimes \sL_{\ell_k}$},
  \end{cases}
\end{equation}
where $\tilde{t}_{k,k+1}$ is the element $\sum_b 1^{\otimes (k-1)} \otimes b \otimes b^\vee \otimes 1^{\otimes (n-k-1)}$, with the sum being over a basis of $\END_F L_{\ell_k}$.  (Recall that $\END_F L_{\ell_k}$ is isomorphic to $\kk$ if $L_{\ell_k}$ is of type $\tM$, and to $\Cl$ if $L_{\ell_k}$ is of type $\tQ$.)
\details{
  We use the fact that, if $t_{1,2}$ does not act as zero on $\sL_{\ell_k} \irtimes \sL_{\ell_k}$, then $r_k = 1$ by Lemma~\ref{lem:t12-action}, and so $\sL_{\ell_k}(\sa_k) = \sL_{\ell_k}$ as $F$-modules, with $x$ acting on the $\kk[x] \ltimes F$-module $\sL_{\ell_k}(\sa_k)$ as $\sa_k \tau_k$.  By definition, we have that $t_{k,k+1}$ acts as $\rho_{k,k+1} \ssb_k \tau_k$ on $\sL_1(\sa_1) \boxtimes \dotsb \boxtimes (\sL_{\ell_k} \irtimes \sL_{\ell_k}) \boxtimes \dotsb \boxtimes \sL_N(\sa_N)$, with the understanding that $\tau_k$ acts on the $k$-th factor.  If $\sL_{\ell_k}$ is of type $\tM$, then $\ssb_k = \sa_k$, $\sL_{\ell_k} \irtimes \sL_{\ell_k} = \sL_{\ell_k} \boxtimes \sL_{\ell_k}$, $\tilde{t}_{k,k+1}=1$, and the result is clear.

  Now suppose that $L_{\ell_k}$ is of type $\tQ$.  Fix an isomorphism $\END_F L_{\ell_k} \cong \Cl$, which we will view as equality.  Then the isomorphism $\END_F L_{\ell_k} \cong \END_F \left( {}^\psi L_{\ell_k} \right)$, $z \mapsto \tau_k z \tau_k^{-1}$, gives us an isomorphism $\END_F \left( {}^\psi L_{\ell_k} \right) \cong \Cl$, which we also view as equality.  Thus we have actions of $\Cl$ on $L_{\ell_k}$ and ${}^\psi L_{\ell_k}$ that commute with $\tau_k$.  Note that $\tilde{t}_{k,k+1} = 1 + c \otimes c \in \END_F L_{\ell_k}$ is independent of these choices of isomorphisms.

  Consider the idempotent
  \[
    \epsilon = \frac{1}{2}(1 + \sqrt{-1} c \otimes c) \in \left( \END_F L_{\ell_k} \right) \otimes \left( \END_F L_{\ell_k} \right).
  \]
  Then we have a decomposition of $F$-modules
  \[
    L_{\ell_k} \boxtimes L_{\ell_k}
    = \epsilon \left( L_{\ell_k} \boxtimes L_{\ell_k} \right)
    \oplus (1-\epsilon) \left( L_{\ell_k} \boxtimes L_{\ell_k} \right).
  \]
  We can take $L_{\ell_k} \irtimes L_{\ell_k} = \epsilon \left( L_{\ell_k} \boxtimes L_{\ell_k} \right)$.  Note that
  \[
    (c \otimes 1) \epsilon
    = (1-\epsilon) (c \otimes 1)
  \]
  In other words, $c \otimes 1$ is an odd automorphism of $L_{\ell_k} \boxtimes L_{\ell_k}$ that interchanges $L_{\ell_k} \irtimes L_{\ell_k}$ and
  \[
    (1-\epsilon) \left( L_{\ell_k} \boxtimes L_{\ell_k} \right) \cong \Pi \left( L_{\ell_k} \irtimes L_{\ell_k} \right).
  \]
  Then we have
  \begin{align*}
    t_{k,k+1}
    &= t_{k,k+1} \left( \epsilon + (1-\epsilon) \right) \\
    &= t_{k,k+1} \epsilon + t_{k,k+1} (c \otimes 1) \epsilon (c \otimes 1) \\
    &= t_{k,k+1} \epsilon + (c \otimes 1) \left( t_{k,k+1} \epsilon \right) (c \otimes 1) \\
    &= \rho_{k,k+1} \ssb_k \tau_k \epsilon + (c \otimes 1) \rho_{k,k+1} \ssb_k \tau_k \epsilon (c \otimes 1) \\
    &= (1-\epsilon) \rho_{k,k+1} \ssb_k \tau_k + (c \otimes 1) (1-\epsilon) (1 \otimes c) \rho_{k,k+1} \ssb_k \tau_k \\
    &= \frac{1-\sqrt{-1}}{2} (1 + c \otimes c) \rho_{k,k+1} \ssb_k \tau_k \\
    &= \tilde{t}_{k,k+1} \rho_{k,k+1} \sa_k \tau_k,
  \end{align*}
  where $\epsilon$, $1 \otimes c$, etc.\ lie in the $k$-th and $(k+1)$-st factors (we avoid writing $1^{\otimes (k-1)} \otimes \epsilon \otimes 1^{\otimes (n-k-1)}$, etc.\ to simplify notation).
}

We let
\begin{equation} \label{eq:Emu-def}
  \sE(\mu)
  = \END_{F^{\otimes n}}^\op \sL(\mu)
  \simeq \left( \END_F^\op \sL_1(\sa_1) \right)^{\otimes \mu_1} \otimes \dotsb \otimes \left( \END_F^\op \sL_N(\sa_N) \right)^{\otimes \mu_N},
\end{equation}
so that $\sL(\mu)$ is naturally a right $\sE(\mu)$-module via the action
\[
  vz = (-1)^{\bar v \bar z} zv,\quad z \in \sE(\mu),\ v \in \sL(\mu),
\]
where, on the right side, we view $z$ as an element of $\END_{F^{\otimes n}} \sL(\mu)$.

Note that, for $1 \le k \le N$,
\[
  \END_F \sL_k(\sa_k) \simeq \left(\END_F \sL_k \right)^{\oplus r_k}.
\]
\details{
  We use here the fact that $\HOM_F \left( \prescript{\ell}{}{\sL}_k, \prescript{\ell'}{}{\sL}_k \right) = 0$ unless $\ell = \ell'$, in which case it is evenly isomorphic to $\END_F \sL_k$.
}
Now, for $1 \le k \le N$, we have a natural right action of $\kk[y]$ on $\END_F^\op \sL_k(\sa_k)$ given by
\begin{equation} \label{eq:xEn-action}
  z \cdot y = x^{-1} z x,\quad z \in \END_F^\op \sL_k(\sa_k),
\end{equation}
where, on the right side, $x$ denotes the endomorphism of $\sL_k(\sa_k)$ given by the action of $x$ and, since this action is invertible, $x^{-1}$ denotes its inverse.  Thus we can form the smash product $\kk[y] \ltimes \END_F^\op \sL_k(\sa_k)$.

For $1 \le \ell \le N$ and $n \in \N$, define
\begin{equation} \label{eq:H_n^ell-def}
  \cH_n^\ell =
  \begin{cases}
    \left( \kk[y] \ltimes \END_F^\op \sL_\ell (\sa_\ell) \right)^{\otimes n} \rtimes_\rho S_n & \text{if $t_{1,2}$ acts as zero on $L_\ell \irtimes L_\ell$}, \\
    \cA_n \left( \END_F^\op \sL_\ell \right) & \text{if $t_{1,2}$ does not act as zero on $L_\ell \irtimes L_\ell$}.
  \end{cases}
\end{equation}
(Recall that we use the notation $\rtimes_\rho$ to denote a smash product with $\kk S_n$, where the action is via superpermutations.)  Note that $\cH_n^\ell$ depends on $F$ and our ordering of the simple $F$-modules.  It is also important to note that $r_\ell = 1$ whenever $t_{1,2}$ does not act as zero on $\sL_\ell \irtimes \sL_\ell$.  In this case, we have
\[
  \END_F^\op \sL_\ell = \END_F^\op \sL_\ell(\sa_\ell) \simeq
  \begin{cases}
    \kk & \text{if $\sL_\ell$ is of type $\tM$}, \\
    \Cl & \text{if $\sL_\ell$ is of type $\tQ$}.
  \end{cases}
\]

Define
\begin{equation} \label{eq:cR_n-def}
  \cR_n = \bigoplus_{\mu \in \cC_n} \cR_\mu,\quad
  \cR_\mu = \cH_{\mu_1}^1 \otimes \dotsb \otimes \cH_{\mu_N}^N.
\end{equation}
We will denote the polynomial generators in $\cR_n$ by $y_1,\dotsc,y_n$\label{y_i-def} to avoid confusion with the generators $x_1,\dotsc,x_n \in \cA_n(F)$.  Note that $\sE(\mu)$ can naturally be viewed as a subalgebra of $\cR_\mu$.  Then $\cR_\mu$ is generated as an algebra by $\sE(\mu)$, $y_1,\dotsc,y_n$, and $S_\mu$.

%----------------------------------------
\subsection{An equivalence of categories}
%----------------------------------------

If $A_1$ and $A_2$ are algebras, $V$ is an $(A_1,A_2)$-bimodule, and $W$ is an $A_1$-module, then $\HOM_{A_1}(V,W)$ is naturally a (left) $A_2$-module under the action
\[
  (a \alpha)(v) = (-1)^{\bar a (\bar \alpha + \bar v)} \alpha(va),\quad a \in A_2,\ \alpha \in \HOM_{A_1}(V,W),\ v \in V.
\]
\details{
  We have
  \[
    \big( (a_1 a_2)(\alpha) \big)(v)
    = (-1)^{(\bar a_1 + \bar a_2)(\bar \alpha + \bar v)} \alpha(v a_1 a_2)
    = (-1)^{\bar a_1(\bar \alpha + \bar v + \bar a_2)} (a_2 \alpha)(v a_1)
    = \big(a_1(a_2\alpha))(v).
  \]
}

\begin{lem} \label{lem:category-equivalence}
  Suppose $A$ is a graded superalgebra and $V$ is a finite-dimensional semisimple $A$-module.  Let $\mathscr{C}_V$ denote the full subcategory of $A\md$ whose objects are evenly isomorphic to finite direct sums of degree shifts of simple submodules of $V$.  Then, viewing $V$ as an $(A,\END_A^\op V)$-bimodule, the functors
  \[
    V \otimes_{\END_A^\op V} - \colon (\END_A^\op V)\md \to \mathscr{C}_V,\quad
    \HOM_A(V,-) \colon \mathscr{C}_V \to (\END_A^\op V)\md,
  \]
  yield an equivalence of categories between $\mathscr{C}_V$ and $(\END_A^\op V)\md$.
\end{lem}

\begin{proof}
  Suppose $U \in (\END_A^\op V)\md$.  Then
  \[
    \HOM_A(V, V \otimes_{\END_A^\op V} U)
    \simeq (\END_A V) \otimes_{\END_A^\op V} U
    \simeq U,
  \]
  and these isomorphisms are natural in $U$.  (We use the fact that, for any algebra $R$, $R \simeq R^\op$ as $(R,R)$-bimodules.)
  \details{
    An algebra $R$ is naturally an $(R,R)$-bimodule, with the left and right actions given by multiplication on the left and right, respectively.  Also, $R^\op$ is an $(R,R)$-bimodule, with the left and right actions given by signed multiplication on the right and left, respectively. Indeed, it is a left $R$-module since
    \begin{multline*}
      a_1 \heartsuit (a_2 \heartsuit a)
      = (-1)^{\bar a_2 \bar a + \bar a_1 \bar a + \bar a_1 \bar a_2} (a \cdot a_2) \cdot a_1
      = a_1 a_2 a
      \\
      = (-1)^{\bar a (\bar a_1 + \bar a_2)} a \cdot (a_1 a_2)
      = (a_1 a_2) \heartsuit a,\quad
      a \in R^\op,\ a_1, a_2 \in R,
    \end{multline*}
    where $\heartsuit$ denotes this left action of $R$ on $R^\op$, $\cdot$ denotes the multiplication in $R^\op$, and juxtaposition denotes the multiplication in $R$.  It is a right $R$-module since
    \[
      (a \spadesuit a_1) \spadesuit a_2
      = (-1)^{\bar a_2 \bar a + \bar a_1 \bar a + \bar a_1 \bar a_2} a_2 \cdot (a_1 \cdot a)
      = a a_1 a_2
      = (-1)^{\bar a (\bar a_1 + \bar a_2)} (a_1 a_2) \cdot a
      = a \spadesuit (a_1 a_2),
    \]
    where $\spadesuit$ denotes the right action of $R$ on $R^\op$ described above.  It is clear that the right action $\heartsuit$ commutes with the left action $\spadesuit$.

    With these definitions, the identity map $\phi$ from $R$ to $R^\op$ is an even isomorphism of $(R,R)$-bimodules.  Indeed, it is clearly invertible, and
    \[
      a_1 \heartsuit \phi(a) \spadesuit a_2
      = (-1)^{\bar a_1 \bar a + \bar a_1 \bar a_2 + \bar a \bar a_2} a_2 \cdot a \cdot a_1
      = a_1 a a_2
      = \phi(a_1 a a_2).
    \]
  }

  On the other hand, for $W \in \mathscr{C}_V$, we have an even homomorphism of $A$-modules
  \[
    V \otimes_{\END_A^\op V} \HOM_A(V,W) \to W,\quad
    v \otimes \phi \mapsto \phi(v).
  \]
  This homomorphism is surjective by the definition of $\mathscr{C}_V$ and Schur's Lemma.  It is then injective since its domain and codomain have the same dimension by the double centralizer property.
  \details{
    Replacing $A$ by $A/\Ann_A V$, we may assume that $A$ is finite-dimensional, and that its action on $V$ is faithful.  Then we have
    \[
      V \otimes_{\END_A^\op V} \HOM_A(V,W)
      \simeq V \otimes_{\END_A^\op V} V^* \otimes_A W
      \simeq A \otimes_A W
      \simeq A.
    \]
  }
  It is clearly natural in $W$.
\end{proof}

For $V \in \cA_n(F)\smod$ and $\mu \in \cC_n$, let $I_\mu V$\label{ImuV-def} be the sum of all $F^{\otimes n}$-submodules of $V$ evenly isomorphic to degree shifts of submodules of $\sL(\mu)$.  Then define
\begin{equation} \label{eq:Vmu-def}
  V_\mu := \sum_{\pi \in S_n} \pi(I_\mu V).
\end{equation}

\begin{lem} \label{lem:isotypic}
  Suppose $\mu \in \cC_n$ and $V \in \cA_n(F)\smod$.  Then $I_\mu V$ is an $\cA_\mu(F)$-submodule of $V$ and $V_\mu$ is an $\cA_n(F)$-submodule of $V$.  Furthermore, $V_\mu \simeq \Ind_\mu^n (I_\mu V)$.
\end{lem}

\begin{proof}
  As discussed in Section~\ref{subsec:kxF-modules}, it follows from \eqref{rel:xF-commutation} that $x_i$ maps a simple $F^{\otimes n}$-submodule $W$ of $V$ either to zero or to a submodule evenly isomorphic to $\prescript{\psi_i}{}{W}$.  Hence, it follows from the definition of $L(\mu)$ that $I_\mu V$ is invariant under the action of $P_n(F)$.  It also follows from \eqref{rel:SF-commutation} that $\pi \in S_\mu$ maps a simple $F^{\otimes n}$-submodule of $V$ evenly isomorphic to degree shift of a simple summand of $\sL(\mu)$ to one evenly isomorphic to another degree shift of a simple summand of $\sL(\mu)$.  Thus $I_\mu V$ is an $\cA_\mu(F)$-submodule of $V$.  It then follows immediately from the definition that $V_\mu$ is an $\cA_n(F)$-submodule of $V$.

  If $I_\mu V = 0$, the isomorphism asserted in the lemma is trivially true.  So we assume that $I_\mu V \ne 0$.  Then the inclusion $I_\mu V \hookrightarrow V_\mu$ is a nonzero even homomorphism of $\cA_\mu(F)$-modules.  Since induction is left adjoint to restriction, we have a nonzero even homomorphism $\phi \colon \Ind_\mu^n (I_\mu V) \to V_\mu$ of $\cA_n(F)$-modules.  It follows from the definition of $V_\mu$ that $\phi$ is surjective.  Furthermore, if $X$ is a complete set of representatives of left cosets of $S_\mu$ in $S_n$, we have
  \[
    \Ind_\mu^n I_\mu V
    = \cA_n(F) \otimes_{\cA_\mu(F)} I_\mu V
    = \bigoplus_{\pi \in X} \pi \otimes I_\mu V
    \qquad \text{and} \qquad
    V_\mu = \bigoplus_{\pi \in X} \pi(I_\mu V).
  \]
  Therefore, $\Ind_\mu^n (I_\mu V)$ and $V_\mu$ have the same dimension, and so $\phi$ is injective.
\end{proof}

\begin{lem} \label{lem:dHn-mod-cC-decomp}
  For $V \in \cA_n(F)\smod$, we have the decomposition $V = \bigoplus_{\mu \in \cC_n} V_\mu$ in $\cA_n(F)\smod$.
\end{lem}

\begin{proof}
  Let $V \in \cA_n(F)\smod$.  By definition, $V$ is semisimple as an $F^{\otimes n}$-module.  For $\mu \in \cC_n$, $V_\mu$ is the direct sum of the isotypic components corresponding to simple $F^{\otimes n}$-modules containing, for each $1 \le i \le N$, exactly $\mu_i$ tensor factors evenly isomorphic to degree shifts of $\prescript{\ell}{}{\sL}_i$ for some $0 \le \ell \le r_i$.  The lemma follows.
\end{proof}

Note that $S_\mu$ acts on $\sL(\mu)$ by superpermuting the factors.  For $\pi \in S_\mu$, we will denote this action on $\sL(\mu)$ by $\rho_\pi$, to avoid confusion with the action of $\pi$ on $\cA_n(F)$-modules.

\begin{prop} \label{prop:DH-to-A-functor}
  Suppose $\mu \in \cC_n$ and $V \in \cA_n(F)\smod$.  Then $\HOM_{F^{\otimes n}} (\sL(\mu),V)$ is an $\cR_\mu$-module under the action
  \begin{gather*}
    \pi \diamond \phi = \pi \phi \rho_\pi^{-1},\quad \pi \in S_\mu, \\
    y_k \diamond \phi = x_k \phi x_k^{-1},\quad 1 \le k \le n, \\
    z \diamond \phi = (-1)^{\bar \phi \bar z} \phi z,\quad z \in \sE(\mu),
  \end{gather*}
  for $\phi \in \HOM_{F^{\otimes n}} (\sL(\mu),V)$.  Thus, we have a functor
  \[
    \HOM_{F^{\otimes n}}(\sL(\mu),-) \colon \cA_n(F)\smod \to \cR_\mu\md.
  \]
\end{prop}

\begin{proof}
  It is straightforward to verify that $\pi \diamond \phi$, $y_k \diamond \phi$, and $z \diamond \phi$ are $F^{\otimes n}$-module homomorphisms.
  \details{
    Suppose $v \in \sL(\mu)$ and $\bsf \in F^{\otimes n}$.  For $\pi \in S_\mu$, we have
    \begin{multline*}
      (\pi \diamond \phi)(\bsf v)
      = \pi \phi \rho_\pi^{-1} (\bsf v)
      = \pi \phi (\rho_\pi^{-1} \bsf) (\rho_\pi^{-1} v)
      \\
      = (-1)^{\bar \phi \bar \bsf} \pi (\rho_\pi^{-1} \bsf) \phi (\rho_\pi^{-1} v)
      = (-1)^{\bar \phi \bar \bsf} \bsf \pi \phi \rho_\pi^{-1} (v)
      = (-1)^{\bar \phi \bar \bsf} \bsf (\pi \diamond \phi)(v).
    \end{multline*}
    Thus, $\pi \diamond \phi$ is an $F^{\otimes n}$-module homomorphism.  For $1 \le k \le n$, we also have
    \begin{multline*}
      (y_k \diamond \phi)(\bsf v)
      = x_k \phi x_k^{-1} (\bsf v)
      = x_k \phi \psi_k(\bsf) x_k^{-1} (v)
      \\
      = (-1)^{\bar \bsf \bar \phi} x_k \psi_k(\bsf) \phi x_k^{-1} (v)
      = (-1)^{\bar \bsf \bar \phi} \bsf x_k \phi x_k^{-1} (v)
      = \bsf (y_k \diamond \phi)(v).
    \end{multline*}
    Thus $y_k \diamond \phi$ is an $F^{\otimes n}$-module homomorphism.  Finally, for $z \in \sE(\mu)$,
    \[
      (z \diamond \phi)(\bsf v)
      = (-1)^{\bar \phi \bar z} \phi z \bsf v
      = (-1)^{\bar \phi \bar z + (\bar \phi + \bar z) \bar \bsf} \bsf \phi z v
      = (-1)^{(\bar \phi + \bar z) \bar \bsf} \bsf (z \diamond \phi)(v).
    \]
    Thus $z \diamond \phi$ is an $F^{\otimes n}$-module homomorphism.
  }

  Now suppose $1 \le k \le n-1$ such that $\ell_k = \ell_{k+1}$ (i.e.\ such that $s_k \in S_\mu$).  Then we have
  \begin{gather*}
    (s_k y_k) \diamond \phi
    = s_k x_k \phi x_k^{-1} \rho_{k,k+1}
    \quad \text{and} \\
    (y_{k+1} s_k) \diamond \phi
    = x_{k+1} s_k \phi \rho_{k,k+1} x_{k+1}^{-1}
    = x_{k+1} s_k \phi x_k^{-1} \rho_{k,k+1}.
  \end{gather*}
  Thus
  \begin{multline*}
    (s_k y_k - y_{k+1} s_k) \diamond \phi
    = (s_k x_k - x_{k+1} s_k) \phi x_k^{-1} \rho_{k,k+1}
    \stackrel{\eqref{rel:sx-commutation}}{=} - t_{k,k+1} \phi x_k^{-1} \rho_{k,k+1}
    \\
    = - \phi t_{k,k+1} x_k^{-1} \rho_{k,k+1}
    \stackrel{\eqref{eq:t-L-hat-n-action}}{=}
    \begin{cases}
      0 & \text{if $t_{1,2}$ acts as zero on $\sL_{\ell_k} \irtimes \sL_{\ell_k}$}, \\
      - \phi \tilde{t}_{k,k+1} & \text{if $t_{1,2}$ does not act as zero on $\sL_{\ell_k} \irtimes \sL_{\ell_k}$}.
    \end{cases}
  \end{multline*}

  For $z \in \sE(\mu)$, we also have
  \[
    (z y_k) \diamond \phi
    = (-1)^{\bar \phi \bar z} x_k \phi x_k^{-1} z
    \stackrel{\eqref{eq:xEn-action}}{=} (-1)^{\bar \phi \bar z} x_k \phi (z \cdot y_k) x_k^{-1}
    = (y_k (z \cdot y_k)) \diamond \phi.
  \]
  The remainder of the relations, involving only elements of $S_\mu$, only elements of $\sE(\mu)$, or only the $y_k$, are straightforward to verify.
\end{proof}

\begin{prop} \label{prop:A-to-DF-functor}
  Suppose $M$ is an $\cR_\mu$-module.  Then $\sL(\mu) \otimes_{\sE(\mu)} M$ is an $\cA_\mu(F)$-module under the action
  \begin{gather*}
    \bsf * (w \otimes v) = \bsf w \otimes v,\quad \bsf \in F^{\otimes n}, \\
    \pi * (w \otimes v) = \rho_\pi w \otimes \pi v,\quad \pi \in S_\mu, \\
    x_k * (w \otimes v) = x_k w \otimes y_k v,
  \end{gather*}
  for $w \in \sL(\mu)$, $v \in M$.
\end{prop}

\begin{proof}
  It is straightforward to verify that the given actions of $\bsf$, $\pi$, and $x_k$ are well-defined on the tensor product, that is, that they are balanced with respect to the tensor product over $\sE(\mu)$.
  \details{
    Suppose $\bsf \in F^{\otimes n}$, $w \in \sL(\mu)$, $v \in M$, $\pi \in S_\mu$, and $z \in \sE(\mu)$.  Then
    \[
      \bsf * (w z \otimes v)
      = \bsf wz \otimes v
      = \bsf w \otimes zv
      = \bsf * (w \otimes v),
    \]
    and
    \begin{multline*}
      \pi * (wz \otimes v)
      = \rho_\pi (wz) \otimes \pi v
      = \rho_\pi (w) \rho_\pi (z) \otimes \pi v
      \\
      = \rho_\pi (w) \otimes \rho_\pi(z) \pi v
      = \rho_\pi (w) \otimes \pi z v
      = \pi * (w \otimes zv),
    \end{multline*}
    and
    \begin{multline*}
      x_k * (wz \otimes v)
      = (-1)^{\bar w \bar z} x_k * (zw \otimes v)
      = (-1)^{\bar w \bar z} x_k z w \otimes y_k v
      \stackrel{\eqref{eq:xEn-action}}{=} (-1)^{\bar w \bar z} (z \cdot y_k^{-1}) x_k w \otimes y_k v
      \\
      = x_k w (z \cdot y_k^{-1}) \otimes y_k v
      = (x_k w) \otimes (z \cdot y_k^{-1}) y_k v
      \stackrel{\eqref{eq:xEn-action}}{=} (x_k w) \otimes y_k z v
      = x_k * (w \otimes zv),
    \end{multline*}
    where we use the fact that the action of $y_k$ on $\sE(\mu)$ is invertible, and use the notation $y_k^{-1}$ to denote the inverse of the action by $y_k$.
  }

  The relations involving only elements of $F^{\otimes n}$, only elements of $S_\mu$, or only the $x_k$ are clear.  If $1 \le k \le n$, $\bsf \in F$, $w \in \sL(\mu)$, and $v \in M$, then
  \[
    (\bsf x_k) * (w \otimes v)
    = \bsf x_k w \otimes y_k v
    = x_k \psi_k(\bsf) w \otimes y_k v
    = (x_k \psi_k(\bsf)) * (w \otimes v).
  \]
  If $1 \le k \le n-1$ and $\ell_k = \ell_{k+1}$ (i.e.\ $s_k \in S_\mu$), we also have
  \begin{gather*}
    (s_k x_k) * (w \otimes v)
    = \rho_{k,k+1} x_k w \otimes s_k y_k v,
    \\
    (x_{k+1} s_k) * (w \otimes v)
    = x_{k+1} \rho_{k,k+1} w \otimes y_{k+1} s_k v
    = \rho_{k,k+1} x_k w \otimes y_{k+1} s_k v.
  \end{gather*}
  So we have
  \begin{multline*}
    (s_k x_k - x_{k+1} s_k) * (w \otimes v)
    = \rho_{k,k+1} x_k w \otimes (s_k y_k - y_{k+1} s_k)v
    \\
    = - \rho_{k,k+1} x_k w \otimes \tilde{t}_{k,k+1} v
    \stackrel{\eqref{eq:t-L-hat-n-action}}{=} -t_{k,k+1} * (w \otimes v). \qedhere
  \end{multline*}
\end{proof}

\begin{prop} \label{prop:Phi-isom}
  Let $M$ be an $\cR_\mu$-module.  Then
  \[
    \Phi \colon M \to \HOM_{F^{\otimes n}} \left( \sL(\mu), \sL(\mu) \otimes_{\sE(\mu)} M \right),\quad \Phi(v)(w) = (-1)^{\bar v \bar w} w \otimes v,\ v \in M,\ w \in \sL(\mu),
  \]
  is an even isomorphism of $\cR_\mu$-modules.  Furthermore, $\sL(\mu) \otimes_{\sE(\mu)} M$ is a simple $\cA_\mu(F)$-module if and only if $M$ is a simple $\cR_\mu$-module.
\end{prop}

\begin{proof}
  It is clear that $\Phi(v) \in \HOM_{F^{\otimes n}}(\sL(\mu), \sL(\mu) \otimes_{\sE(\mu)} M)$ for all $v \in M$.  It is also straightforward to verify that $\Phi$ is a homomorphism of $\cR_\mu$-modules.
  \details{
    Let $v \in M$.  For $\pi \in S_\mu$, $1 \le k \le n$, and $z \in \sE(\mu)$, we have
    \begin{gather*}
      \Phi(\pi v)
      = (w \mapsto (-1)^{\bar v \bar w} w \otimes \pi v)
      = \left( w \mapsto (-1)^{\bar v \bar w} \pi * (\rho_\pi^{-1} w \otimes v) \right)
      = \pi \diamond \Phi(v),
      \\
      \Phi(y_k v)
      = (w \mapsto (-1)^{\bar v \bar w} w \otimes y_k v)
      = \left( w \mapsto (-1)^{\bar v \bar w} x_k * (x_k^{-1} w \otimes v) \right)
      = y_k \diamond \Phi(v),
      \\
      \Phi(z v)
      = \left( w \mapsto (-1)^{(\bar z + \bar v) \bar w} w \otimes z v = (-1)^{\bar v \bar w} zw \otimes v \right)
      = (-1)^{\bar z \bar v} (w \mapsto (-1)^{\bar v \bar w} w \otimes v) \circ z
      = z \diamond \Phi(v)
    \end{gather*}
  }
  By Lemma~\ref{lem:category-equivalence}, $\Phi$ is an isomorphism of $\sE(\mu)$-modules.  Therefore, it is bijective and hence an isomorphism of $\cR_\mu$-modules.

  Suppose that $\sL(\mu) \otimes_{\sE(\mu)} M$ is a simple $\cA_\mu(F)$-module and $W$ is a nonzero $\cR_\mu$-submodule of $M$.  Then $\sL(\mu) \otimes_{\sE(\mu)} W$ is an $\cA_\mu(F)$-submodule of $\sL(\mu) \otimes_{\sE(\mu)} M$.  Hence $W = M$.  So $M$ is simple.

  Conversely, suppose that $M$ is a simple $\cR_\mu$-module and that $W$ is a nonzero $\cA_\mu(F)$-submodule of $\sL(\mu) \otimes_{\sE(\mu)} M$.  Then, by Proposition~\ref{prop:DH-to-A-functor}, $\HOM_{F^{\otimes n}} (\sL(\mu), W)$ is a nonzero $\cR_\mu$-submodule of $\HOM_{F^{\otimes n}} (\sL(\mu), \sL(\mu) \otimes_{\sE(\mu)} M) \simeq M$, which is simple.  Hence $\HOM_{F^{\otimes n}} (\sL(\mu), W) \simeq M$.  Then, as an $F^{\otimes n}$-module, $W \simeq \sL(\mu) \otimes_{\sE(\mu)} M$ by Lemma~\ref{lem:category-equivalence}.  So $\sL(\mu) \otimes_{\sE(\mu)} M$ is simple.
\end{proof}

\begin{prop} \label{prop:Upsilon-isom}
  Suppose $V \in \cA_n(F)\smod$.  Then
  \[
    \Upsilon \colon \sL(\mu) \otimes_{\sE(\mu)} \HOM_{F^{\otimes n}} (\sL(\mu), I_\mu V) \to I_\mu V,\quad
    v \otimes \varphi \mapsto (-1)^{\bar v \bar \varphi} \varphi(v),
  \]
  defines an even isomorphism of $\cA_\mu(F)$-modules.
\end{prop}

\begin{proof}
  By Lemma~\ref{lem:isotypic}, $I_\mu V$ is an $\cA_\mu(F)$-module.  Then, by Propositions~\ref{prop:DH-to-A-functor} and~\ref{prop:A-to-DF-functor}, \break $\sL(\mu) \otimes_{\sE(\mu)} \HOM_{F^{\otimes n}} (\sL(\mu), I_\mu V)$ is an $\cA_n(F)$-module.
  \details{
    We verify that $\Upsilon$ is balanced (i.e.\ respects the tensor product over $\sE(\mu)$).  For $v \in \sL(\mu)$, $\varphi \in \HOM_{F^{\otimes n}}(\sL(\mu), I_\mu V)$, and $z \in \sE(\mu)$, we have
    \[
      \Upsilon(vz \otimes \varphi)
      = (-1)^{\bar v \bar z} \Upsilon(zv \otimes \varphi)
      = (-1)^{\bar v \bar z + (\bar v + \bar z) \bar \varphi} \varphi(zv)
      = (-1)^{\bar v (\bar z + \bar \varphi)} (z \diamond \varphi)(v)
      = \Upsilon \left(v \otimes (z \diamond \varphi) \right).
    \]
  }
  It is straightforward to verify that $\Upsilon$ is a homomorphism of $\cA_\mu(F)$-modules.
  \details{
    Let $v \in \sL(\mu)$ and $\varphi \in \HOM_{F^{\otimes n}} (\sL(\mu), I_\mu V)$.  For $\bsf \in F^{\otimes n}$, $\pi \in S_\mu$, and $1 \le k \le n$, we have
    \begin{gather*}
      \Upsilon(\bsf * (v \otimes \varphi))
      = \Upsilon(\bsf v \otimes \varphi)
      = (-1)^{(\bar \bsf + \bar v)\bar \varphi} \varphi(\bsf v)
      = (-1)^{\bar v \bar \varphi} \bsf \varphi(v)
      = \bsf \Upsilon(v \otimes \varphi),
      \\
      \Upsilon(\pi * (v \otimes \varphi))
      = \Upsilon(\rho_\pi v \otimes \pi \diamond \varphi)
      = \Upsilon \left( \rho_\pi v \otimes \pi \varphi \rho_\pi^{-1} \right)
      = (-1)^{\bar v \bar \varphi} \pi \varphi(v)
      = \pi \Upsilon(v),
      \\
      \Upsilon(x_k * (v \otimes \varphi))
      = \Upsilon(x_k v \otimes y_k \diamond \varphi)
      = \Upsilon \left( x_k v \otimes x_k \varphi x_k^{-1} \right)
      = (-1)^{\bar v \bar \varphi} x_k \varphi(v)
      = x_k \Upsilon(v \otimes \varphi).
    \end{gather*}
  }
  Now, as an $F^{\otimes n}$-module, $I_\mu V$ is isomorphic to a finite direct sum of modules evenly isomorphic to degree shifts of simple summands of $\sL(\mu)$.  It follows from Lemma~\ref{lem:category-equivalence} that $\Upsilon$ is bijective.
\end{proof}

The following theorem is a generalization of \cite[Th.~3.9]{WW08}, which treats the case where $F$ is the group algebra of a finite group (see Example~\ref{eg:wreath-Hecke}).

\begin{theo} \label{theo:category-equivalence}
  The functor $\mathbf{F} \colon \cA_n(F)\smod \to \cR_n\md$ defined by
  \[
    \mathbf{F}(V)
    = \bigoplus_{\mu \in \cC_n} \HOM_{F^{\otimes n}} \left( \sL(\mu), I_\mu V \right)
  \]
  is an equivalence of categories with inverse  $\mathbf{G} \colon \cR_n\md \to \cA_n(F)\smod$ given by
  \[
    \mathbf{G} \left( \bigoplus_{\mu \in \cC_n} M_\mu \right)
    = \bigoplus_{\mu \in \cC_n} \Ind_\mu^n \left( \sL(\mu) \otimes_{\sE(\mu)} M_\mu \right).
  \]
\end{theo}

\begin{proof}
  The map $\Phi$ of Proposition~\ref{prop:Phi-isom} is natural in $M$, and the map $\Upsilon$ of Proposition~\ref{prop:Upsilon-isom} is natural in $V$.  Then, using Lemmas~\ref{lem:isotypic} and~\ref{lem:dHn-mod-cC-decomp} and Propositions~\ref{prop:DH-to-A-functor}--\ref{prop:Upsilon-isom}, it is straightforward to verify that $\mathbf{F} \mathbf{G} \cong \id$ and $\mathbf{G} \mathbf{F} \cong \id$.
  \details{
    We have
    \begin{align*}
      \mathbf{F} \mathbf{G} \left( \bigoplus_{\mu \in \cC_n} M_\mu \right)
      &= \bigoplus_{\mu \in \cC_n} \HOM_{F^{\otimes n}} \left( \sL(\mu), I_\mu \left( \bigoplus_{\mu \in \cC_n} \Ind_\mu^n \left( \sL(\mu) \otimes_{\sE(\mu)} M_\mu \right) \right) \right) \\
      &\simeq \bigoplus_{\mu \in \cC_n} \HOM_{F^{\otimes n}} \left( \sL(\mu), \sL(\mu) \otimes_{\sE(\mu)} M_\mu \right) \\
      &\simeq \bigoplus_{\mu \in \cC_n} M_\mu,
    \end{align*}
    where the first even isomorphism follows from Lemma~\ref{lem:isotypic} and Lemma~\ref{lem:dHn-mod-cC-decomp}, and the last even isomorphism follows from Proposition~\ref{prop:Phi-isom}.

    On the other hand, we have
    \begin{align*}
      \mathbf{G} \mathbf{F} (V)
      &= \bigoplus_{\mu \in \cC_n} \Ind_\mu^n \left( \sL(\mu) \otimes_{\sE(\mu)} \HOM_{F^{\otimes n}} \left( \sL(\mu), I_\mu V \right) \right) \\
      &\simeq \bigoplus_{\mu \in \cC_n} \Ind_\mu^n ( I_\mu V ) \\
      &\simeq V,
    \end{align*}
    where the first even isomorphism follows from Proposition~\ref{prop:Upsilon-isom}, and the second even isomorphism follows from Lemma~\ref{lem:isotypic} and Lemma~\ref{lem:dHn-mod-cC-decomp}.
  }
\end{proof}

\begin{rem}
  Note that if $F$ is semisimple, then $\cA_n(F)\md = \cA_n(F)\smod$ and thus Theorem~\ref{theo:category-equivalence} implies that $\cA_n(F)$ is Morita equivalent to $\cR_n$.
\end{rem}

%-------------------------------------
\subsection{Simple $\cA_n(F)$-modules}
%-------------------------------------

We can now classify the simple $\cA_n(F)$-modules.

\begin{prop} \label{prop:simples-are-Fn-semisimple}
  Every simple $\cA_n(F)$-module is semisimple as an $F^{\otimes n}$-module.  In particular, the categories $\cA_n(F)\smod$ and $\cA_n(F)\md$ have the same class of simple modules.
\end{prop}

\begin{proof}
  Suppose $V$ is a simple $\cA_n(F)$-module.  Let $W$ be a simple $P_n(F)$-submodule of $V$.  Then $\sum_{\pi \in S_n} \pi W$ is an $\cA_n(F)$-submodule of $V$ (e.g., by Lemma~\ref{lem:Sn-commutation-mod-lower-terms}) and hence $V = \sum_{\pi \in S_n} \pi W$ since $V$ is simple.  By Proposition~\ref{prop:kxF-simple-modules}, $W \simeq L_1(a_1) \irtimes \dotsb \irtimes L_n(a_n)$, where $L_1,\dotsc,L_n$ are simple $F$-modules and $a_1,\dotsc,a_n \in \kk$.  Thus each $\pi W$ is semisimple as an $F^{\otimes n}$-module, and so $V$ is semisimple as an $F^{\otimes n}$-module.
\end{proof}

The following theorem is a generalization of \cite[Th.~4.4]{WW08}, which treats the case where $F$ is the group algebra of a finite group (see Example~\ref{eg:wreath-Hecke}).

\begin{theo} \label{theo:simple-modules}
  Every simple $\cA_n(F)$-module is evenly isomorphic to a module of the form
  \begin{equation} \label{eq:D-hatn-def}
    \Ind_\mu^n \left( \sL(\mu) \otimes_{\sE(\mu)} (V_1 \irtimes \dotsb \irtimes V_N) \right),
  \end{equation}
  where $\mu = (\mu_1,\dotsc,\mu_N) \in \cC_n$, and $V_k$ is a simple $\cH_{\mu_\ell}^\ell$-module for $1 \le \ell \le N$.  Furthermore, the above modules (over all $\mu \in \cC_n$ and $V_\ell$, $1 \le \ell \le N$, ranging over a set of representatives of even isomorphism classes, up to degree shift) form a complete set of pairwise not evenly-isomorphic simple $\cA_n(F)$-modules, up to degree shift.
\end{theo}

\begin{proof}
  This follows immediately from Proposition~\ref{prop:simples-are-Fn-semisimple} and Theorem~\ref{theo:category-equivalence}.
\end{proof}

\begin{rem}
  Theorem~\ref{theo:simple-modules} reduces the study of simple $\cA_n(F)$-modules to the study of simple modules for the $\cH_m^\ell$.  If $t_{1,2}$ does not act as zero on $\sL_\ell \irtimes \sL_\ell$, then $\cH_m^\ell$ is either a degenerate affine Hecke algebra (when $\sL_\ell$ is of type $\tM$) or an affine Sergeev algebra (when $\sL_\ell$ is of type $\tQ$), as explained in Examples~\ref{eg:daHa} and~\ref{eg:affine-Sergeev}.  The simple modules for these algebras have been classified.  See, for example, \cite{Kle05}.  On the other hand, if $t_{1,2}$ acts as zero on $\sL_\ell \irtimes \sL_\ell$, then $\cH_m^\ell$ is a wreath product algebra of the form $(\kk[y] \ltimes A)^{\otimes m} \rtimes_\rho S_m$, for a finite-dimensional algebra $A$.  Modules for such wreath product algebras can be classified using the results of Section~\ref{subsec:kxF-modules} and Clifford theory (e.g., see \cite[\S4]{RS17} for the characteristic zero case).  In this way, Theorem~\ref{theo:simple-modules} provides a complete classification of the simple $\cA_n(F)$-modules.
\end{rem}

\begin{rem}[Nontrivial $\Z$-gradings] \label{rem:nontrivial-gradings}
  The case where $\delta > 0$ (i.e.\ the $\Z$-grading on $F$ is nontrivial) is of particular interest in Heisenberg categorification.  In particular, it is exactly this property that allows one to conclude in \cite{CL12,RS17} that the Grothendieck groups of the categories defined there are isomorphic to Heisenberg algebras.  For the original Heisenberg category of \cite{Kho14}, where the grading is trivial, Khovanov proves that the Heisenberg algebra embeds into the Grothendieck group, and it is still an open conjecture that one has equality.

  For affine wreath product algebras, the study of simple modules simplifies considerably in the presence of nontrivial $\Z$-gradings.  Since $|x_i| = \delta > 0$, the $x_i$ act as zero on any simple module.  Similarly, the $t_{i,j}$, which also have degree $\delta$, act as zero on simple modules.  Thus, the study of simple $\cA_n(F)$-modules reduces to the study of simple modules for wreath product algebras, which can be classified using Clifford theory.  However, the full representation theory (i.e.\ the study of modules that are not necessarily simple) remains much more intricate in general.
\end{rem}

%%%%%%%%%%%%%%%%%%%%%%%%%%%%%%
%
\section{Cyclotomic quotients\label{sec:cyclotomic}}
%
%%%%%%%%%%%%%%%%%%%%%%%%%%%%%%

In this section we introduce and study cyclotomic quotients of affine wreath product algebras.  These simultaneously unify and generalize cyclotomic quotients of degenerate affine Hecke algebras (see, e.g., \cite[\S7.3]{Kle05}), wreath Hecke algebras (see \cite[\S5]{WW08}) and affine Sergeev algebras (see, e.g., \cite[\S15.3]{Kle05}).  Choosing particular Frobenius algebras $F$ will recover known results as well as proofs of open conjectures (Corollaries~\ref{cor:KM3.21} and \ref{cor:KM3.22}).  In this section $\kk$ is an arbitrary commutative ring of characteristic not equal to $2$.

%----------------------------------
\subsection{Shifting homomorphisms}
%----------------------------------

For $1 \le i \le n$ and $k \in \Z$, define
\begin{equation} \label{eq:bF_i^(k)-def}
  \bF_i^{(k)}
  := \left( \bF_\psi^{(0^{i-1},k,0^{n-i})} \right)^{S_n^i} \subseteq F^{\otimes n} \subseteq \cA_n(F),
  \quad \text{where }
  S_n^i  = \{ \pi \in S_n \mid \pi i = i\},
\end{equation}
and $(0^{i-1},k,0^{n-i}) = (0,\dotsc,0,k,0,\dotsc,0)$, where the $k$ appears in the $i$-th place.  Intuitively, one should think of $\bF_i^{(k)}$ as the subspace of $F^{\otimes n}$ consisting of those elements that commute with elements of $\cA_n(F)$ just as $x_i^k$ does.

\begin{lem} \label{lem:shifting-spaces}
  For $1 \le i \le n$ and $\pi \in S_n$, we have $\pi \bF_i^{(k)} \pi^{-1} = \bF_{\pi i}^{(k)}$.
\end{lem}

\begin{proof}
  It suffices to show that $\pi \bF_i^{(k)} \pi^{-1} \subseteq \bF_{\pi i}^{(k)}$, since the reverse inclusion then follows by considering $\pi^{-1}$.  First, it is clear that
  \[
    \pi \bF_i^{(k)} \pi^{-1}
    \stackrel{\eqref{rel:SF-commutation}}{=} \prescript{\pi}{}{\left(\bF_i^{(k)}\right)}
    \subseteq \bF_\psi^{\left( 0^{\pi i-1},k,0^{n-\pi i} \right)},
  \]
  where the last inclusion follows from the fact that $\prescript{\pi}{}{\left(\bF_\psi^{(\alpha)}\right)} = \bF_\psi^{(\pi \cdot \alpha)}$ for $\alpha \in \Z^n$.  Now suppose $\bsf \in \bF_i^{(k)}$.  For $\pi_1 \in S_n^{\pi i}$, we have
  \[
    \prescript{\pi_1}{}{\left( \pi \bsf \pi^{-1} \right)}
    \stackrel{\eqref{rel:SF-commutation}}{=} \pi_1 \pi \bsf \pi^{-1} \pi_1^{-1}
    = \pi \left( \pi^{-1} \pi_1 \pi \right) \bsf \left( \pi^{-1} \pi_1 \pi \right)^{-1} \pi^{-1}
    = \pi \left( \prescript{\pi^{-1} \pi_1 \pi}{}{\bsf} \right) \pi^{-1}
    = \pi \bsf \pi^{-1},
  \]
  where the last equality follows from the fact that $\pi^{-1} \pi_1 \pi \in S_n^i$.
\end{proof}

\begin{prop} \label{prop:shift-x_i}
  Suppose $\bc \in \bF_1^{(1)}$ is even of degree $\delta$.  Then
  \[
    x_i \mapsto x_i + s_{i-1} \dotsm s_1 \bc s_1 \dotsm s_{i-1},\quad
    \bsf \mapsto \bsf,\quad
    \pi \mapsto \pi,\qquad
    1 \le i \le n,\ \bsf \in F^{\otimes n},\ \pi \in S_n,
  \]
  determines an algebra automorphism of $\cA_n(F)$.
\end{prop}

\begin{proof}
  Let $\zeta$ denote the given map.  It suffices to show that $\zeta$ is an algebra homomorphism, since it is then clearly invertible, with inverse given by the same map with $\bc$ replaced by $-\bc$.

  For $1 \le i \le n$, define
  \[
    \bc^{(i)} = s_{i-1} \dotsm s_1 \bc s_1 \dotsm s_{i-1}.
  \]
  By Lemma~\ref{lem:shifting-spaces}, we have $\bc^{(i)} \in \bF_i^{(1)}$.  It follows that, for all $1 \le i,j \le n$ and $\bsf \in F^{\otimes n}$, we have
  \[
    \bc^{(i)} \bc^{(j)} = \bc^{(j)} \bc^{(i)},\quad
    x_i \bc^{(j)} = \bc^{(j)} x_i,\quad
    \bsf \bc^{(i)} = \bc^{(i)} \psi_i(\bsf),
  \]
  and that
  \[
    s_i \bc^{(j)} = \bc^{(j)} s_i,\quad 1 \le i \le n-1,\ 1 \le j \le n,\ j \ne i,i+1.
  \]

  It is clear that $\zeta$ is a homomorphism when restricted to $F^{\otimes n}$ and $\kk S_n$.  It follows easily from the above relations that $\zeta(x_i) \zeta(x_j) = \zeta(x_j) \zeta(x_i)$ for all $1 \le i,j \le n$. Hence $\zeta$ is also a  homomorphism when restricted to $\kk[x_1,\dotsc,x_n]$.

  To prove that $\zeta$ preserves \eqref{rel:sx-commutation}, we compute
  \[
    \zeta(s_i x_i)
    = s_i x_i + s_i \bc^{(i)}
    \stackrel{\eqref{rel:sx-commutation}}{=} x_{i+1} s_i - t_{i,i+1} + s_i \bc^{(i)}
    = x_{i+1} s_i + \bc^{(i+1)} s_i - t_{i,i+1}
    = \zeta \left( x_{i+1} s_i - t_{i,i+1} \right).
  \]
  The remaining relations \eqref{rel:xF-commutation}, \eqref{rel:sx-triv-commutation}, and \eqref{rel:SF-commutation} are straightforward to verify.
  \details{
    For $1 \le i \le n$, $\bsf \in F^{\otimes n}$, we have
    \[
      \zeta(\bsf x_i)
      = \bsf x_i + \bsf \bc^{(i)}
      = x_i \psi_i(\bsf) + \bc^{(i)} \psi_i(\bsf)
      = \zeta \left( x_i \psi_i(\bsf) \right),
    \]
    and so $\zeta$ preserves \eqref{rel:xF-commutation}.

    For $1 \le i \le n-1$, $1 \le j \le n$, $j \ne i,i+1$, we have
    \[
      \zeta(s_i x_j)
      = s_i x_j + s_i \bc^{(j)}
      = x_j s_i + \bc^{(j)} s_i
      = \zeta(x_j s_i),
    \]
    and so $\zeta$ preserves \eqref{rel:sx-triv-commutation}.

    For $\pi \in S_n$ and $\bsf \in F^{\otimes n}$ we have
    \[
      \zeta(\pi \bsf)
      = \pi \bsf
      = \prescript{\pi}{}{\bsf} \pi
      = \zeta \left( \prescript{\pi}{}{\bsf} \pi \right),
    \]
    and so $\zeta$ preserves \eqref{rel:SF-commutation}.
  }
\end{proof}

%----------------------------------------------
\subsection{Cyclotomic wreath product algebras}
%----------------------------------------------

For $1 \le k \le \theta$, choose $e_k \in \N$ and degree $k \delta$ elements
\[
  \bc^{(k,1)},\dotsc,\bc^{(k,e_k)} \in \bF_1^{(k)}.
\]
Define
\[
  \bC = \left( \bc^{(1,1)}, \dotsc, \bc^{(1,e_1)},\dotsc, \bc^{(\theta,1)},\dotsc, \bc^{(\theta,e_\theta)} \right).
\]
Let $J_\bC$ be the two-sided ideal in $\cA_n(F)$ generated by the homogeneous element
\begin{equation} \label{eq:chi_C-def}
  \chi_\bC = \prod_{k=1}^\theta \prod_{j=1}^{e_k} \left( x_1^k - \bc^{(k,j)} \right).
\end{equation}
Note that $\chi_\bC$ is independent of the order of the factors in \eqref{eq:chi_C-def} by the definition of $\bF_1^{(k)}$.  In fact, this was the essential motivation for the definitions \eqref{eq:F^(k)-def}, \eqref{eq:bF^alpha-def}, and \eqref{eq:bF_i^(k)-def}.  We choose the elements $\bc^{(k,j)}$ so that they commute with elements of $\cA_n(F)$ in the same way that $x_1^k$ does.

We define the \emph{cyclotomic wreath product algebra} to be the quotient
\begin{equation} \label{eq:cyclotomic-WPA}
  \cA_n^\bC(F) := \cA_n(F)/J_\bC.
\end{equation}
By convention, we set $\cA_0^\bC(F) = \kk$.  Since $\chi_\bC$ is homogeneous, $\cA_n^\bC(F)$ inherits the structure of a graded superalgebra.  We define the \emph{level} of $\bC$ and the corresponding algebra $\cA_n^\bC(F)$ to be the polynomial degree of $\chi_\bC$, which we denote by $d_\bC$.  Thus
\begin{equation} \label{eq:d_C-def}
  d = d_\bC = \sum_{k=1}^\theta k e_k.
\end{equation}
For any element of $\cA_n(F)$, we will denote its canonical image in $\cA_n^\bC(F)$ by the same symbol.

\begin{prop}
  If $\delta > 0$ or $\kk$ is an algebraically closed field, then every finite-dimensional $\cA_n(F)$-module is the inflation of an $\cA_n^\bC(F)$-module for some $\bC$.
\end{prop}

\begin{proof}
  Let $V$ be a finite-dimensional $\cA_n(F)$-module.  If $\delta > 0$, then the action of $x_1$ on $V$ is nilpotent by degree considerations and the result is clear.  On the other hand, when $\delta=0$, we simply take $\chi_\bC$ to be the minimal polynomial of $x_1$ on $V$, which factors in the form \eqref{eq:chi_C-def} when $\kk$ is algebraically closed.
\end{proof}

%-------------------------
\subsection{Basis theorem}
%-------------------------

We now describe an explicit basis for $\cA_n^\bC(F)$.  Our approach is inspired by the methods of \cite[\S7.5, \S15.4]{Kle05}.

Let $\chi_1 := \chi_\bC$ and, for $i=2,\dotsc,n$, define
\[
  \chi_i = s_{i-1} \dotsm s_1 \chi_1 s_1 \dotsm s_{i-1}.
\]

\begin{lem}
  For $f \in F$ and $1 \le i,j \le n$, we have
  \[
    f_i \chi_j =
    \begin{cases}
      \chi_j f_i & \text{if } i \ne j, \\
      \chi_j \psi^d(f_i) & \text{if } i = j.
    \end{cases}
  \]
  In particular,
  \begin{equation} \label{eq:Fchi-commutation}
    F^{\otimes n} \chi_j = \chi_j F^{\otimes n},\qquad 1 \le j \le n.
  \end{equation}
\end{lem}

\begin{proof}
  The case $j=1$ follows immediately from the definition of $\chi_\bC$.  The result for general $j$ then follows from a straightfoward calculation using the definition of $\chi_j$.
  \details{
    If $i < j$, we have
    \begin{multline*}
      f_i \chi_j
      = f_i s_{j-1} \dotsm s_1 \chi_1 s_1 \dotsm s_{j-1}
      = s_{j-1} \dotsm s_1 f_{i+1} \chi_1 s_1 \dotsm s_{j-1}
      \\
      = s_{j-1} \dotsm s_1 \chi_1 f_{i+1} s_1 \dotsm s_{j-1}
      = s_{j-1} \dotsm s_1 \chi_1 s_1 \dotsm s_{j-1} f_i
      = \chi_j f_i.
    \end{multline*}
    If $i > j$, we have
    \begin{multline*}
      f_i \chi_j
      = f_i s_{j-1} \dotsm s_1 \chi_1 s_1 \dotsm s_{j-1}
      = s_{j-1} \dotsm s_1 f_i \chi_1 s_1 \dotsm s_{j-1}
      \\
      = s_{j-1} \dotsm s_1 \chi_1 f_i s_1 \dotsm s_{j-1}
      = s_{j-1} \dotsm s_1 \chi_1 s_1 \dotsm s_{j-1} f_i
      = \chi_j f_i.
    \end{multline*}
    Finally, we have
    \begin{multline*}
      f_j \chi_j
      = f_j s_{j-1} \dotsm s_1 \chi_1 s_1 \dotsm s_{j-1}
      = s_{j-1} \dotsm s_1 f_1 \chi_1 s_1 \dotsm s_{j-1}
      \\
      = s_{j-1} \dotsm s_1 \chi_1 \psi^d(f_1) s_1 \dotsm s_{j-1}
      = s_{j-1} \dotsm s_1 \chi_1 s_1 \dotsm s_{j-1} \psi^d(f_j)
      = \chi_j \psi^d(f_j).
    \end{multline*}
  }
\end{proof}

\begin{lem} \label{lem:x-chi-commute}
  We have $x_1 \chi_1 = \chi_1 x_1$.  For $1 \le i < j \le n$, we also have $x_j \chi_i = \chi_i x_j$.
\end{lem}

\begin{proof}
  The first statement follows from the fact that $\bc^{(k)} \in F_\psi^{\otimes n}$ for all $k$.  For $1 \le i < j \le n$, we have
  \[
    x_j \chi_i
    \stackrel{\eqref{rel:sx-triv-commutation}}{=} s_{i-1} \dotsm s_1 x_j \chi_1 s_1 \dotsm s_{i-1}
    = s_{i-1} \dotsm s_1 \chi_1 x_j s_1 \dotsm s_{i-1}
    \stackrel{\eqref{rel:sx-triv-commutation}}{=} \chi_i x_j. \qedhere
  \]
\end{proof}

\begin{lem} \label{lem:chi_i-leading}
  For $i=1,\dotsc,n$, we have
  \[
    \chi_i - x_i^d \in \sum_{e=0}^{d-1} P_{i-1} x_i^e F^{\otimes i} S_i.
  \]
\end{lem}

\begin{proof}
  The case $i=1$ is immediate.  Assuming the result for some $1 \le i < n$, we have
  \begin{multline*}
    \chi_{i+1} - x_{i+1}^d
    \stackrel{\eqref{eq:sx^k-commutation1}}{=} s_i \chi_i s_i - s_i x_i^d s_i - t_{i,i+1}^{(d)} s_i
    = s_i \left( \chi_i - x_i^d \right) s_i - t_{i,i+1}^{(d)} s_i
    \\
    \in s_i \sum_{e=0}^{d-1} P_{i-1} x_i^e F^{\otimes i} S_i s_i - t_{i,i+1}^{(d)} s_i
    \subseteq \sum_{e=0}^{d-1} P_i x_{i+1}^e F^{\otimes (i+1)} S_{i+1}. \qedhere
  \end{multline*}
\end{proof}

For $Z = \{z_1 < \dotsb < z_k\} \subseteq \{1,\dotsc,n\}$, let
\[
  \chi_Z := \chi_{z_1} \chi_{z_2} \dotsm \chi_{z_k} \in \cA_n(F).
\]
We also define
\begin{gather*}
  \Pi_n := \{ (\alpha,Z) \mid Z \subseteq \{1,\dotsc,n\},\ \alpha \in \N^n,\ \alpha_i < d \text{ whenever } i \notin Z\}, \\
  \Pi_n^+ := \{(\alpha,Z) \in \Pi_n \mid Z \ne \varnothing\}.
\end{gather*}

\begin{lem} \label{lem:dHn-right-FS-module-basis}
  We have that $\cA_n(F)$ is a free right $F^{\otimes n} \rtimes_\rho S_n$-module on the basis
  \[
    \{x^\alpha \chi_Z \mid (\alpha,Z) \in \Pi_n\}.
  \]
\end{lem}

\begin{proof}
  Consider the ordering $\prec$ on $\N^n$ given by $\alpha \prec \alpha'$ if and only if
  \[
    \alpha_n = \alpha_n', \dotsc, \alpha_{k+1} = \alpha_{k+1}',\ \alpha_k < \alpha_k'
  \]
  for some $k \in \{1,\dotsc,n\}$.  Define a function
  \[
    \gamma \colon \Pi_n \to \N^n,\quad
    \gamma(\alpha,Z) := (\gamma_1,\dotsc,\gamma_n), \quad \text{where }
    \gamma_i =
    \begin{cases}
      \alpha_i & \text{if } i \notin Z, \\
      \alpha_i + d & \text{if } i \in Z.
    \end{cases}
  \]
  Using induction on $n$ and Lemma~\ref{lem:chi_i-leading}, we see that, for $(\alpha,Z) \in \Pi_n$, we have
  \begin{equation} \label{eq:x^alpha.chi_Z-leading}
    x^\alpha \chi_Z - x^{\gamma(\alpha,Z)} \in \sum_{\beta \prec \gamma(\alpha,Z)} x^\beta F^{\otimes n} S_n.
  \end{equation}
  \details{
    Consider the base case $n=1$.  If $Z = \varnothing$, then $\gamma(\alpha,Z) = \alpha$, and the claim is clearly true.  If $Z = \{1\}$, then $\gamma(\alpha,Z) = \alpha+d$ and we have
    \[
      x^\alpha \chi_Z - x^{\gamma(\alpha,Z)}
      = x_1^\alpha \chi_1 - x_1^{\alpha + d}
      = x_1^\alpha (\chi_1 - x_1^d)
      \in \sum_{e=0}^{d-1} x_1^{e + \alpha} F
      \subseteq \sum_{\beta \prec \gamma(\alpha,Z)} x^\beta F.
    \]
    Thus, the claim is true for $n=1$.

    Now suppose $n > 1$ and the result holds for $n-1$.  Let $\alpha' = (\alpha_1,\dotsc,\alpha_{n-1})$. If $n \notin Z$, then
    \begin{align*}
      x^\alpha \chi_Z - x^{\gamma(\alpha,Z)}
      &= x^{\alpha'} x_n^{\alpha_n} \chi_Z - x^{\gamma(\alpha',Z)} x_n^{\alpha_n} \\
      &= x_n^{\alpha_n} \left( x^{\alpha'} \chi_Z - x^{\gamma(\alpha',Z)} \right) \\
      &\in \sum_{\beta \prec \gamma(\alpha',Z)} x^\beta x_n^{\alpha_n} F^{\otimes (n-1)} S_{n-1} \\
      &\subseteq \sum_{\beta \prec \gamma(\alpha,Z)} x^\beta F^{\otimes n} S_n.
    \end{align*}
    On the other hand, if $n \in Z$, then, setting $Z' = Z \setminus \{n\}$, we have
    \begin{align*}
      x^\alpha \chi_Z - x^{\gamma(\alpha,Z)}
      &= x^{\alpha'} x_n^{\alpha_n} \chi_{Z'} \chi_n - x^{\gamma(\alpha',Z')} x_n^{\alpha_n+d} \\
      &= x_n^{\alpha_n} \left( x^{\alpha'} \chi_{Z'} \chi_n - x^{\gamma(\alpha',Z')} x_n^d \right) \\
      &= x_n^{\alpha_n} \left( \left( x^{\alpha'} \chi_{Z'} - x^{\gamma(\alpha',Z')} \right) \chi_n + x^{\gamma(\alpha',Z')} \left( \chi_n - x_n^d \right) \right) \\
      &\in x_n^{\alpha_n} \sum_{\beta \prec \gamma(\alpha',Z')} x^\beta F^{\otimes (n-1)} S_{n-1} \chi_n + x_n^{\alpha_n} x^{\gamma(\alpha',Z')} \sum_{e=0}^{d-1} P_{n-1} x_n^e F^{\otimes n} S_n \\
      &= \sum_{\beta \prec \gamma(\alpha',Z')} x^\beta x_n^{\alpha_n} F^{\otimes (n-1)} S_{n-1} \chi_n + \sum_{e=0}^{d-1} P_{n-1} x_n^{\alpha_n+e} F^{\otimes n} S_n \\
      &\subseteq \sum_{\beta \prec \gamma(\alpha,Z)} x^\beta F^{\otimes n} S_n,
    \end{align*}
    where in the fourth line we used the induction hypothesis and Lemma~\ref{lem:chi_i-leading}
  }
  Now, $\gamma \colon \Pi_n \to \N^n$ is a bijection and, by Theorem~\ref{theo:AnF-basis}, $\{x^\alpha \mid \alpha \in \N^n\}$ is a basis for $\cA_n(F)$ viewed as a right $F^{\otimes n} \rtimes_\rho S_n$-module.  Thus, the lemma follows from \eqref{eq:x^alpha.chi_Z-leading}.
\end{proof}

\begin{lem} \label{lem:n-1_chi_n-simplify}
  For $n > 1$, we have $F^{\otimes (n-1)} S_{n-1} \chi_n F^{\otimes n} S_n = \chi_n F^{\otimes n} S_n$.
\end{lem}

\begin{proof}
  It follows from the definition of $\chi_n$ and the relations in $\cA_n(F)$ that multiplication by $s_j$, $1 \le j \le n-2$, leaves the space $\chi_n F^{\otimes n} S_n$ invariant.
  \details{
    For $1 \le j \le n-2$, we have
    \begin{align*}
      s_j \chi_n F^{\otimes n} S_n
      &= s_j s_{n-1} \dotsm s_1 \chi_1 s_1 \dotsm s_{n-1} F^{\otimes n} S_n \\
      &= s_{n-1} \dotsm s_1 s_{j+1} \chi_1 s_1 \dotsm s_{n-1} F^{\otimes n} S_n \\
      &= s_{n-1} \dotsm s_1 \chi_1 s_{j+1} s_1 \dotsm s_{n-1} F^{\otimes n} S_n \\
      &= s_{n-1} \dotsm s_1 \chi_1 s_1 \dotsm s_{n-1} s_j F^{\otimes n} S_n \\
      &= s_{n-1} \dotsm s_1 \chi_1 s_1 \dotsm s_{n-1} F^{\otimes n} S_n.
    \end{align*}
  }
  That multiplication by $F^{\otimes (n-1)}$ also leaves this space invariant follows from \eqref{eq:Fchi-commutation}.
\end{proof}

\begin{lem} \label{lem:J_c-sum}
  We have $J_\bC = \sum_{i=1}^n P_n \chi_i F^{\otimes n} S_n$.
\end{lem}

\begin{proof}
  We have
  \begin{align*}
    J_\bC
    &= \cA_n(F) \chi_1 \cA_n(F)
    = \cA_n(F) \chi_1 P_n F^{\otimes n} S_n
    = \cA_n(F) \chi_1 F^{\otimes n} S_n \\
    &= P_n F^{\otimes n} S_n \chi_1 F^{\otimes n} S_n
    = \sum_{i=1}^n \sum_{u \in S_{(1,n-1)}} P_n F^{\otimes n} s_{i-1} \dotsm s_1 u \chi_1 F^{\otimes n} S_n \\
    &= \sum_{i=1}^n P_n F^{\otimes n} s_{i-1} \dotsm s_1 \chi_1 F^{\otimes n} S_n
    = \sum_{i=1}^n P_n F^{\otimes n} \chi_i F^{\otimes n} S_n
    = \sum_{i=1}^n P_n \chi_i F^{\otimes n} S_n,
  \end{align*}
  where the third equality uses Lemma~\ref{lem:x-chi-commute}, and the final equality uses \eqref{eq:Fchi-commutation}.
\end{proof}

\begin{lem} \label{lem:J_c-description}
  For $d > 0$, we have $J_\bC = \sum_{(\alpha,Z) \in \Pi_n^+} x^\alpha \chi_Z F^{\otimes n} S_n$.
\end{lem}

\begin{proof}
  We proceed by induction on $n$.  When $n=1$, the right side of the equality in the statement of the lemma is
  \[
    \sum_{k \in \N} x_1^k \chi_1 F
    = P_1 \chi_1 F
    = J_\bC,
  \]
  where, in the last equality, we use \eqref{eq:Fchi-commutation} and Lemma~\ref{lem:x-chi-commute}.

  Now suppose $n>1$.  Let $J_\bC' := \cA_{n-1}(F) \chi_1 \cA_{n-1}(F)$, so that
  \begin{equation} \label{eq:J_c'}
    J_\bC' = \sum_{(\alpha',Z') \in \Pi_{n-1}^+} x^{\alpha'} \chi_{Z'} F^{\otimes (n-1)} S_{n-1}
  \end{equation}
  by the induction hypothesis.  Let $J = \sum_{(\alpha,Z) \in \Pi_n^+} x^\alpha \chi_Z F^{\otimes n} S_n$.  It is clear that $J \subseteq J_\bC$.  Therefore, by Lemma~\ref{lem:J_c-sum}, it suffices to prove that $x^\alpha \chi_i F^{\otimes n} S_n \subseteq J$ for all $\alpha \in \N^n$ and $1 \le i \le n$.

  First consider $x^\alpha \chi_n F^{\otimes n} S_n$ for $\alpha = (\alpha_1,\dotsc,\alpha_n) \in \N^n$.  Let $\beta = (\alpha_1,\dotsc,\alpha_{n-1}) \in \N^{n-1}$, so that $x^\alpha = x_n^{\alpha_n} x^\beta$.  Expanding $x^\beta$ in terms of the basis of $\cA_{n-1}(F)$ from Lemma~\ref{lem:dHn-right-FS-module-basis}, we see that
  \begin{multline*}
    x^\alpha \chi_n F^{\otimes n} S_n
    \subseteq \sum_{(\alpha',Z') \in \Pi_{n-1}} x_n^{\alpha_n} x^{\alpha'} \chi_{Z'} F^{\otimes (n-1)} S_{n-1} \chi_n F^{\otimes n} S_n
    \\
    \subseteq \sum_{(\alpha',Z') \in \Pi_{n-1}} x_n^{\alpha_n} x^{\alpha'} \chi_{Z'} \chi_n F^{\otimes n} S_n
    \subseteq J,
  \end{multline*}
  where the second inclusion follows from Lemma~\ref{lem:n-1_chi_n-simplify}.

  Now suppose $1 \le i < n$ and consider $x^\alpha \chi_i F^{\otimes n} S_n$.  Again, let $\beta = (\alpha_1,\dotsc,\alpha_{n-1}) \in \N^{n-1}$, so that $x^\alpha = x_n^{\alpha_n} x^\beta$.  By the induction hypothesis, we have
  \[
    x^\alpha \chi_i F^{\otimes n} S_n
    = x_n^{\alpha_n} x^\beta \chi_i F^{\otimes n} S_n
    \subseteq \sum_{(\alpha',Z') \in \Pi_{n-1}^+} x_n^{\alpha_n} x^{\alpha'} \chi_{Z'} F^{\otimes n} S_n.
  \]
  We now show by induction on $\alpha_n$ that $x_n^{\alpha_n} x^{\alpha'} \chi_{Z'} F^{\otimes n} S_n \subseteq J$ for all $(\alpha', Z') \in \Pi_{n-1}^+$.  This follows immediately from the definition of $\Pi_n^+$ if $\alpha_n < d$.  So suppose $\alpha_n \ge d$.  By Lemmas~\ref{lem:x-chi-commute} and~\ref{lem:chi_i-leading}, we have
  \[
    x_n^{\alpha_n} x^{\alpha'} \chi_{Z'} F^{\otimes n} S_n
    = x_n^{\alpha_n-d} x^{\alpha'} \chi_{Z'} x_n^d F^{\otimes n} S_n
    \in x_n^{\alpha_n-d} x^{\alpha'} \chi_{Z'} \chi_n F^{\otimes n} S_n
    + \sum_{e=0}^{d-1} x_n^{\alpha_n-d+e} J_\bC' F^{\otimes n} S_n.
  \]
  By the definition of $J$, we have $x_n^{\alpha_n-d} x^{\alpha'} \chi_{Z'} \chi_n F^{\otimes n} S_n \in J$.  Now, by \eqref{eq:J_c'}, for $0 \le e < d$, we have
  \[
    x_n^{\alpha_n-d+e} J_\bC' F^{\otimes n} S_n
    \subseteq \sum_{(\alpha',Z') \in \Pi_{n-1}^+} x_n^{\alpha_n-d+e} x^{\alpha'} \chi_{Z'} F^{\otimes n} S_n.
  \]
  Since $0 \le \alpha_n - d + e < \alpha_n$, each term in the above sum is contained in $J$ by induction.  This completes the proof.
\end{proof}

\begin{theo}[Basis theorem for cyclotomic quotients] \label{theo:basis-cyclotomic}
  The canonical images of the elements
  \[
    \{x^\alpha \bb \pi \mid \alpha \in \N^n \text{ with } \alpha_1,\dotsc,\alpha_n < d_\bC,\ \bb \in B^{\otimes n},\ \pi \in S_n\}
  \]
  form a basis for $\cA_n^\bC(F)$.
\end{theo}

\begin{proof}
  By Lemmas~\ref{lem:dHn-right-FS-module-basis} and~\ref{lem:J_c-description}, the elements $\{x^\alpha \chi_Z \mid (a,Z) \in \Pi_n^+\}$ form a basis for $J_\bC$ viewed as a right $F^{\otimes n} \rtimes_\rho S_n$-module.  Thus Lemma~\ref{lem:dHn-right-FS-module-basis} implies that
  \[
    \{x^\alpha \mid \alpha \in \N^n \text{ with } \alpha_1,\dotsc,\alpha_n < d\}
  \]
  is a basis for a complement to $J_\bC$ in $\cA_n(F)$, viewed as a right $F^{\otimes n} \rtimes_\rho S_n$-module.  The theorem follows.
\end{proof}

When $F = \kk$ or $F = \Cl$, Theorem~\ref{theo:basis-cyclotomic} recovers known results.  (See, e.g., \cite[Th.~7.5.6 and Th.~15.4.6]{Kle05}.)  When $F$ is the group algebra of a group (see Example~\ref{eg:wreath-Hecke}), the result is stated without proof in \cite[Prop.~5.5]{WW08}.  In other cases, Theorem~\ref{theo:basis-cyclotomic} seems to be new.  In particular, as noted in the introduction, when $F$ is a symmetric algebra, concentrated in even parity, $\cA_n(F)$ is the affinized symmetric algebra considered by Kleshchev and Muth \cite[\S3]{KM15}.  Under these additional assumptions, those authors prove that the elements given in Theorem~\ref{theo:basis-cyclotomic} are a spanning set and then conjecture that they are a basis \cite[Conj.~3.21]{KM15}.

\begin{cor} \label{cor:KM3.21}
  Conjecture~3.21 of \cite{KM15} holds.
\end{cor}

\begin{proof}
  This follows immediately from the special case of Theorem~\ref{theo:basis-cyclotomic} where $F$ is purely even and symmetric (i.e.\ $\psi = \id$), and $\bC \in (Z(F)^{\otimes n})^{S_n}$.
\end{proof}

\begin{cor}
  Every level one cyclotomic wreath product algebra is isomorphic to $F^{\otimes n} \rtimes_\rho S_n$.
\end{cor}

\begin{proof}
  If $\chi_\bC = x_1$, then $\cA_n^\bC(F) \simeq F^{\otimes n} \rtimes_\rho S_n$.  The projection $\cA_n(F) \twoheadrightarrow \cA_n^\bC(F) \simeq F^{\otimes n} \rtimes_\rho S_n$ is precisely the homomorphism described in Proposition~\ref{prop:JM-quotient}.  The general result then follows from Proposition~\ref{prop:shift-x_i}.
\end{proof}

\begin{rem}
  When $F = \kk$, Brundan proved in \cite[Th.~1]{Bru08} that the centers of the cyclotomic quotients consist of symmetric polynomials in $x_1,\dotsc,x_n$.  However, this is \emph{not} true in general (i.e.\ for arbitrary $F$).  Indeed, the $x_i$ have degree $\delta$.  Thus, if $\delta > 0$, but $Z(F) \cap F_\psi \ne \kk$, there are elements of the center of $\cA_n^\bC(F)$ that cannot possibly be expressed as symmetric polynomials in $x_1,\dotsc,x_n$ for degree reasons.
\end{rem}

%-----------------------------------------
\subsection{Frobenius algebra structure}
%-----------------------------------------

By Theorem~\ref{theo:basis-cyclotomic}, we can define an even linear map $\tr_\bC \colon \cA_n^\bC(F) \to \kk$ by defining
\begin{equation} \label{eq:full-cyclotomic-trace}
  \tr_\bC \left( x^\alpha \bsf \pi \right) = \delta_{\alpha,(d-1,\dotsc,d-1)} \tr^{\otimes n}(\bsf) \delta_{\pi,1}, \quad
  \bsf \in F^{\otimes n},\ \alpha = (\alpha_1,\dotsc,\alpha_n) \in \N^n,\ \alpha_1,\dotsc,\alpha_n < d_\bC.
\end{equation}
and extending by linearity.

\begin{theo} \label{theo:cyclotomic-Frobenius-algebra}
  The cyclotomic quotient $\cA_n^\bC(F)$ is an $\N$-graded Frobenius superalgebra with trace map $\tr_\bC$ and Nakayama automorphism given by
  \[
    x_i \mapsto x_i,\quad
    \bsf \mapsto \left( \psi^{d_\bC} \right)^{\otimes n} (\bsf),\quad
    \pi \mapsto \pi,\qquad
    1 \le i \le n,\ \bsf \in F^{\otimes n},\ \pi \in S_n.
  \]
  It follows that $\cA_n(F)$ is a symmetric algebra if the level $d_\bC$ is a multiple of the order of the Nakayama automorphism $\psi$ of $F$.  In particular, if $F$ is a symmetric algebra (i.e.\ $\psi = \id$), then so is $\cA_n^\bC(F)$.
\end{theo}

\begin{proof}
  To check that $\tr_\bC$ satisfies the required property that \eqref{eq:trace-nondegen} is an isomorphism, it is enough to verify that the basis of $\cA_n^\bC(F)$ given in Theorem~\ref{theo:basis-cyclotomic} has a left dual basis with respect to $\tr_\bC$.  Consider the strong Bruhat order on $S_n$ and the total order on $\N^n$ given by
  \begin{equation} \label{eq:Nn-partial-order}
    \alpha < \beta \iff \alpha_1 = \beta_1,\dotsc,\alpha_{i-1} = \beta_{i-1}, \alpha_i < \beta_i \text{ for some } 1 \le i \le n.
  \end{equation}
  It suffices to prove that, for each basis element $x^\alpha \bb \pi$, we can find $z \in \cA_n^\bC(F)$ such that $\tr_\bC(z x^\alpha \bb \pi)=1$ and $\tr_\bC(z x^{\alpha'} \bb' \pi') = 0$ for all
  \[
    \pi' \not \le \pi \text{ or } \left( \pi'=\pi,\ \alpha' < \alpha \right) \text{ or } \left( \pi=\pi',\ \alpha=\alpha',\ \bb' \ne \bb \right).
  \]
  Indeed, if this is true, one can find a dual basis by inverting a unitriangular matrix.

  Fix a basis element $x^\alpha \bb \pi$.  By Lemma~\ref{lem:Sn-commutation-mod-lower-terms}, we have
  \[
    \pi^{-1} x^\alpha \bb \pi \in \prescript{\pi^{-1}}{}{\left( x^\alpha \bb \right)} + \sum_{\pi' \ne 1} P_n(F) \pi'.
  \]
  Hence, without loss of generality, we may assume that $\pi = 1$.
  \details{
    Note that $\sigma < \pi^{-1}$ implies $\sigma^{-1} < \pi$.  Thus, if $\pi' \not \le \pi$, we have $\sigma \pi' \ne 1$ for all $\sigma < \pi^{-1}$.  Hence, by Lemma~\ref{lem:Sn-commutation-mod-lower-terms}, we have
    \[
      \tr_\bC \left( \pi^{-1} x^{\alpha'} \bb' \pi' \right) = 0
    \]
    when $\pi' \not \le \pi$.
  }
  Now, let
  \[
    \bsg = \left( \psi^{-\alpha_1} \otimes \psi^{-\alpha_2} \otimes \dotsb \otimes \psi^{-\alpha_n} \right) (\bb^\vee), \quad
    \beta = (d-1-\alpha_1,\dotsc,d-1-\alpha_n).
  \]
  Then we have
  \[
    \tr_\bC \left( x^\beta \bsg x^\alpha \bb \right)
    = \tr_\bC \left( x_1^{d-1} \dotsm x_n^{d-1} \bb^\vee \bb \right)
    = 1
  \]
  It also follows from Lemma~\ref{lem:chi_i-leading} that $\tr_\bC \left( x^\beta \bsg x^{\alpha'} \bb' \right) = 0$ when $\alpha' < \alpha$ or when $\alpha' = \alpha$ and $\bb' \ne \bb$.  This completes the proof that $\tr_\bC$ satisfies the defining property of a trace map.

  That the Nakayama auotomorphism is the given map on $F^{\otimes n}$ and $S_n$ is verified via a straightforward direct computation using Theorem~\ref{theo:basis-cyclotomic}.
  \details{
    Let $\alpha \in \N^n$ with $\alpha_1,\dotsc,\alpha_n < d$, $\bb = b_1 \otimes \dotsb \otimes b_n \in B^{\otimes n}$, $\pi \in S_n$.  Fix $f \in F$ and $1 \le i \le n$.  If $\pi \ne 1$ or $\alpha_k \ne d-1$ for some $1 \le k \le n$, we have
    \[
      \tr^\bC(f_i x^\alpha \bb \pi) = 0 = \tr^\bC(x^\alpha \bb \pi f_i).
    \]
    Now suppose $\alpha = (d-1,\dotsc,d-1)$.  Then
    \begin{align*}
      \tr_\bC(f_i x^\alpha \bb)
      &= \tr_\bC \left( x^\alpha \psi^{d-1}(f)_i \bb \right) \\
      &= (-1)^{(\bar b_1 + \dotsb + \bar b_{i-1})\bar f} \tr(b_1) \dotsm \tr(b_{i-1}) \tr \left( \psi^{d-1}(f) b_i \right) \tr(b_{i+1}) \dotsm \tr(b_n) \\
      &= (-1)^{(\bar b_1 + \dotsb + b_i)\bar f} \tr(b_1) \dotsm \tr(b_{i-1}) \tr \left( b_i \psi^d(f) \right) \tr(b_{i+1}) \dotsm \tr(b_n) \\
      &= (-1)^{\bar \bb \bar f} \tr_\bC \left( x^\alpha \bb \psi^d(f_i) \right).
    \end{align*}
    This proves that the Nakayama automorphism of $\cA_n^\bC(F)$ maps $f_i$ to $\psi^d(f)_i$.  It follows that it maps $\bsf \in F^{\otimes n}$ to $(\psi^d)^{\otimes n}(\bsf)$.

    Since $\tr^\bC$ involves projection onto the highest polynomial degree monomial $x_1^{d-1} \dotsm x_n^{d-1}$ in the $x_i$ with respect to the basis of Theorem~\ref{theo:basis-cyclotomic}, it follows from Lemma~\ref{lem:Sn-commutation-mod-lower-terms} that
    \begin{multline*}
      \tr_\bC \left( s_i x^\alpha \bb \pi \right)
      = \tr_\bC \left( x^\alpha (\prescript{s_i}{}{\bb}) s_i \pi \right)
      = \delta_{s_i \pi, 1} \delta_{\alpha,(d-1,\dotsc,d-1)} \tr^{\otimes n} \left( \prescript{s_i}{}{\bb} \right)
      \\
      = \delta_{\pi s_i, 1} \delta_{\alpha,(d-1,\dotsc,d-1)} \tr^{\otimes n} (\bb)
      = \tr_\bC \left( x^\alpha \bb \pi s_i \right).
    \end{multline*}
    This proves that the Nakayama automorphism of $\cA_n^\bC(F)$ maps $s_i$ to $s_i$.  Since it is an automorphism, it therefore maps $\pi$ to $\pi$ for all $\pi \in S_n$.
  }

  Now suppose that the Nakayama automorphism maps $x_i$ to $x_i$ for some $1 \le i \le n-1$.  Then, by \eqref{rel:sx-commutation}, it maps $x_{i+1}$ to
  \[
    s_i x_i s_i + \left(\psi^d\right)^{\otimes n} (t_{i,i+1})
    = s_i x_i s_i + t_{i,i+1}
    = x_{i+1},
  \]
  where the second equality follows from the fact that the definition of $t_{i,i+1}$ is independent of the chosen basis.  Therefore, to complete the proof of the theorem, it suffices to show that the Nakayama automorphism associated to $\tr_\bC$ maps $x_1$ to $x_1$.  That is, it remains to show that
  \begin{equation} \label{eq:Nakayama-fix-x_1}
    \tr_\bC \left( x^\alpha \bsf \pi x_1 \right)
    = \tr_\bC \left( x_1 x^\alpha \bsf \pi \right)
  \end{equation}
  for all $\bsf \in F^{\otimes n}$, $\pi \in S_n$, and $\alpha \in \N^n$ with $\alpha_1,\dotsc,\alpha_n \le d-1$.

  First suppose $\pi(1) = 1$.  If $\alpha_1 < d-1$, then
  \begin{multline*}
    \tr_\bC \left( x^\alpha \bsf \pi x_1 \right)
    = \tr_\bC \left( x_1 x^\alpha \psi_1(\bsf) \pi \right)
    = \delta_{\alpha_1,d-2} \delta_{\alpha_2,d-1} \dotsm \delta_{\alpha_n,d-1} \delta_{\pi,1} \tr^{\otimes n}(\psi_1(\bsf))
    \\
    = \delta_{\alpha_1,d-2} \delta_{\alpha_2,d-1} \dotsm \delta_{\alpha_n,d-1} \delta_{\pi,1} \tr^{\otimes n}(\bsf)
    = \tr_\bC \left( x_1 x^\alpha \bsf \pi \right).
  \end{multline*}
  Now assume $\alpha_1=d-1$.  It follows from the definition of $\chi_\bC$ that we can write
  \begin{equation} \label{eq:x_1^d-expansion}
    x_1^d = \sum_{i=0}^{d-1} \bsf_{(i)} x_1^i,\quad \bsf_{(i)} \in F_\psi^{\otimes n},\ 0 \le i \le d-1.
  \end{equation}
  Then, if $\alpha' = (0,\alpha_2,\alpha_3,\dotsc,\alpha_n)$, we have
  \begin{align*}
    \tr_\bC \left( x^\alpha \bsf \pi x_1 \right)
    &= \tr_\bC \left( x_1 x^\alpha \psi_1(\bsf) \pi \right) \\
    &= \sum_{i=0}^{d-1} \tr_\bC \left( x_1^i x^{\alpha'} \bsf_{(i)} \psi_1(\bsf) \pi \right) \\
    &= \delta_{\alpha_2,d-1} \dotsm \delta_{\alpha_n,d-1} \delta_{\pi,1} \tr^{\otimes n} \left( \bsf_{(d-1)} \psi_1(\bsf) \right) \\
    &= \delta_{\alpha_2,d-1} \dotsm \delta_{\alpha_n,d-1} \delta_{\pi,1} \tr^{\otimes n} \left( \psi_1( \bsf_{(d-1)}\bsf) \right) & \text{(since $\bsf_{(d-1)} \in F_\psi^{\otimes n}$)} \\
    &= \delta_{\alpha_2,d-1} \dotsm \delta_{\alpha_n,d-1} \delta_{\pi,1} \tr^{\otimes n} \left( \bsf_{(d-1)}\bsf \right) \\
    &= \sum_{i=0}^{d-1} \tr_\bC \left( x_1^i x^{\alpha'} \bsf_{(i)} \bsf \pi \right) \\
    &= \tr_\bC \left( x_1 x^\alpha \bsf \pi \right).
  \end{align*}
  This proves \eqref{eq:Nakayama-fix-x_1} in the case that $\pi (1) = 1$.

  Next we consider the case where $\pi = s_1 \pi'$ for some $\pi' \in S_n$ with $\pi' (1) = 1$.
  Using \eqref{eq:x_1^d-expansion} and the fact that $s_1 \pi' \ne 1$, we see that
  \[
    \tr_\bC \left( x_1 x^\alpha \bsf s_1 \pi' \right)
    = 0.
  \]
  On the other hand,
  \begin{equation} \label{eq:trace-s_1pi-right-x_1}
    \tr_\bC \left( x^\alpha \bsf s_1 \pi' x_1 \right)
    = \tr_\bC \left( x^\alpha \bsf s_1 x_1 \pi' \right)
    \stackrel{\eqref{rel:sx-commutation}}{=} \tr_\bC \left( x^\alpha \bsf x_2 s_1 \pi' \right) - \tr_\bC \left( x^\alpha \bsf t_{1,2} \pi' \right).
  \end{equation}
  This is clearly equal to zero if $\alpha \ne (d-1,\dotsc,d-1)$, so we now assume $\alpha = (d-1,\dotsc,d-1)$.  Now,
  \begin{align*}
    x_2^d
    \ \ &\stackrel{\mathclap{\eqref{eq:sx^k-commutation1}}}{=}\ \ s_1 x_1^d s_1 + t_{1,2}^{(d)} s_1 \\
    &= \sum_{i=0}^{d-1} s_1 \bsf_{(i)} x_1^i s_1 + t_{1,2}^{(d)} s_1 \\
    &= \sum_{i=0}^{d-1} \prescript{s_1}{}{\bsf_{(i)}} s_1 x_1^i s_1 + t_{1,2}^{(d)} s_1 \\
    &\stackrel{\mathclap{\eqref{eq:sx^k-commutation1}}}{=}\ \ \sum_{i=0}^{d-1} \prescript{s_1}{}{\bsf_{(i)}} \left( x_2^i - t_{1,2}^{(i)} s_1 \right) + t_{1,2}^{(d)} s_1.
  \end{align*}
  Thus, if $\beta = (d-1,0,d-1,\dotsc,d-1)$, we have
  \begin{align*}
    \tr_\bC \left( x^\alpha \bsf x_2 s_1 \pi' \right)
    &= \tr_\bC \left( x^\alpha x_2 \psi_2(\bsf) s_1 \pi' \right) \\
    &= \sum_{i=0}^{d-1} \tr_\bC \left( x^{\beta} \left( \prescript{s_1}{}{\bsf_{(i)}} \right) \left( x_2^i - t_{1,2}^{(i)} s_1 \right) \psi_2(\bsf) s_1 \pi' \right) + \tr_\bC \left( x^{\beta} t_{1,2}^{(d)} s_1 \psi_2(\bsf) s_1 \pi' \right) \\
    &= \tr_\bC \left( x^\alpha \prescript{s_1}{}{\bsf_{(d-1)}} \psi_2(\bsf) s_1 \pi' \right) + \tr_\bC \left( x^{\beta} t_{1,2}^{(d)} s_1 \psi_2(\bsf) s_1 \pi' \right) \\
    &= \tr_\bC \left( x^\alpha t_{1,2} \prescript{s_1}{}{\psi_2(\bsf)} \pi' \right) \\
    &\stackrel{\mathclap{\eqref{eq:Ft-commutation}}}{=}\ \ \tr_\bC \left( x^\alpha \bsf t_{1,2} \pi' \right),
  \end{align*}
  where the third equality follows from consideration of the powers of $x_2$ and the fact that the highest power of $x_2$ appearing in $t_{1,2}^{(k)}$ is $x_2^{k-1}$, and the fourth equality follows from the fact that $s_1 \pi' \ne 1$ and $t_{1,2}^{(d)} = x_2^{d-1} t_{1,2}$ up to lower order terms in $x_2$. Combined with \eqref{eq:trace-s_1pi-right-x_1}, this proves that
  \begin{equation} \label{eq:trace-s_1pi-right-x_1-zero}
    \tr_\bC \left( x^\alpha \bsf s_1 \pi' x_1 \right) = 0 \quad \text{when } \pi'(1) = 1.
  \end{equation}
  This completes the proof of \eqref{eq:Nakayama-fix-x_1} in the case that $\pi = s_1 \pi'$, where $\pi' (1) =1$.

  Now consider the general situation where $\pi(1) \ne 1$.  Then we can write $\pi = \pi_1 s_1 \pi_2$, where $\pi_1, \pi_2 \in S_n$ satisfy $\pi_1(1)=\pi_2(1)=1$.
  \details{
    Given $\pi \in S_n$ with $\pi(1) \ne 1$, let $\pi_1 = s_{\pi(1),2}$ (so $\pi_1(1)=1$).  Then $\pi_2 := s_1 \pi_1 \pi$ fixes $1$, and we have $\pi = \pi_1 s_1 \pi_2$.
  }
  Using \eqref{eq:x_1^d-expansion} and the fact that $\pi_1 s_1 \pi_2 \ne 1$, we see that
  \[
    \tr_\bC \left( x_1 x^\alpha \bsf \pi_1 s_1 \pi_2 \right)
    = 0.
  \]
  On the other hand,
  \begin{align*}
    \tr_\bC \left( x^\alpha \bsf \pi_1 s_1 \pi_2 x_1 \right)
    &= \tr_\bC \left( x^\alpha \pi_1 \prescript{\pi_1}{}{\bsf} s_1 \pi_2 x_1 \right) \\
    &= \tr_\bC \left( \pi_1 x^{\pi^{-1}(\alpha)} \prescript{\pi_1}{}{\bsf} s_1 \pi_2 x_1 \right) + \sum_{\sigma < \pi_1} \tr_\bC \left( \sigma p_\sigma \prescript{\pi_1}{}{\bsf} s_1 \pi_2 x_1 \right) & \text{(by Lemma~\ref{lem:Sn-commutation-mod-lower-terms})} \\
    &= \tr_\bC \left( x^{\pi^{-1}(\alpha)} \prescript{\pi_1}{}{\bsf} s_1 \pi_2 x_1 \pi_1 \right) + \sum_{\sigma < \pi_1} \tr_\bC \left( p_\sigma \prescript{\pi_1}{}{\bsf} s_1 \pi_2 x_1 \sigma \right) \\
    &= \tr_\bC \left( x^{\pi^{-1}(\alpha)} \prescript{\pi_1}{}{\bsf} s_1 \pi_2 \pi_1 x_1 \right) + \sum_{\sigma < \pi_1} \tr_\bC \left( p_\sigma \prescript{\pi_1}{}{\bsf} s_1 \pi_2 \sigma x_1 \right) \\
    &\stackrel{\mathclap{\eqref{eq:trace-s_1pi-right-x_1-zero}}}{=}\ \ 0,
  \end{align*}
  where the third equality follows from the fact that the Nakayama automorphism associated to $\tr_\bC$ fixes elements of $S_n$, the fourth equality follows from the fact that $\sigma < \pi_1$ implies $\sigma(1)=1$, and the fifth equality follows from taking $\pi'$ in \eqref{eq:trace-s_1pi-right-x_1-zero} equal to $\pi_2 \pi_1$ and $\pi_2 \sigma$.  This completes the proof of \eqref{eq:Nakayama-fix-x_1} in the remaining case that $\pi(1) \ne 1$.
\end{proof}

When $F = \kk$, Theorem~\ref{theo:cyclotomic-Frobenius-algebra} recovers the result that degenerate cyclotomic Hecke algebras are symmetric algebras, a fact that follows from \cite[Cor.~6.18]{HM10} together with \cite[Th.~1.1]{BK09}.  However, the argument given above is more direct.  In the case where $F$ is the group algebra of a finite cyclic group, we recover \cite[Th.~2.4]{Cui16}.   In other cases, the result seems to be new.  Even in the $F = \Cl$ case, when $\cA_n(\Cl)$ is the affine Sergeev algebra, while it was known that $\cA_n(\Cl)$ is a Frobenius algebra (see, e.g., \cite[Cor.~15.6.4]{Kle05}), the precise form of the Nakayama automorphism does not seem to have appeared in the literature.  In addition, Conjecture~3.22 of \cite{KM15} can be reformulated as the assertion that $\cA_n(F)$ is a symmetric algebra when $F$ is a symmetric algebra concentrated in even parity.

\begin{cor} \label{cor:KM3.22}
  Conjecture 3.22 of \cite{KM15} holds.
\end{cor}

\begin{proof}
  This follows immediately from the special case of Theorem~\ref{theo:cyclotomic-Frobenius-algebra} where $F$ is purely even and symmetric (i.e.\ $\psi = \id$), and $\bC \in (Z(F)^{\otimes n})^{S_n}$.
\end{proof}

%-------------------------------------
\subsection{Cyclotomic Mackey Theorem}
%-------------------------------------

For the remainder of this section, we assume
\begin{equation} \label{eq:c-in-first-factor}
  \bc^{(k,j)} \in F_\psi^{(k)} \quad \text{for all } 1 \le k \le \theta,\ 1 \le j \le e_k,
\end{equation}
where we identify $F$ with the first factor of $F \otimes \kk^{\otimes (n-1)} \subseteq F^{\otimes n}$ for each $n \ge 1$.  Hence $\cA_n^\bC(F)$ is defined for all $n \in \N$.

Theorem~\ref{theo:basis-cyclotomic} implies that the subalgebra of $\cA_{n+1}^\bC(F)$ generated by $x_1,\dotsc,x_n$, $F^{\otimes n} \otimes 1$, and $S_n$ is isomorphic to $\cA_n^\bC(F)$.  Thus we have induction and restriction functors $\prescript{\bC}{}{\Ind_n^{n+1}}$ and $\prescript{\bC}{}{\Res^{n+1}_n}$.\label{ind-res-cyclotomic-def}

\begin{prop} \phantomsection \label{prop:cyclotomic-basis}
  Recall that $d = d_\bC$ is the level of $\bC$ (see \eqref{eq:d_C-def}).
  \begin{enumerate}
    \item \label{prop-item:cyclotomic-right-module-basis} $\cA_{n+1}^\bC(F)$ is a free right $\cA_n^\bC(F)$-module with basis
      \[
        \{x_j^a b_j s_j \dotsm s_n \mid 0 \le a < d,\ b \in B,\ 1 \le j \le n+1\}.
      \]

    \item \label{prop-item:cyclotomic-bimodule-decomp} We have a decomposition of $(\cA_n^\bC(F),\cA_n^\bC(F))$-bimodules
      \[
        \cA_{n+1}^\bC(F) = \cA_n^\bC(F) s_n \cA_n^\bC(F) \oplus \bigoplus_{0 \le a < d,\ b \in B} x_{n+1}^a b_{n+1} \cA_n^\bC(F).
      \]

    \item \label{prop-item:cyclotomic-bimodule-isoms} Suppose $0 \le a < d$ and $f \in F$.  Then we have isomorphisms of $(\cA_n^\bC(F), \cA_n^\bC(F))$-bimodules
      \begin{gather*}
        \cA_n^\bC(F) s_n \cA_n^\bC(F) \simeq \cA_n^\bC(F) \otimes_{\cA_{n-1}^\bC(F)} \cA_n^\bC(F),\\
        x_{n+1}^a f_{n+1} \cA_n^\bC(F) \simeq \Pi^{\bar f} \cA_n^\bC(F).
      \end{gather*}
  \end{enumerate}
\end{prop}

\begin{proof}
  The proof is almost identical to those of \cite[Lem.~7.6.1]{Kle05} and \cite[Lem.~15.5.1]{Kle05} and so will be omitted.
  \details{
    \begin{asparaenum}
      \item It follows from Theorem~\ref{theo:basis-cyclotomic} that the $\kk$-dimension of $\cA_n^\bC(F)$ is $n!(d\dim F)^n$.  The number of elements in the proposed basis is $(n+1) d (\dim F)$.  Thus, the dimension of the right $\cA_n^\bC(F)$-submodule of $\cA_{n+1}^\bC(F)$ generated by this set is $\le \dim \cA_{n+1}^\bC(F)$, with equality if and only the elements are independent.  Since it follows from \eqref{eq:sx^k-commutation1} that the given elements generate $\cA_{n+1}^\bC(F)$ as a right $\cA_n^\bC(F)$-module, we are done.

      \item By part \ref{prop-item:cyclotomic-right-module-basis} and \eqref{eq:sx^k-commutation1},
        \[
          \{x_j^a b_j s_j \dotsm s_n \mid 0 \le a < d,\ b \in B,\ 1 \le j \le n\}
        \]
        is a basis of $\cA_n^\bC(F) s_n \cA_n^\bC(F)$ as a free right $\cA_n^\bC(F)$-module.  The result follows.

      \item The second isomorphism follows immediately from part \ref{prop-item:cyclotomic-right-module-basis}.  (We can choose the basis $B$ of $F$ to contain the element $f$.)  To prove the first isomorphism, note that the map
        \[
          \cA_n^\bC(F) \times \cA_n^\bC(F) \to \cA_n^\bC(F) s_n \cA_n^\bC(F),\quad (u,v) \mapsto u s_n v,
        \]
        is $\cA_{n-1}^\bC(F)$-balanced, and hence induces an even homomorphism
        \[
          \Phi \colon \cA_n^\bC(F) \otimes_{\cA_{n-1}^\bC(F)} \cA_n^\bC(F) \to \cA_n^\bC(F) s_n \cA_n^\bC(F)
        \]
        of $(\cA_n^\bC(F), \cA_n^\bC(F))$-bimodules.  By part \ref{prop-item:cyclotomic-right-module-basis}, $\cA_n^\bC(F) \otimes_{\cA_{n-1}^\bC(F)} \cA_n^\bC(F)$ is a free right $\cA_n^\bC(F)$-module with basis
        \[
          \{x_j^a b_j s_j \dotsm s_{n-1} \otimes 1 \mid 0 \le a < d ,\ b \in B,\ 1 \le j \le n\}.
        \]
        Since $\Phi$ maps these elements to the basis for $\cA_n^\bC(F) s_n \cA_n^\bC(F)$ as a right $\cA_n^\bC(F)$-module given in part \ref{prop-item:cyclotomic-bimodule-decomp}, it follows that $\Phi$ is an isomorphism.
    \end{asparaenum}\
  }
\end{proof}

\begin{theo}[Cyclotomic Mackey Theorem] \label{theo:cyclotomic-Mackey}
  Recall that $d = d_\bC$ is the level of $\bC$ (see \eqref{eq:d_C-def}).  For $n \in \N_+$, we have a natural isomorphism of functors
  \[
    \prescript{\bC}{}{\Res^{n+1}_n} \prescript{\bC}{}{\Ind_n^{n+1}}
    \cong \id^{\oplus d \dim F_\even} \oplus \Pi^{\oplus d \dim F_\odd} \oplus \prescript{\bC}{}{\Ind_{n-1}^n} \prescript{\bC}{}{\Res^n_{n-1}},
  \]
  where $F_\even$ and $F_\odd$ are the even and odd parts of $F$, respectively.
\end{theo}

\begin{proof}
  This follows from Proposition~\ref{prop:cyclotomic-basis}.
\end{proof}

When $F = \kk$ or $F = \Cl$, Theorem~\ref{theo:cyclotomic-Mackey} recovers the Cyclotomic Mackey Theorem for the degenerate cyclotomic Hecke algebras and cyclotomic Sergeev algebras, respectively (see \cite[Th.~7.6.2 and Th.~15.5.2]{Kle05}).

%-----------------------------------------
\subsection{Frobenius extension structure}
%-----------------------------------------

We continue to make the assumption \eqref{eq:c-in-first-factor}.  Let
\[
  F_{n+1} = 1^{\otimes n} \otimes F = \Span_\kk \{f_{n+1} \mid f \in F\} \subseteq F^{\otimes (n+1)}.
\]
By Proposition~\ref{prop:cyclotomic-basis}\ref{prop-item:cyclotomic-bimodule-decomp}, we have a decomposition of $(\cA_n^\bC(F),\cA_n^\bC(F))$-bimodules
\begin{equation} \label{eq:cyclotomic-trace-bimodule-decomp}
  \cA_{n+1}^\bC(F)
  = x_{n+1}^{d-1} F_{n+1} \cA_n^\bC(F) \oplus \bigoplus_{a=0}^{d-2} x_{n+1}^a F_{n+1} \cA_n^\bC(F) \oplus \cA_n^\bC(F) s_n \cA_n^\bC(F).
\end{equation}
Using Proposition~\ref{prop:cyclotomic-basis}\ref{prop-item:cyclotomic-bimodule-isoms} and the trace map of $F$, we have a homomorphism of $(\cA_n^\bC(F),\cA_n^\bC(F))$-bimodules
\begin{equation} \label{eq:cyclotomic-trace-isom}
  x_{n+1}^{d-1} F_{n+1} \cA_n^\bC(F) \to \cA_n^\bC(F),\quad
  x_{n+1}^{d-1} f_{n+1} y \mapsto \tr(f) y,\qquad f \in F,\ y \in \cA_n^\bC(F).
\end{equation}
Let
\begin{equation} \label{eq:cyclotomic-trace}
  \tr^\bC_{n+1} \colon \cA_{n+1}^\bC(F) \to \cA_n^\bC(F)
\end{equation}
be homomorphism of $(\cA_n^\bC(F),\cA_n^\bC(F))$-bimodules given by the projection onto the first summand in \eqref{eq:cyclotomic-trace-bimodule-decomp} followed by the homomorphism \eqref{eq:cyclotomic-trace-isom}.  Note that the degree of $\tr_{n+1}^\bC$ is $-d \delta$.  Since, in the current setup, the notation $\tr_\bC$ is ambiguous, let $\tr_\bC^n \colon \cA_n(F) \to \kk$ denote the trace map \eqref{eq:full-cyclotomic-trace}.  Then
\[
  \tr_\bC^{n+1}
  = \tr^\bC_1 \circ \tr^\bC_2 \circ \dotsb \circ \tr^\bC_{n+1}
  = \tr_\bC^{n} \circ \tr^\bC_{n+1}.
\]

\begin{theo} \label{theo:cyclotomic-Frobenius-extension}
  The cyclotomic quotient $\cA_{n+1}^\bC(F)$ is a Frobenius extension of $\cA_n^\bC(F)$ with trace map $\tr^\bC_{n+1}$.
\end{theo}

\begin{proof}
  Let $\cB$ be a basis of $\cA_n(F)$.  By Theorem~\ref{theo:cyclotomic-Frobenius-algebra}, \cite[Cor.~7.4]{PS16}, and \cite[Cor.~3.6]{PS16}, $\cA_{n+1}(F)$ is a Frobenius extension of $\cA_n(F)$ with trace map
  \[
    z \mapsto \sum_{b \in \cB} \tr_\bC^{n+1}(b^\vee z) b.
  \]
  (Note that $b^\vee$ denotes the \emph{right} dual of $b$ in \cite{PS16}, whereas it denotes the \emph{left} dual in the current paper.)
  \details{
    Let $\psi_k$ denote the Nakayama automorphism of $\cA_k(F)$ described in Theorem~\ref{theo:cyclotomic-Frobenius-algebra}.  Then $\psi_n$ is the restriction of $\psi_{n+1}$ to $\cA_n(F)$.  It follows that
    \[
      \psi_{n+1} \colon \prescript{\psi_n}{\cA_n(F)}{\cA_{n+1}(F)}^{\psi_{n+1}}_{\cA_n(F)} \to {_{\cA_n(F)} \cA_{n+1}(F)}_{\cA_n(F)}
    \]
    (in the notation of \cite{PS16}) is a homomorphism of $(\cA_n(F),\cA_n(F))$-bimodules.  Composing with the trace map of \cite[Prop.~7.3]{PS16} gives a (nondegenerate) trace map
    \[
      \tr \colon \prescript{}{\cA_n(F)}{\cA_{n+1}(F)}_{\cA_n(F)} \to \cA_n(F),\quad
      z \mapsto \sum_{b \in \cB} \tr_\bC^{n+1}(\psi_n(b^\vee) \psi_{n+1}(z)) b
      = \sum_{b \in \cB} \tr_\bC^{n+1}(b^\vee z) b,
    \]
    where we have used the fact that $\tr_\bC^{n+1}$ is invariant under the action of $\psi_{n+1}$.
  }
  Now
  \[
    \sum_{b \in \cB} \tr_\bC^{n+1}(b^\vee z) b
    = \sum_{b \in \cB} \tr_\bC^n \left( \tr_{n+1}^\bC (b^\vee z) \right) b
    = \sum_{b \in \cB} \tr_\bC^n \left( b^\vee \tr_{n+1}^\bC (z) \right) b
    \stackrel{\eqref{eq:f-in-basis}}{=} \tr_{n+1}^\bC(z),
  \]
  completing the proof.
\end{proof}

\begin{cor} \label{cor:biadjoint}
  The exact functor $\prescript{\bC}{}{\Ind_n^{n+1}}$ is left adjoint to $\prescript{\bC}{}{\Res^{n+1}_n}$ and right adjoint to $\prescript{\bC}{}{\Res^{n+1}_n}$ up to degree shift.
\end{cor}

\begin{proof}
  Induction is always left adjoint to restriction.  That induction is right adjoint to restriction is a standard property of Frobenius extensions.  In the graded super setting considered in the current paper see, for example, \cite[Th.~6.2]{PS16}.  (Note that the twistings $\alpha$ and $\beta$ of \cite{PS16} are trivial in the setting of the current paper.)
\end{proof}

%%%%%%%%%%%%%%%%%%%%%%%%%%%
%
\section{Future directions\label{sec:future}}
%
%%%%%%%%%%%%%%%%%%%%%%%%%%%

As mentioned in the introduction, we expect that many of the results for and applications of degenerate affine Hecke algebras and their analogs have natural extensions to the setting of affine wreath product algebras.  We conclude this paper with a brief discussion of some such potential directions of future research.

%-------------------------------------------------------
\subsection{Double affine versions and $q$-deformations\label{subsec:q-deform}}
%-------------------------------------------------------

Degenerate affine Hecke algebras (Example~\ref{eg:daHa}), affine Sergeev algebras (Example~\ref{eg:affine-Sergeev}), and wreath Hecke algebras corresponding to cyclic groups (Example~\ref{eg:wreath-Hecke}) can be viewed as degenerations of affine Hecke algebras, affine Hecke--Clifford algebras, and affine Yokonuma--Hecke algebras, respectively.  It would be interesting to try to construct a $q$-deformation of affine wreath product algebras.  Similarly, it is natural to wonder if there exist natural double affine versions of wreath product algebras generalizing double affine Hecke algebras (also known as Cherednik algebras) and their various degenerations.  This should be related to the algebras introduced in \cite[\S4.2]{CG03}.

%---------------------------------------
\subsection{Heisenberg categorification}
%---------------------------------------

In \cite{Kho14}, Khovanov gave a conjectural categorification of the Heisenberg algebra based on a graphical category motivated by the representation theory of the symmetric group.  This procedure was generalized in \cite{RS17} (see also \cite[\S7]{RS15}), where the group algebra of the symmetric group was replaced by wreath product algebras.  Provided that the Frobenius algebra $F$ is not concentrated in degree zero, it was proved that the corresponding graphical categories do indeed categorify the associated lattice Heisenberg algebras.  On the other hand, Khovanov's construction was generalized in a different direction in \cite{MS17}, where group algebras of symmetric groups were replaced by degenerate cyclotomic Hecke algebras, yielding a conjectural categorification of higher level Heisenberg algebras.  We expect that the cyclotomic wreath product algebras introduced in Section~\ref{sec:cyclotomic} provide a framework for unifying these two generalizations.  In particular, one should be able to define a graphical category based on these cyclotomic quotients that specializes to the categories of \cite{MS17} when $F=\kk$ and to the categories of \cite{RS17} when the quotient is of level one.\footnote{Since the writing of the current paper, this category has been defined in \cite{Sav18}.}  More ambitiously, $q$-deformations as in \S\ref{subsec:q-deform} might allow one to also incorporate the $q$-deformed Heisenberg category of \cite{LS13}.  Considering trace decategorifications of such categories should lead to generalizations of the results of \cite{CLLS15,CLLSS16,LRS16}.

%---------------------------
\subsection{Branching rules}
%---------------------------

It would be interesting to investigate branching rules for affine wreath product algebras.  For the special cases of degenerate affine Hecke algebras and affine Sergeev algebras, we refer the reader to the exposition in \cite{Kle05}.  When $F$ is the group algebra of a finite group (Example~\ref{eg:wreath-Hecke}), branching rules were obtained in \cite{WW08}.  When the characteristic of the field $\kk$ does not divide the order of the group, the branching rules were identified with crystal graphs of integrable modules for quantum affine algebras (see \cite[\S5.5]{WW08}).  This assumption on the characteristic of $\kk$ implies that $F$ is semisimple, and we expect that the arguments of \cite{WW08} should generalize to the setting where $F$ is an arbitrary semisimple Frobenius algebra.  However, the situation where $F$ is not semisimple would likely be more involved.

%------------------
\appendix
\section{Notation\label{app:notation}}
%------------------

We let $\N$ denote the set of nonnegative integers, and let $\N_+$ denote the set of positive integers.  We assume that $\kk$ is a ring whose characteristic is not equal to two.  Any additional assumptions on $\kk$ are stated at the beginning of each section.  We also assume that the action of the Nakayama automorphism on $F$ is diagonalizable and that the characteristic of $\kk$ does not divide the order of the Nakayama automorphism.  We let $\kk^\times$ denote the group of invertible (under multiplication) elements of $\kk$. For the convenience of the reader, we list here some of the most important notation used throughout the paper, in order of appearance.

\medskip

\begin{center}
  \begin{longtable}{lll}
    \toprule
    Notation & Meaning & Definition \\
    \midrule \endfirsthead
    \toprule
    Notation & Meaning & Definition \\
    \midrule \endhead
    \bottomrule \endfoot
    \bottomrule \endlastfoot
    $\bar a$ & parity of $a$ & p.~\pageref{bar-def} \\
    $|a|$ & $\Z$-degree of $a$ & p.~\pageref{degree-notation} \\
    $A^\op$ & opposite algebra of $A$ & \eqref{eq:opposite-alg-def} \\
    $\HOM_A(M,N)$ & graded hom space & \eqref{eq:HOM-def} \\
    $\Hom_A(M,N)$ & homogeneous hom space & \eqref{eq:hom-def} \\
    $\simeq$ & even isomorphism & p.~\pageref{cong-simeq-def} \\
    $\cong$ & isomorphism (not nec.\ preserving degree or parity) & p.~\pageref{cong-simeq-def} \\
    $\Cl$ & 2-dimensional Clifford algebra & p.~\pageref{Cl-def} \\
    $\boxtimes$ & outer tensor product of modules & p.~\pageref{boxtimes-def} \\
    $\irtimes$ & simple tensor product & p.~\pageref{irtimes-def} \\
    $\cS(A)$ & set of even isom.\ classes of $A$-modules up to degree shift & p.~\pageref{cS_simeq-def} \\
    $\prescript{\alpha}{}{V}$ & twisted module & \eqref{eq:twisted-module-def} \\
    $F$ & graded Frobenius superalgebra & p.~\pageref{F-def} \\
    $\tr$ & trace map of $F$ & p.~\pageref{tr-def} \\
    $\psi$ & Nakayama automorphism of $F$ & p.~\pageref{Nakayama-def} \\
    $\delta$ & maximum degree of $F$ & p.~\pageref{delta-def} \\
    $\theta$ & order of $\psi$ & p.~\pageref{theta-def} \\
    $B$ & basis of $F$ & p.~\pageref{B-def} \\
    $b^\vee$ & left dual basis element & p.~\pageref{dual-basis-def} \\
    $\rtimes$, $\ltimes$ & smash product, wreath product algebra & \S\ref{subsec:smash} \\
    $\prescript{\pi}{}{a}$ & action of $\pi \in S_n$ on $a$ by superpermutation & p.~\pageref{perm-notation} \\
    $\rtimes_\rho$ & smash product with action given by superpermutation & p.~\pageref{rtimes_rho-def} \\
    $f_i$, $f \in F$ & $1^{\otimes (i-1)} \otimes f \otimes 1^{\otimes (n-i)}$ & \eqref{eq:f_i-def} \\
    $\psi_i$ & $\id^{\otimes (i-1)} \otimes \psi \otimes \id^{\otimes (n-i)}$ & \eqref{eq:psi_i-def} \\
    $\cA_n(F)$ & affine wreath product algebra & Def.~\ref{def:affine-wreath} \\
    $t_{i,j}$ & $\sum_{b \in B} b_i b^\vee_j$ & \eqref{def:tij} \\
    $t_{i,j}^{(k)}$ & $\sum_{b \in B} b_i \frac{x_i^k - x_j^k}{x_i-x_j} b_j^\vee$ & \eqref{eq:t^k-def} \\
    $P_n$ & $\kk[x_1,\dotsc,x_n]$ & p.~\pageref{P_n-def} \\
    $P_n(F)$ & $(\kk[x] \ltimes F)^{\otimes n}$ & \eqref{eq:P_n(F)-def} \\
    $\Delta_i$ & deformed divided difference operator & \S\ref{subsec:divided-diff} \\
    $x^\alpha$ & $x_1^{\alpha_1} x_2^{\alpha_2} \dotsm x_n^{\alpha_n}$ & \eqref{eq:x^alpha-def} \\
    $F_\psi$ & $\{f \in F \mid \psi(f) = f\}$ & \eqref{eq:F_psi-def} \\
    $F^{(k)}$, $F_\psi^{(k)}$ & see definition & \eqref{eq:F^(k)-def} \\
    $\bF^{(\alpha)}$ & $F^{(\alpha_1)} \otimes \dotsb \otimes F^{(\alpha_n)}$ & \eqref{eq:bF^alpha-def} \\
    $\bF^{(\alpha)}_\psi$ & $F^{(\alpha_1)}_\psi \otimes \dotsb \otimes F^{(\alpha_n)}_\psi$ & \eqref{eq:bF^alpha-def} \\
    $S_\mu$ & Young subgroup of $S_n$ & \eqref{eq:S_mu-def} \\
    $\cA_\mu(F)$ & parabolic subalgebra of $\cA_n(F)$ & p.~\pageref{parabolic-subalg-def} \\
    $\Ind_\mu^n$, $\Res_\mu^n$ & induction and restriction functors for parabolic subalgebras & \eqref{eq:Ind-Res-def} \\
    $\prescript{k}{}{L}$ & twisted module $\prescript{\psi^k}{}{L}$ & p.~\pageref{^kL-def}, \eqref{eq:twisted-module-def} \\
    $r_L$ & smallest positive integer such that $\prescript{r_L}{}{L} \simeq L$ & p.~\pageref{r_L-def} \\
    $m_L$ & smallest positive integer such that $\psi^{m_L}(f)(v) \ \forall\ f \in F$, $v \in L$ & p.~\pageref{m_L-def} \\
    $\tau_L$ & even $F$-module isomorphism $L \xrightarrow{\simeq} \prescript{r_L}{}{L}$ & \eqref{eq:tau_L-def} \\
    $\flip$ & $\flip(v_1 \otimes v_2) = (-1)^{\bar v_1 \bar v_2} v_2 \otimes v_1$ & \eqref{eq:flip-def} \\
    $L(a)$ & see definition & \eqref{eq:L(a)-def} \\
    $\rho_{k,\ell}$ & superpermutation of $k$-th and $\ell$-th factors & \eqref{eq:rho_kl-def} \\
    $N$ & number of $\psi$-orbits in $\cS(F)$ & p.~\pageref{N-def} \\
    $\sL_1,\dotsc,\sL_N$ & representatives of $\psi$-orbits in $\cS(F)$ & \eqref{eq:sL-def} \\
    $r_k$ & $r_{\sL_k}$, smallest positive integer such that $\prescript{r_k}{}{\sL_k} \simeq \sL_k$ & p.~\pageref{r_k-def} \\
    $\cA_n(F)\smod$ & category of f.d.\ $\cA_n(F)$-modules semisimple as $F^{\otimes n}$-modules & p.~\pageref{smod-def} \\
    $\cC_n$ & set of compositions of $n$ of length at most $N$ & \eqref{eq:cC_n-def} \\
    $\ell_k$ & unique integer such that $\mu_1 + \dotsb + \mu_{\ell_k-1} < k \le \mu_1 + \dotsc + \mu_{\ell_k}$ & \eqref{eq:l_k-def} \\
    $\tau_k$ & $\tau_{\sL_k}$, even $F$-module isomorphism $\sL_k \xrightarrow{\simeq} \prescript{r_k}{}{\sL_k}$ & p.~\pageref{tau_k-def} \\
    $\sa_k$ & see definition & p.~\pageref{sa_k-def} \\
    $\sL(\mu)$ & $\sL_1(\sa_1)^{\boxtimes \mu_1} \boxtimes \dotsb \boxtimes \sL_N(\sa_N)^{\boxtimes \mu_N}$ & \eqref{eq:Lmu-def} \\
    $\sE(\mu)$ & $\END_{F^{\otimes n}}^\op \sL(\mu)$ & \eqref{eq:Emu-def} \\
    $\cH_n^\ell$ & see definition & \eqref{eq:H_n^ell-def} \\
    $\cR_\mu$, $\cR_n$ & $\cR_\mu = \cH_{\mu_1}^1 \otimes \dotsb \otimes \cH_{\mu_N}^N$, $\cR_n = \bigoplus_{\mu \in \cC_n} \cR_\mu$ & \eqref{eq:cR_n-def} \\
    $y_1,\dotsc,y_n$ & polynomial generators of $\cR_n$ & p.~\pageref{y_i-def} \\
    $I_\mu V$ & see definition & p.~\pageref{ImuV-def} \\
    $V_\mu$ & $\sum_{\pi \in S_n} \pi(I_\mu V)$ & \eqref{eq:Vmu-def} \\
    $\bF_i^{(k)}$ & see definition & \eqref{eq:bF_i^(k)-def} \\
    $J_\bC$, $\chi_\bC$ & see definition & \eqref{eq:chi_C-def} \\
    $\cA_n^\bC(F)$ & cyclotomic wreath product algebra & \eqref{eq:cyclotomic-WPA} \\
    $d=d_\bC$ & level of $\bC$ & \eqref{eq:d_C-def} \\
    $\tr_\bC$ & trace map for $\cA_n^\bC(F)$ & \eqref{eq:full-cyclotomic-trace} \\
    $\prescript{\bC}{}{\Ind_n^{n+1}}$, $\prescript{\bC}{}{\Res^{n+1}_n}$ & induction and restriction functors for cyclotomic quotients & p.~\pageref{ind-res-cyclotomic-def}
  \end{longtable}
\end{center}

%%%%%%%%%%%%%%%%%%%%%%%%%%%%%%%%%%%%%%%%%%%%%%%%%%%%%%%%%%%%%%%%%%%
% References
%%%%%%%%%%%%%%%%%%%%%%%%%%%%%%%%%%%%%%%%%%%%%%%%%%%%%%%%%%%%%%%%%%%

\bibliographystyle{alphaurl}
\bibliography{Savage-AffineWreath}

\end{document}